\documentclass{article}
\usepackage[T1]{fontenc}  
\usepackage[utf8]{inputenc}
\usepackage[a4paper]{geometry}
\usepackage{amsmath,amsfonts,amssymb,amsthm}
\usepackage{stmaryrd} 
\usepackage[pdftex]{graphicx}
\usepackage{tikz}
\usetikzlibrary{matrix}
\usepackage{authblk}
\usepackage{algorithm}
\usepackage{algpseudocode}
\usepackage{caption}
\usepackage{subcaption}
\usepackage{todonotes}
\usepackage{array,multirow,makecell}
\setcellgapes{1pt}
\makegapedcells
\newcolumntype{R}[1]{>{\raggedleft\arraybackslash }b{#1}}
\newcolumntype{L}[1]{>{\raggedright\arraybackslash }b{#1}}
\newcolumntype{C}[1]{>{\centering\arraybackslash }b{#1}}
\theoremstyle{plain}

\newtheorem{theorem}{Theorem}[section]
\newtheorem{lemma}[theorem]{Lemma}

\newtheorem{remark}[theorem]{Remark}
\newtheorem{assumption}[theorem]{Assumption}
\usepackage{lineno}
\nolinenumbers
\usepackage{lineno}
\nolinenumbers
\newcommand*\patchAmsMathEnvironmentForLineno[1]{%
  \expandafter\let\csname old#1\expandafter\endcsname\csname #1\endcsname
  \expandafter\let\csname oldend#1\expandafter\endcsname\csname end#1\endcsname
  \renewenvironment{#1}%
     {\linenomath\csname old#1\endcsname}%
     {\csname oldend#1\endcsname\endlinenomath}}%
\newcommand*\patchBothAmsMathEnvironmentsForLineno[1]{%
  \patchAmsMathEnvironmentForLineno{#1}%
  \patchAmsMathEnvironmentForLineno{#1*}}%
\AtBeginDocument{%
\patchBothAmsMathEnvironmentsForLineno{equation}%
\patchBothAmsMathEnvironmentsForLineno{align}%
\patchBothAmsMathEnvironmentsForLineno{flalign}%
\patchBothAmsMathEnvironmentsForLineno{alignat}%
\patchBothAmsMathEnvironmentsForLineno{gather}%
\patchBothAmsMathEnvironmentsForLineno{multline}%
}
\usepackage{color}
\usepackage{siunitx}
\usepackage[numbers,sort]{natbib}
\newcommand{\dd}{\operatorname{d}\!}

\newcommand{\E}{\mathbb{E}}

\newcommand{\Cov}{\mathrm{Cov}}

\newcommand\norm[1]{\left\lVert#1\right\rVert}
\newcommand{\bs}{\boldsymbol}
\def\RR{ \mathbb R}

\def\Letters{A,B,C,D,E,F,G,H,I,J,K,L,M,N,O,P,Q,R,S,T,U,V,W,X,Y,Z}
\makeatletter
\@for \@l:=\Letters \do{%
  \expandafter\edef\csname\@l bb\endcsname{\noexpand\ensuremath{\noexpand\mathbb{\@l}}}%
  \expandafter\edef\csname\@l bf\endcsname{{\noexpand\bs \@l}}%
  \expandafter\edef\csname\@l cal\endcsname{\noexpand\ensuremath{\noexpand\mathcal{\@l}}}%
  \expandafter\edef\csname\@l eu\endcsname{\noexpand\ensuremath{\noexpand\EuScript{\@l}}}%
  \expandafter\edef\csname\@l frak\endcsname{\noexpand\ensuremath{\noexpand\mathfrak{\@l}}}%
  \expandafter\edef\csname\@l rm\endcsname{{\noexpand\rm \@l}}%
  \expandafter\edef\csname\@l scr\endcsname{\noexpand\ensuremath{\noexpand\mathscr{\@l}}}%
}
\newcommand{\isdef}{\mathrel{\mathrel{\mathop:}=}}

\title{Space-time multilevel quadrature methods \\
and their application for cardiac electrophysiology}

\author[1]{S.~Ben~Bader}
\author[2]{H.~Harbrecht}
\author[1]{R.~Krause}
\author[1]{M.~Multerer}
\author[1,3]{A.~Quaglino}
\author[2]{M.~Schmidlin}

\affil[1]{Center for Computational Medicine in Cardiology,\par
Institute of Computational Science, \par
Universit\`a della Svizzera italiana,
Lugano, Switzerland}
\affil[2]{Department of Mathematics and Computer Science,\par
University of Basel, Basel, Switzerland}
\affil[3]{NNAISENSE SA, Lugano, Switzerland}

\date{\small\sffamily Last update: \today}

\begin{document}

\maketitle

\begin{abstract}
We present a novel approach which aims at high-performance uncertainty
quantification for cardiac electrophysiology simulations. Employing
the monodomain equation to model the transmembrane potential inside
the cardiac cells, we evaluate the effect of spatially correlated
perturbations of the heart fibers on the statistics of the resulting
quantities of interest. Our methodology relies on a close integration of
multilevel quadrature methods, parallel iterative solvers
and space-time finite element discretizations, allowing for a fully
parallelized framework in space, time and stochastics.
Extensive numerical studies are presented to
evaluate convergence rates and to compare the performance of
classical Monte Carlo methods such as standard Monte Carlo (MC)
and quasi-Monte Carlo (QMC), as well as multilevel strategies,
i.e.\ multilevel Monte Carlo (MLMC) and multilevel
quasi-Monte Carlo (MLQMC) on hierarchies of nested meshes.
We especially also employ a recently suggested variant
of the multilevel approach for non-nested meshes
to deal with a realistic heart geometry.
\end{abstract}
\newpage


\section{Introduction}
The heart is by all means one of the most complex and fascinating organs
in the human body. It harmoniously orchestrates the body activity
through the vital supply of blood to all of its components. This
activity is achieved via its pumping function that is a result of
a very complex contraction and relaxation cycle occurring in the cardiac
cells. The latter is itself controlled by a non-trivial pattern of
electrical activation. 

A misregulation of the electrical activity of the heart
can result in several diseases, having in worst cases lethal 
consequences. Therefore, it is of major importance to model 
and understand the heart activity, as this would 
allow for a better clinical diagnosis and treatment of patients.  

The heart consists of fibers, compare Figure \ref{fig::fibers}, 
that help propagate the electrical potential
inside the cardiac muscle. This process is originally initiated by a
stimulus coming from the sinoatrial node (SA) located on top of
the left and right atria. From there, the signal spreads all over the 
heart muscle in the form of a traveling wave front. The propagation 
also takes place in the heart cells that have the ability to actively 
respond to the electrical stimulation through voltage-gated ion channels.

\begin{figure}[H] 
\centering
\includegraphics[width=0.55\textwidth]{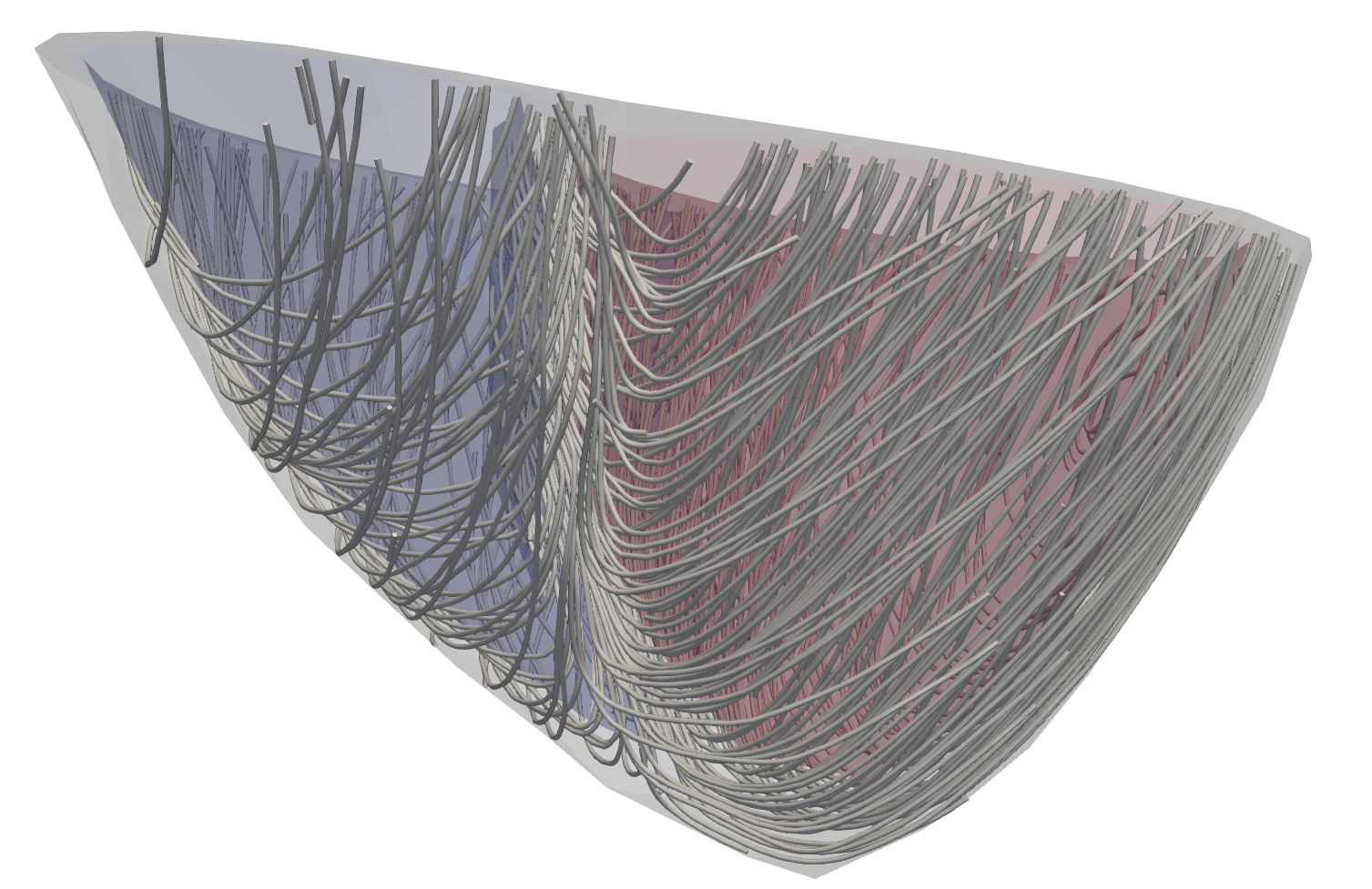}
\caption{Mathematical reconstruction of the
fiber field surrounding a synthetic heart geometry.}
\label{fig::fibers}
\end{figure} 

A mathematical model for describing the potential inside the cardiac
muscle therefore has to provide a suitable ionic channel model
in order to fully capture the phenomenological behaviour. Combined with
time--dependency, spatial diffusion and a forcing function modelling
the initial stimulus, one arrives at the monodomain equation,
which has been derived in \cite{hurtado2014gradient,miller1978simulation}.
It can be written in the following form:

\begin{alignat}{2} \label{monodomain}
  \dfrac{\partial u({\bs x}, t) }{\partial t} 
 -  \nabla \cdot\big( {\bs G}({\bs x}) \nabla u({\bs x}, t)\big)
 + I_\text{ion}\big(u({\bs x}, t)\big) & = I_\text{app} ({\bs x} ,t),
 &&\quad\text{for } ({\bs x},t) \in D \times (0,T],\nonumber \\
    {\bs G} ({\bs x}) \nabla u ({\bs x}, t) \cdot n & = 0,
    &&\quad\text{for  }({\bs x},t) \in \partial D \times (0,T],\\ 
     u({\bs x},0) & =  0, &&\quad\text{for }{\bs x} \in D.
     \nonumber
\end{alignat}
Here, 
$D \subset \mathbb{R}^d$ is the domain representing the heart,
$u=u({\bs x}, t)$ is the electrical potential, 
${\bs G} \colon D\to\mathbb{R}^{d\times d}$ is an anisotropic 
conductivity tensor
modeling the fiber direction,
$T\in \mathbb{R}^+$ is the end time,
$I_\text{app}\colon D \times [0,T] \rightarrow \mathbb{R}$ is
the forcing function for the stimulus created by the SA node,
and $I_\text{ion}\colon\mathbb{R} \rightarrow \mathbb{R}$
is an ion channel model.
The latter can be modeled in several ways accounting for different
levels of detail and complexity, see \cite{hodgkin1952quantitative}.
We rely here on the Fitz-Hugh Nagumo model, see
\cite{fitzhugh1961impulses},
for which we have: 
\begin{equation}
I_\text{ion}(u) = \alpha (u-u_\text{rest})(u-u_\text{th})
(u-u_\text{peak}),\quad\alpha>0.
\label{ionic_term}
\end{equation}
The values $u_\text{rest}$, $u_\text{th}$ and $u_\text{peak}$ 
are characteristic potential values of the electrical activation process.
They respectively represent the resting potential $u_\text{rest}$
(cell is unactivated), the threshold potential $u_\text{th}$
(cell is triggered) and the peak value $u_\text{peak}$ (cell is activated).

As depicted in Figure \ref{fig::fibers}, the fibers have a very complex but 
also well-organized structure, exhibiting key features that can be identified 
in all healthy subjects, such as a helical distribution with opposite orientations, 
from the endocardium to the epicardium. However, the exact fiber dislocations 
can vary not only from patient to patient, but can also change over time within 
the same patient due to pathologies, such as infarctions. Then, the fiber 
structure is perturbed with the introduction of high variability areas
in the presence of scars. To faithfully model the conductivity tensors 
used in electrophysiology, accurate measurements of these fibers 
are required. However, such measurements cannot be made available on a routine basis. 
Given that a highly accurate model of fibers and thus the conductivities 
are generally unavailable except for a few test subjects, it is of paramount importance 
to gauge the influence that uncertainties in the conductivity tensor have 
on the simulated activation patterns. Modeling and simulating the uncertainty 
in the fiber directions will be one of the major aspects of this article. 

We apply state of the art methods in uncertainty quantification (UQ). 
This means that we combine space-time GMRES with a block Jacobi 
preconditioner \cite{Benedusi2018} for solving the monodomain 
equation with multilevel quadrature methods for the UQ.
In our practical implementation, we use the multilevel (quasi-) 
Monte Carlo method, compare~\cite{barth2011multi,G08,GHM20,harbrecht2012multilevel,H98}.
Therefore, additional smoothness of the solution is required as already pointed out 
in \cite{harbrecht2012multilevel}. This smoothness has been verified in the stationary 
case for anistropic random diffusion problems in \cite{HaS20}, for linear random 
advection-diffusion-reaction problems in \cite{HeS20}, and for semilinear random 
diffusion problems in \cite{Schmidlin21}. Our numerical results in Section~\ref{sct:numerix}
show that the quasi-Monte Carlo method based on Halton points is superior 
over the Monte Carlo method and that the multilevel versions are superior
over the single-level versions of these quadrature methods. Indeed, the highest 
efficiency is provided by the multilevel quasi-Monte Carlo method.

The rest of this article is organized as follows. In Section~\ref{sec:Preliminaries},
we present the random model for the fibers of the heart muscle. Then, 
Section~\ref{sec:FEM}, is concerned with the space-time solver for
the monodomain equation. Quadrature methods to treat the 
randomness are outlined in Section~\ref{section:methods}. Finally,
in Section~\ref{sct:numerix}, numerical experiments are presented
in order to validate the present approach also in case of simulations
for realistic heart geometries. 

\section{Preliminaries}\label{sec:Preliminaries}
 \subsection{Random fiber directions}\label{subsec:randDiffusion}
Within this article, we will consider an uncertainty on the random fiber directions. To this end,
let \((\Omega,\mathcal{F},\mathbb{P})\) denote a complete and separable
probability space.
Then, for a given Banach space $\Xcal$ and $1 \leq p \leq \infty$,
the space $L^p(\Omega;\Xcal)$ 
denotes the Lebesgue-Bochner space,
see \cite{HP57},
which contains all equivalence classes of strongly measurable functions
$v \colon \Omega\to \Xcal$ with finite norm
\begin{equation*}
  \norm{v}_{L^p} \isdef
  \begin{cases}
    {\displaystyle\bigg(\int_\Omega
    \|{v(\omega)}\|_{\Xcal}^p}\bigg)^{1/p}\dd\Pbb(\omega),
    & p < \infty , \\
    \displaystyle\operatorname{ess\,sup}_{\omega\in\Omega}
    \|{v(\omega)}\|_{\Xcal} , & p = \infty.
  \end{cases}
\end{equation*}
In this context, a function $v \colon \Omega \to \Xcal$ is said to be
strongly measurable if there exists 
a sequence of simple functions
$v_n \colon \Omega\to \Xcal$,
such that for almost every $\omega\in\Omega$ we have
$\lim_{n \to \infty} v_n(\omega) = v(\omega)$.
Note that we also have the usual inclusion
$L^p(\Omega; \Xcal) \supset L^q(\Omega; \Xcal)$ provided that 
$1\leq p \le q \leq \infty$.
Given that $\big(\Xcal,(\cdot,\cdot)_{\Xcal}\big)$ is a
separable Hilbert space,
the Bochner space $L^2(\Omega; \Xcal)$ is a separable Hilbert space
as well, where
the inner product is defined as
\begin{equation*}
  (u, v)_{L^2} \isdef \int_\Omega\big({u(\omega),
  v(\omega)}\big)_{\Xcal}\dd\Pbb(\omega).
\end{equation*}
In particular, this space is isometrically isomorphic to the tensor
product space
$L^2(\Omega) \otimes \Xcal$, we refer to \cite{LC} for the details.

Subsequently, we will always equip the space $\Rbb^d$ with the
Euclidean
norm $\norm{\cdot}_2$
induced by the canonical inner product $\langle\cdot,\cdot\rangle$
and $\Rbb^{d \times d}$ with the norm $\norm{\cdot}_F$
induced by the Frobenius inner product $\langle{\bs A},{\bs B}\rangle_F
\isdef\operatorname{tr}({\bs A}^\intercal{\bs B})$.
To account for the anisotropies generated by the cardiac fibers, we
consider a conductivity tensor as proposed and analyzed in \cite{HPS17},
see also \cite{HaS20}. It is of the form
\begin{equation}
\label{eq:AdmV}
  {\bs G}({\bs x}, \omega) \isdef g {\bs I} +
  \big(\|{{\bs V}({\bs x}, \omega)}\|_2 - g\big)
 \frac{{\bs V}({\bs x}, \omega) {\bs V}^\intercal({\bs x}, \omega)}
 {{\bs V}^\intercal({\bs x}, \omega){\bs V}({\bs x}, \omega)} ,
\end{equation}
where $g>0$ is a given value and
${\bs V}\in L^\infty\big(\Omega; L^\infty(D; \Rbb^d)\big)$
is a random vector field.
Moreover, we require that there exist some constants
$b_{\min} \leq 1$ and $b_{\max} \geq 1$ such that 
$b_{\min} \leq g \leq b_{\max}$ and
\begin{equation}
  \label{eq:Vellipticity}
  b_{\min} \leq \operatorname{ess\,inf}_{{\bs x} \in D}
  \|{\bs V}({\bs x}, \omega)\|_2
  \leq \operatorname{ess\,sup}_{{\bs x} \in D}
  \|{\bs V}({\bs x}, \omega)\|_2 \leq b_{\max}
  \quad\text{$\Pbb$-almost surely}.
\end{equation}

The model \eqref{eq:AdmV} represents a medium 
that has homogeneous diffusion strength $g$ perpendicular to ${\bs V}$
and diffusion strength $\|{\bs V}({\bs x}, \omega)\|_2$
in the direction of ${\bs V}$. The randomness of the specific direction 
and length of ${\bs V}$ therefore quantifies the uncertainty of this 
notable direction and its diffusion strength.

\begin{lemma}
 A conductivity tensor of the form \eqref{eq:AdmV} is well-defined
and indeed also satisfies a uniform ellipticity condition, i.e.\
\begin{equation}\label{eq:Gellipticity}
b_{\min}\leq\operatorname{ess\,inf}_{{\bs x} \in D}
\|{\bs G}({\bs x},\omega)\|_2\leq
\operatorname{ess\,sup}_{{\bs x} \in D}\|{\bs G}({\bs x},\omega)\|_2\leq b_{\max}
  \quad\text{$\Pbb$-almost surely}.
\end{equation}
\end{lemma}

\begin{proof}
  For almost every $\omega \in \Omega$ and almost every ${\bs x} \in D$, 
  we have that ${\bs G}({\bs x}, \omega)$ is well-defined, because of
\[
    {\bs V}^\intercal({\bs x}, \omega) {\bs V}({\bs x}, \omega)
    = \|{\bs V}({\bs x}, \omega)\|_2^2 \geq b_{\min}^2 > 0 ,
\]
  and clearly symmetric. 
 Furthermore, we can choose ${\bs u}_2, \ldots, {\bs u}_d \in \Rbb^d$
 that are perpendicular to ${\bs V}({\bs x}, \omega)$
  and are linearly independent.
  Thus, for $i = 2, \ldots, d$, it holds that
  \begin{equation*}
    {\bs G}({\bs x}, \omega) {\bs u}_i = g {\bs u}_i
\quad\text{and}\quad 
    {\bs G}({\bs x}, \omega) {\bs V}({\bs x}, \omega) =
    \|{\bs V}({\bs x}, \omega)\|_2 {\bs V}({\bs x}, \omega) .
  \end{equation*}
  Consequently, we obtain
  for almost every $\omega \in \Omega$ and almost every ${\bs x} \in D$ that
  \begin{align*}
    \lambda_{\min}\big({\bs G}({\bs x}, \omega)\big) &=
    \min\{g, \|{\bs V}({\bs x}, \omega)\|_2\} \geq b_{\min},\\
   \lambda_{\max}\big({\bs G}({\bs x}, \omega)\big) &
   = \max\{g, \|{\bs V}({\bs x}, \omega)\|_2\} \leq b_{\max}.
  \end{align*}
This shows \eqref{eq:Gellipticity}.
\end{proof}

 \subsection{Karhunen-Lo\`eve expansion}\label{subsec:KLexpansion}
To make random (vector) fields feasible for numerical computations,
we separate the spatial variable ${\bs x}$ and the stochastic parameter
$\omega$ by computing the Karhunen-Lo\`eve expansion.
To this end, we require the expectation and the covariance
of the underlying random field.
For example, in case of \({\bs V}\), they are given by
\begin{equation*}
  \Ebb[{\bs V}]({\bs x})
  = \int_\Omega {\bs V}({\bs x}, \omega) \dd{\Pbb(\omega)}
\end{equation*}
and
\begin{equation*}
  \operatorname{Cov}[{\bs V}]({\bs x}, {\bs x}')
  = \int_\Omega {\bs V}_0({\bs x}, \omega) {\bs V}_0^\intercal({\bs x}', \omega) \dd{\Pbb(\omega)} ,
\end{equation*}
respectively, where
\begin{equation*}
  {\bs V}_0({\bs x}, \omega) \isdef {\bs V}({\bs x}, \omega) - \Ebb[{\bs V}]({\bs x})
\end{equation*}
denotes the centered vector field.

Given the eigenpairs
\(\{\lambda_k,{\boldsymbol\psi}_k\}_k\) of the Hilbert-Schmidt operator
{\(\Ccal\)} defined by \(\operatorname{Cov}[{\bs V}]\), that is
\begin{equation*}
    (\Ccal {\bs u})({\bs x})\isdef \int_D \Cov[{\bs V}]({\bs x}, {\bs x}') {\bs u}({\bs x}') \dd{{\bs x}'},
\end{equation*}
the Karhunen-Lo\`eve expansion of \({\bs V}\) reads
\begin{equation}\label{eq:KL}
 {\bs V}({\bs x}, \omega)
 = \Ebb[{\bs V}]({\bs x}) + \sum_{k=1}^{\infty}
 \sqrt{\lambda_k} {\boldsymbol\psi}_k({\bs x}) Y_k(\omega).
\end{equation}
Herein, the uncorrelated, normalised and centered random variables 
\(\{Y_k\}_k\) are obtained in accordance with
\begin{equation*}
  Y_k(\omega) \isdef \frac{1}{\sqrt{\lambda_k}} \int_D 
  {\bs V}_0^\intercal({\bs x}, \omega){\boldsymbol\psi}_k({\bs x}) \dd{{\bs x}}.
\end{equation*}
Note that as ${\bs V}\in L^\infty\big(\Omega; L^\infty(D; \Rbb^d)\big)$
we particularly know that ${\boldsymbol\psi}_k \in L^\infty(D; \Rbb^d)$
and $Y_k \in L^\infty\big(\Omega; \Rbb)$, see \cite{HPS17}.

Now, by introducing \(\sigma_k\isdef\sqrt{\lambda_k} \norm{Y_k}_{L^\infty(\Omega; \Rbb)}\)
we can assume, without loss of generality, that $Y_k \in[-1,1]$ and thus
may instead consider the vector field ${\bs V}$ in the parametrised form
\begin{equation}
  \label{eq:KLp}
  {\bs V}({\bs x}, {\bs\omega})
  = \Ebb[{\bs V}]({\bs x}) +
  \sum_{k=1}^{\infty} \sigma_k {\boldsymbol\psi}_k({\bs x})\omega_k ,
\end{equation}
where ${\bs\omega}\isdef [\omega_k]_{k \in \Nbb} \in \square \isdef [{-1}, 1]^{\Nbb}$
and \(\omega_k\) is the canonical random variable on the probability
space \(\big([-1,1],\mathcal{B}([-1,1]),\mathbb{P}_{Y_k}\big)\).
Consequently, we can also view ${\bs G}({\bs x}, {\bs\omega})$
as being parametrised by ${\bs\omega}$.

We now impose some common assumptions,
which make the Karhunen-Lo\`eve expansion computationally feasible.
\begin{assumption}
  The random variables $\{Y_k\}_{k \in \Nbb}$ are independent and uniformly distributed
  on \([-1,1]\), which indeed implies that \(\sigma_k = \sqrt{3 \lambda_k}\)
  and that $\mathbb{P}_{Y_k}$ coincides with the normalised Lebesgue measure on $[-1,1]$.
  Moreover, the sequence ${\boldsymbol\gamma} = \{\gamma_k\}_k$, given by 
  \begin{equation*}
    \gamma_k \isdef\|\sigma_k {\boldsymbol\psi}_k\|_{L^\infty(D; \Rbb^d)} ,
  \end{equation*}
  is at least in $\ell^1(\Nbb)$,
  where we have set ${\boldsymbol\psi} \isdef \Ebb[{\bs V}]$ and
  $\sigma_0 \isdef 1$.
\end{assumption}

 \subsection{Discretization of the random vector field}\label{subsec:randField}
 The Karhunen-Lo\`eve expansion in the form of \eqref{eq:KLp} 
 cannot directly be used on a computer. In what follows, we therefore
 present a means how the Karhunen-Lo\`eve expansion can be numerically
 approximated with finite elements. As before, we only consider
 the vector valued case here.
 Let the random vector field be given by its expectation
 \(\E[{\bs V}]({\bs x}) = [\E_i[{\bs V}]({\bs x})]_{i=1}^d\)
 and its covariance function
 \(\Cov[{\bs V}]({\bs x},{\bs x}')
 =[\Cov_{i,j}[{\bs V}]({\bs x},{\bs x}')]_{i,j=1}^d\), which we assume to be at
 least continuous.
Moreover, let \(\{{\bs x}_i\}_{i=1}^n\subset D\) be the vertices of
the nodal finite element basis
\(\{\phi_1,\ldots,\phi_n\}\),
i.e.\ \(\phi_i({\bs x}_j)=\delta_{i,j}\),
coming from the finite element space $\mathcal{S}_L$
where \(n = n_L = \operatorname{dim}(\mathcal{S}_L)\),
see Subsection~\ref{disc}.
Then, we can approximate the expectation by its finite element interpolant
\[
\E[{\bs V}]({\bs x})\approx\sum_{i=1}^n
\E[{\bs V}]({\bs x}_i)\phi_i({\bs x})
\]
and in complete analogy the covariance by
\[
\Cov[{\bs V}]({\bs x},{\bs x}')\approx\sum_{i,j=1}^n
\Cov[{\bs V}]({\bs x}_i,{\bs x}_j)\phi_i({\bs x})\phi_j({\bs x}').
\]

In order to determine the Karhunen-Lo\`eve expansion of \({\bs V}\), 
we have to solve the operator eigenvalue problem
\[
\int_{D}\Cov[{\bs V}]({\bs x},{\bs x}'){\boldsymbol\psi}({\bs x}')\dd{\bs x}'=\lambda{\boldsymbol\psi}({\bs x}).
\]
Thus, by replacing \(\Cov[{\bs V}]\) with its finite element interpolant
and testing with respect to the basis functions \({\phi}_i\otimes{\bs e}_j\), \(i=1,\ldots,n\),
\(j=1,\ldots,d\), where \(\{{\bs e}_j\}_j\) is the canonical basis of \(\Rbb^d\),
we end up with the generalized algebraic eigenvalue problem
\begin{equation}\label{eq:algEig}
\begin{bmatrix}{\bs M} & & \\
& \ddots & \\
& & {\bs M}\end{bmatrix}
{\bs C}
\begin{bmatrix}{\bs M} & & \\
& \ddots & \\
& & {\bs M}\end{bmatrix}
{\bs v}=\lambda\begin{bmatrix}{\bs M} & & \\
& \ddots & \\
& & {\bs M}\end{bmatrix}{\bs v},\quad{\bs v}\in\mathbb{R}^{dn}.
\end{equation}
Herein, the matrix
\[
{\bs C}\isdef\begin{bmatrix}
 \big[\Cov_{1,1}[{\bs V}]({\bs x}_i,{\bs x}_j)\big]_{i,j=1}^n & \cdots & \big[\Cov_{1,d}[{\bs V}]({\bs x}_i,{\bs x}_j)\big]_{i,j=1}^n\\
 \vdots & \ddots & \vdots \\
\big[\Cov_{d,1}[{\bs V}]({\bs x}_i,{\bs x}_j)\big]_{i,j=1}^n &\cdots &
\big[\Cov_{d,d}[{\bs V}]({\bs x}_i,{\bs x}_j)\big]_{i,j=1}^n
\end{bmatrix}\in\mathbb{R}^{dn\times dn}
\]
is the covariance function evaluated in all combinations of grid points, while
\[
{\bs M}\isdef[m_{i,j}]_{i,j=1}^n\in\mathbb{R}^{n\times n}\quad\text{with }m_{i,j}\isdef
\int_{D}\phi_j\phi_i\dd{\bs x}
\]
denotes the finite element mass matrix.

The algebraic eigenvalue problem \eqref{eq:algEig} can now 
efficiently be solved by means of the pivoted Cholesky decomposition 
as follows:
Let \({\bs C}\approx{\bs L}{\bs L}^\intercal\) with \({\bs L}\in\mathbb{R}^{dn\times M}\)
and \(M\ll n\)
be the low-rank approximation generated by the 
pivoted Cholesky decomposition of \({\bs C}\) as described in, e.g.\ \cite{HPS12,HPS14}.
Then, we approximate the eigenvalue problem \eqref{eq:algEig} by
\begin{equation}\label{eq:approxEig}
\begin{bmatrix}{\bs M} & & \\
& \ddots & \\
& & {\bs M}\end{bmatrix}
{\bs L}{\bs L}^\intercal
\begin{bmatrix}{\bs M} & & \\
& \ddots & \\
& & {\bs M}\end{bmatrix}
{\bs v}=\lambda\begin{bmatrix}{\bs M} & & \\
& \ddots & \\
& & {\bs M}\end{bmatrix}{\bs v},\quad{\bs v}\in\mathbb{R}^{dn}.
\end{equation}
This eigenvalue problem is equivalent to the much smaller eigenvalue problem
\begin{equation}\label{eq:redEig}
{\bs L}^\intercal\begin{bmatrix}{\bs M} & & \\
& \ddots & \\
& & {\bs M}\end{bmatrix}{\bs L}\tilde{\bs v}=\lambda\tilde{\bs v},
\quad\tilde{\bs v}\in\mathbb{R}^M.
\end{equation}
In particular, if \(\tilde{\bs v}_i\) is an eigenvector of \eqref{eq:redEig}
with eigenvalue \(\lambda_i\), then \({\bs v}_i\isdef{\bs L}\tilde{\bs v}_i\)
is an eigenvector of \eqref{eq:approxEig} with eigenvalue \(\lambda_i\).
Moreover, there holds
\[
{\bs v}_i^\intercal\begin{bmatrix}{\bs M} &  & \\
& \ddots & \\ & & {\bs M}\end{bmatrix}{\bs v}_j=\lambda_i\delta_{i,j}.
\]
\begin{remark}
The cost for computing the pivoted Cholesky decomposition is
\(\mathcal{O}(dnM^2)\) and, since all entries of
\({\bs C}\) can be computed on the fly without the need of storing the
entire matrix \({\bs C}\), 
the storage cost is \(\mathcal{O}(dnM)\). Moreover,
the small eigenvalue problem \eqref{eq:redEig} can be solved 
with cost \(\mathcal{O}(M^3)\). Thus, since usually \(M\ll n\),
the overall cost for computing the Karhunen-Lo\`eve expansion 
of \({\bs V}\) by the suggested approach 
is also \(\mathcal{O}(dnM^2)\) in total.
\end{remark}

Based on the suggested low-rank approach, we end up with
a discretized random field of the form
\begin{equation}
  \label{eq:KLlowrank}
  \tilde{\bs V}({\bs x}, {\bs\omega})
  = \sum_{i=1}^n\Ebb[{\bs V}]({\bs x}_i)\phi_i({\bs x}) +
  \theta\sum_{k=1}^{M}{\sigma}_k\omega_k
  \sum_{i=1}^n{\bs c}_{k,i}\phi_i({\bs x}),\quad{\bs\omega}\in[-1,1]^M,
\end{equation}
where the coefficients \({\bs c}_{k,i}\in\mathbb{R}^d\) are obtained
from combining all coefficients from the eigenvector \({\bs v}_k\)
that interact with the basis function \({\phi}_i\). Moreover, we introduce the scaling parameter \(\theta>0\) to guarantee \eqref{eq:Vellipticity} in our numerical studies.

\subsection{Quantities of interest} \label{subsection:qoi}
Due to the randomness of the heart fibers' orientations as described 
above, the monodomain equation \eqref{monodomain} 
now translates into the following parametric version 
provided for all ${\bs\omega}\in [-1,1]^M$:
\begin{equation} \label{eq:stochastic-monodomain}
  \dfrac{\partial u({\bs z},{\bs\omega}) }{\partial t} 
 -  \nabla \cdot\big( {\bs G}({\bs x}) \nabla u({\bs z},{\bs\omega})\big)
 + I_\text{ion}\big(u({\bs z},{\bs\omega})\big)  = I_\text{app} ({\bs z}),
 \quad {\bs z}\isdef({\bs x},t) \in D \times (0,T] .
\end{equation}
Our aim is to determine statistics of the random solution \(u({\bs z},{\bs\omega})\),
which amounts to
the evaluation of the high-dimensional integral given by
\begin{equation} \label{eq:qoi-integral}
    \operatorname{QoI}[u] = \int_{[-1,1]^M} \mathcal{F}\big(u(\cdot,{\bs\omega})\big) \rho ({\bs\omega}) \dd{\bs\omega}.
\end{equation}
Here, $\rho ({\bs\omega}) = \prod_{i=1}^M \rho_k(\omega_k)$ is the joint density function of \({\bs\omega}\) from \eqref{eq:KLlowrank} and $\mathcal{F}$ denotes
a functional that encodes a particular quantity of interest. 
We shall focus here on three different quantities of interest. 

\paragraph{Transmembrane potential.}
The transmembrane potential over the totality of the heart geometry 
and its evolution in time is the quantity obtained by solving the 
monodomain equation. This is demonstrated as an electrical potential 
wave travelling through the heart,
cf.\ Figure $\ref{fig::transmembrane}$. The functional $\mathcal{F}$ in this
case is simply the identity function, i.e.\
\[\mathcal{F}\big(u(\cdot,{\bs\omega})\big) = u(\cdot,{\bs\omega}).\]
We remark that considering a fine discretization in space and time, the full information on the transmembrane potential represents a high-dimensional
output that might easily become a burden at the memory level in a context
of a UQ study. 

\begin{figure}[htb]
\centering
\includegraphics[width=0.3\textwidth]{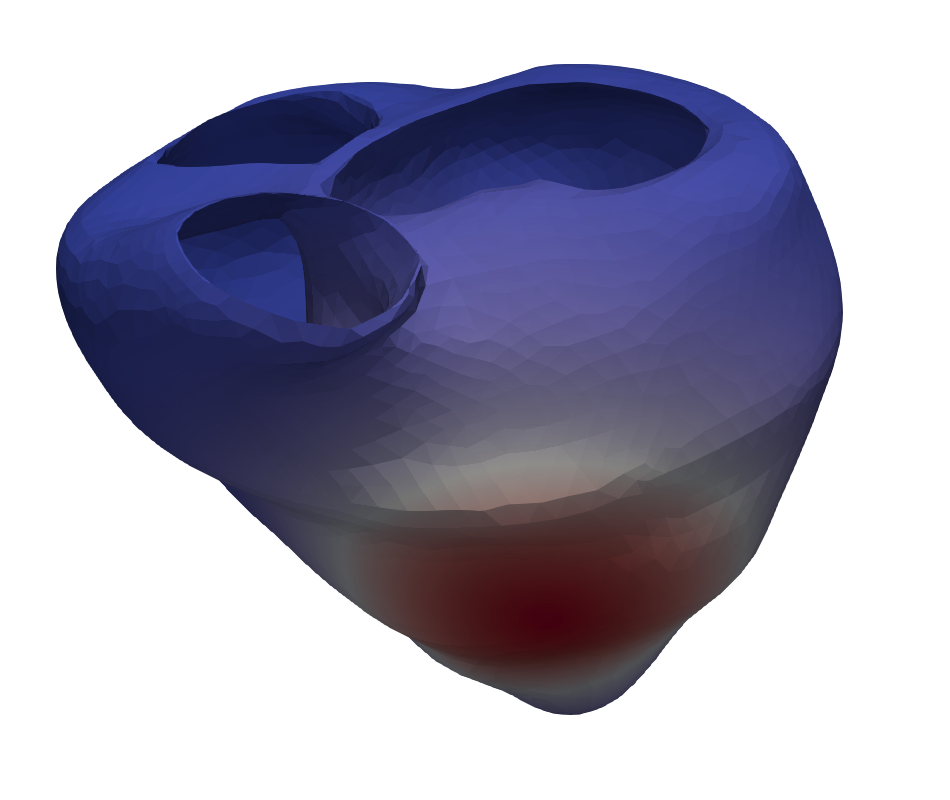}
\includegraphics[width=0.3\textwidth]{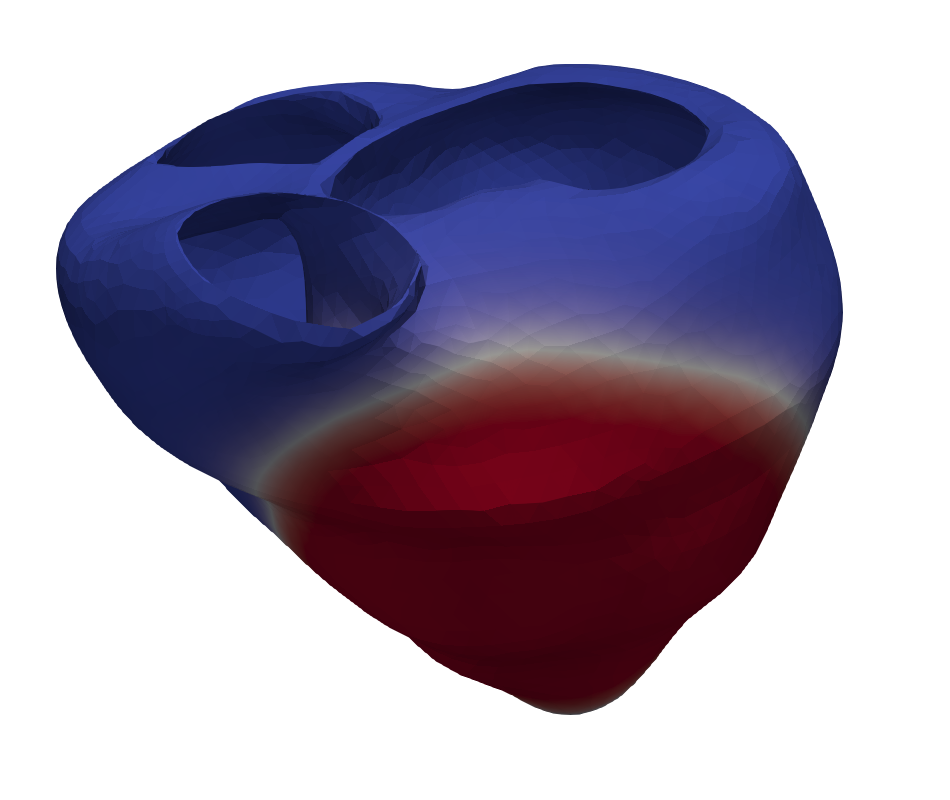}
\includegraphics[width=0.3\textwidth]{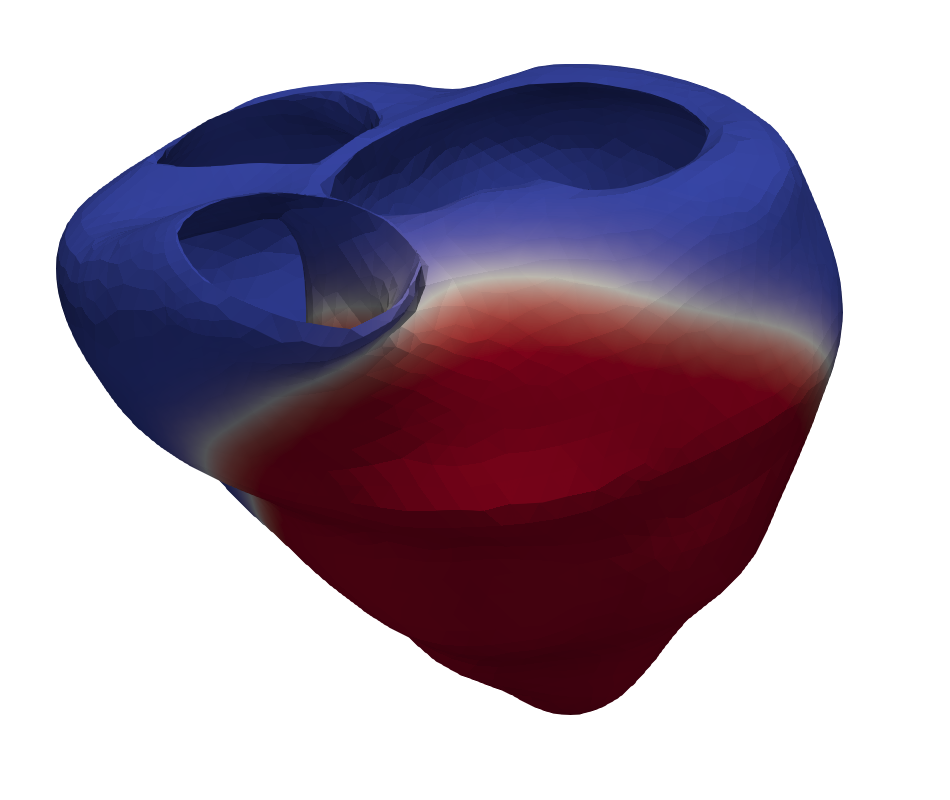}
\includegraphics[width=0.3\textwidth]{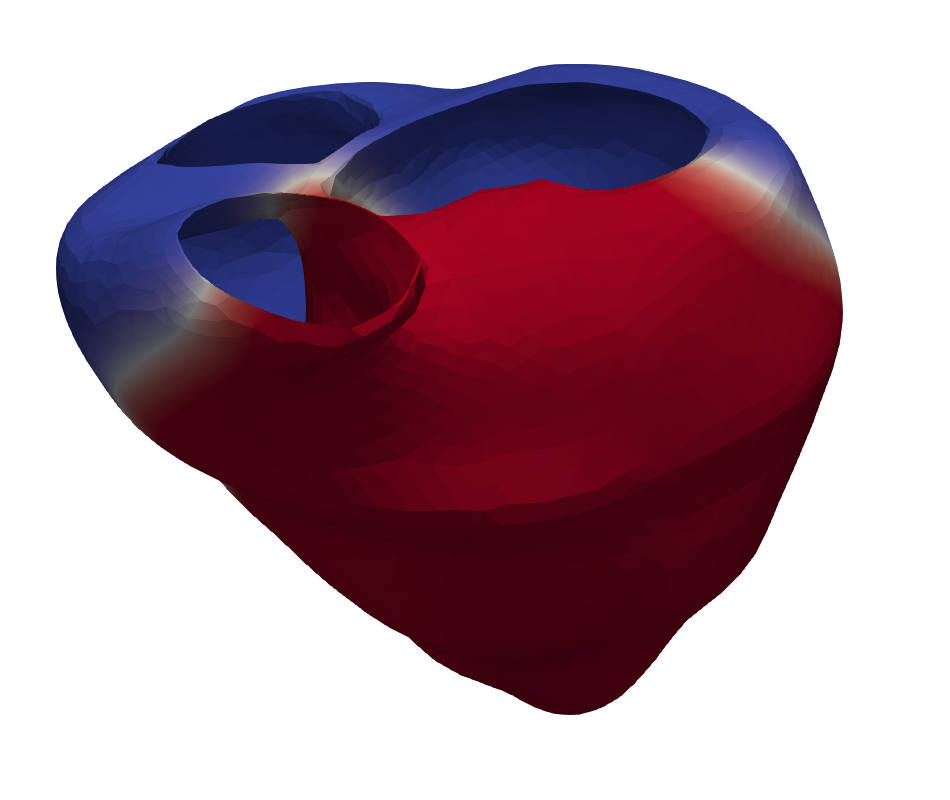}
\includegraphics[width=0.3\textwidth]{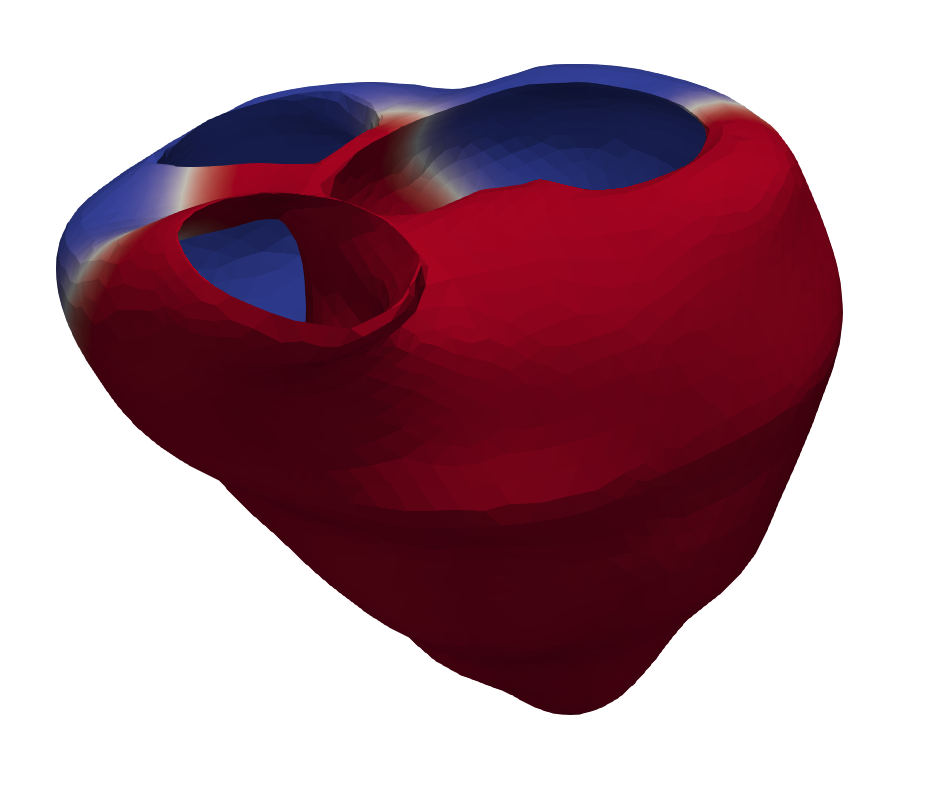}
\includegraphics[width=0.3\textwidth]{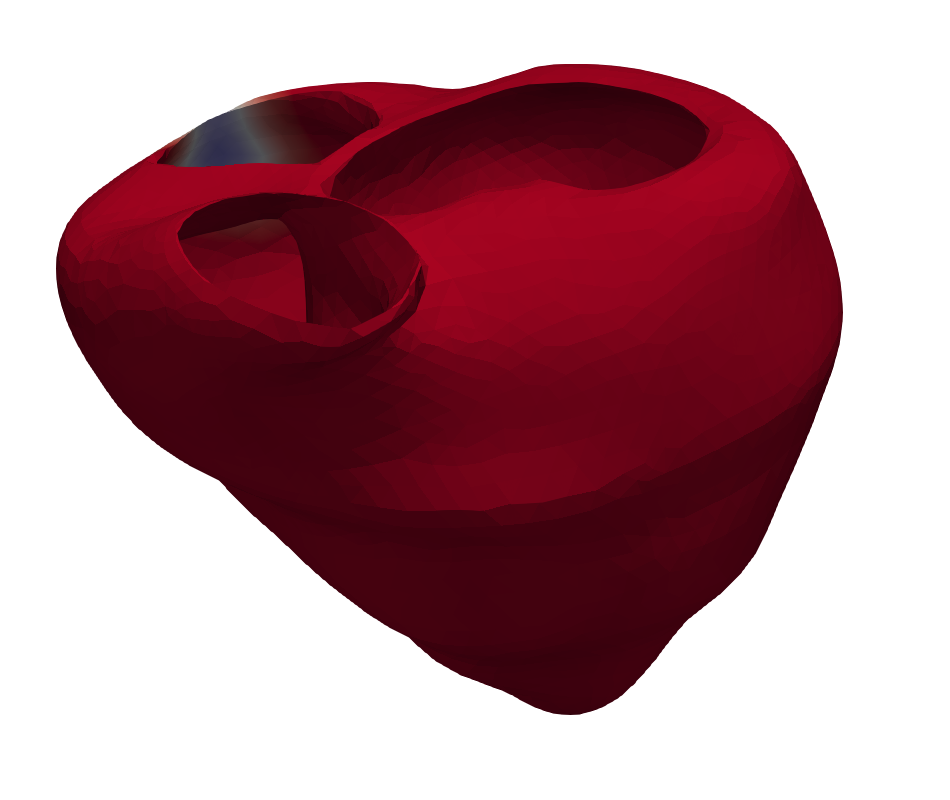}
\caption{Wavefront propagation of transmembrane potential.}
\label{fig::transmembrane}
\end{figure} 

\paragraph{Activation map.}
The transmembrane potential can be used in order to extract the activation map of the heart. This reduces the size of the output to that of the dimension indicating the times at which cells are activated. The activation time $a({\bs x}_0,{\bs\omega})$
at a given location \({\bs x}_0\) is defined as the right inverse  
\[a({\bs x}_0,{\bs\omega}) =
\min \{ t \in [0,T] : u({\bs x}_0,t,{\bs\omega}) \geq u_{\text{th}} \}.
\] The functional $\mathcal{F}$ for this case can therefore be written as
$\mathcal{F}\big(u(\cdot,{\bs\omega})\big) = a({\bs x}_0,{\bs\omega})$.

\begin{figure}[htb]
\centering
\includegraphics[width=0.3\textwidth]{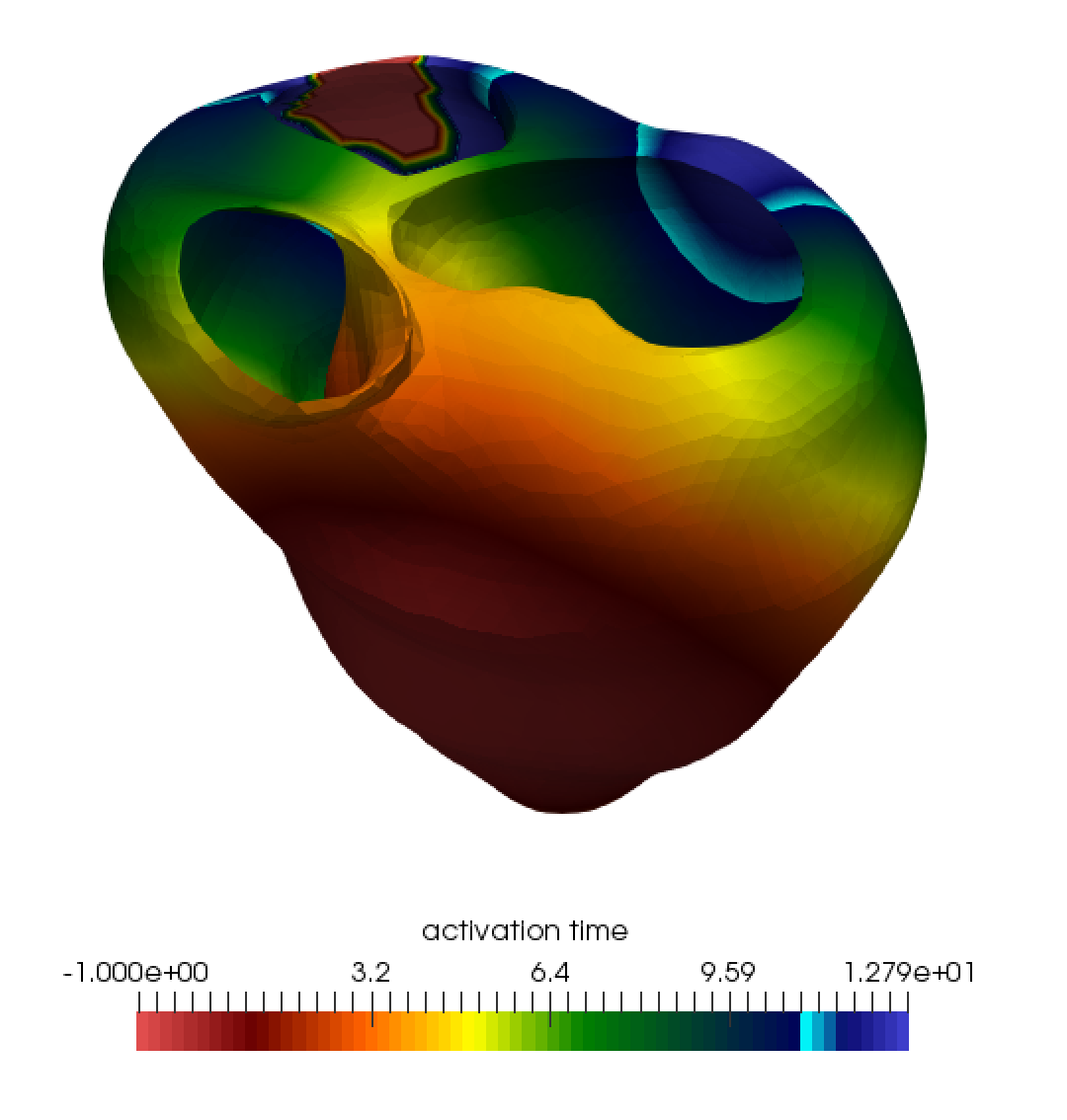}
\includegraphics[width=0.3\textwidth]{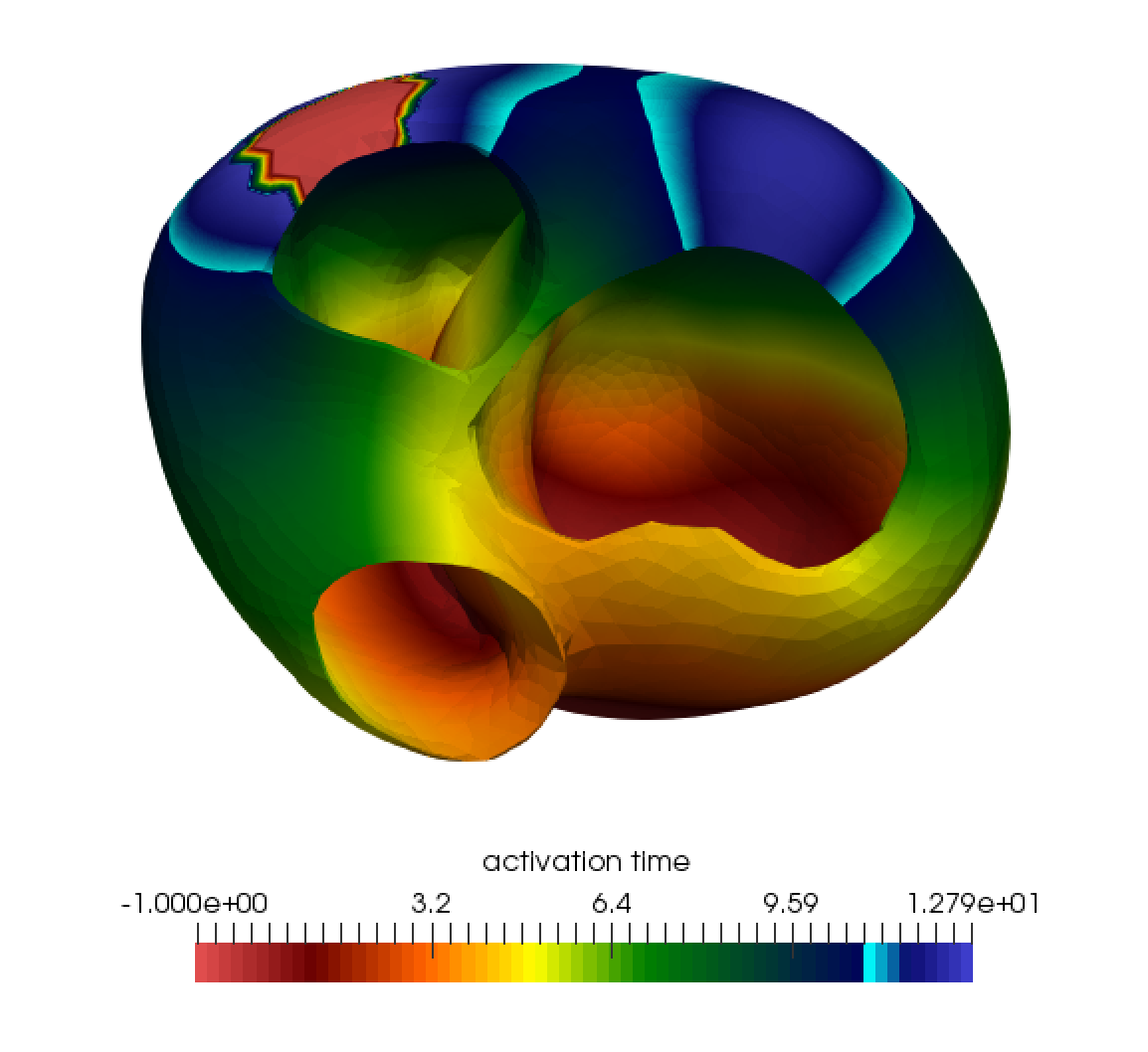}
\includegraphics[width=0.3\textwidth]{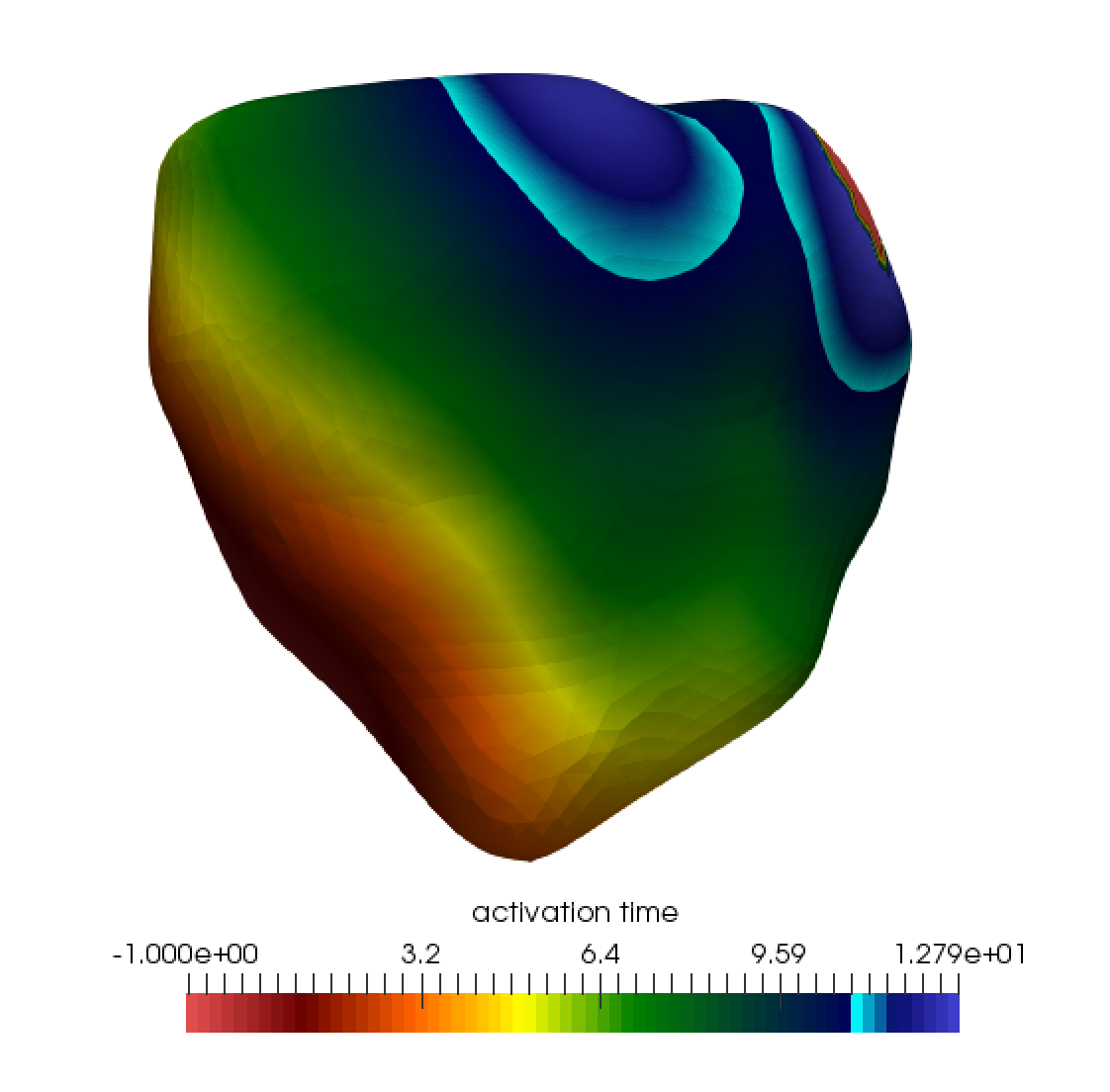}
\caption{Activation map of a heart. Negative value signifies non activated region.}
\label{fig::activation}
\end{figure} 

\paragraph{Action potential.}
Another relevant quantity of interest is the evaluation 
of the transmembrane time evolution at a given location ${\bs x}_0$. 
Mathematically speaking, this corresponds to the functional 
\[
\mathcal{F}\big(u(\cdot,{\bs\omega})\big) 
= u({\bs x}_0,t,{\bs\omega}).
\]

Regarding the approximation of the integral $\eqref{eq:qoi-integral}$ 
for a given functional $\mathcal{F}$, we rely on high-dimensional quadrature methods 
that require solving the monodomain equation in quadrature points represented
by different realizations of ${\bs\omega} \in [-1,1]^M.$ The resulting procedure
is a sampling method and requires a finite element solve for every sample. 

\section{Discretization of the monodomain equation}\label{sec:FEM}
The parametric monodomain equation \eqref{eq:stochastic-monodomain} can be solved numerically for all parameters ${\bs\omega}\in [-1,1]^M$ by means of finite elements in space and finite differences in time, i.e.\ by using a sequential time-stepping method. In view of employing multilevel quadrature methods for approximating the integral \eqref{eq:qoi-integral}, we set the stage for a similar refinement rate of the space and time grid resolutions. In particular, we employ an all-at-once approach in space and time, where we assemble a large space-time system that is solved in parallel \cite{mcdonald2016simple}. 

This approach allows for a similar error decay with respect to the space and time discretization steps. In addition, it enhances the parallel scalability of the numerical method, by allowing parallelization also in the time dimension. 
For a comprehensive review of parallel-in-time methods, see \cite{gander201550}.\\

\subsection{Space-time assembly of the heat equation } \label{disc}



%
Let us consider a nested sequence of shape regular 
tetrahedralizations $\{ \mathcal{T}_l \}_{l \geq 0}$ of the 
spatial domain $D$, where each $\mathcal{T}_l$ is of 
mesh size $h_l \thicksim 2^{-l}$. For all levels $l \geq 0$,  
we define the continuous, piecewise linear finite element spaces 
\[
    \mathcal{S}_l = \{ v_l \in C^0(D) : v_{l|T} \in \mathbb{P}_1(T),\ T \in  \mathcal{T}_l \}.
\]
We denote with $\{\phi_{l,i}\}_{i=1}^{n_l} \subset \mathbb{P}_1$ the 
sets of linear nodal basis functions for $\mathcal{S}_l$.
For each index $l$ we also partition the time interval $[0, T]$ into $m_l-1$ equisized
subintervals of length $\Delta t_l = T/(m_l-1)$, such that $\Delta t_l \sim h_l$.
This uniform partition is thus given by the nodes
$t_{l,k} = (k-1)\Delta t_l$ with $k=1,\ldots,m_l$.

We start by neglecting the non-linear term from \eqref{eq:stochastic-monodomain}, 
that is $I_{\text{ion}}$, to derive instead the space-time linear system arising from
the closely related heat equation. For the sake of readability, we also assume 
that a particular realization of parameter ${\bs\omega} \in [-1,1]^M$ is given, 
and therefore disregard it for the analysis that follows. Assuming that the 
solution $u({\bs x} , t)$ is sufficiently regular in $D$, we derive the weak formulation:
\begin{align*}
&\text{for all $ t \in (0,T] $, find $u(\cdot, t) \in H^1(D)$ such that}\\
&\qquad\int_D \dfrac{\partial u({\bs x} , t)}{\partial t} v({\bs x})\dd{\bs x}
+\int_D {\bs G}({\bs x}) \nabla u({\bs x},t) \nabla v({\bs x}) \dd{\bs x}
= \int_D I_{\text{app}}({\bs x} , t) v({\bs x}) \dd{\bs x}\\
&\hspace*{7cm}\text{for all } v \in H^1(D).
\end{align*}
Assosiated with this weak formulation, we have the Galerkin 
approximation on level $l$ given by:
\begin{align*}
&\text{for all $t \in (0,T] $, find $u_l(\cdot, t) \in \mathcal{S}_l$ such that}\\
&\qquad\int_D \dfrac{\partial u_l({\bs x} , t)}{\partial t} v_l({\bs x}) \dd{\bs x}
+\int_D {\bs G}({\bs x}) \nabla u_l({\bs x},t) \nabla v_l({\bs x}) \dd{\bs x}= \int_D I_{\text{app}}({\bs x} , t) v_l({\bs x})  \dd{\bs x}\\
&\hspace*{7cm}\text{for all } v_l \in S_l.
\end{align*}
As each function $u_l(\cdot, t) \in \mathcal{S}_l$ can be expressed 
as a linear combination of the corresponding basis elements, i.e.
\begin{equation}
    u_l({\bs x},t) = \sum_{i = 1}^{n_l} u_{l,i}(t) \phi_{l,i}({\bs x}),
\end{equation}
we can recover the semi-discrete formulation of the problem:
\begin{equation}
  {\bs M}_l  \dfrac{\partial {\bs u}_l(t)}{\partial t} +{\bs K}_l{\bs u}_l(t) = {\bs I}_{\text{app},l}(t), \quad {\bs u}_l(t) = [u_{l,1}(t),\ldots,u_{l,n_l}(t)]^\intercal.
  \label{semi-disc}
\end{equation}
 Here, ${\bs M}_l \in \mathbb{R}^{n_l \times n_l}$ and
 ${\bs K}_l \in \mathbb{R}^{n_l \times n_l}$ are the 
 mass and stiffness matrices on level \(l\) defined as
\[
{\bs M}_l\isdef \left[ \int_D \phi_{l,j}({\bs x}) \phi_{l,i}({\bs x})  d{\bs x} \right]_{i,j=1}^{n_l}, \quad
{\bs K}_l\isdef \left[ \int_D {\bs G}({\bs x})\nabla \phi_{l,j}({\bs x})\nabla \phi_{l,i}({\bs x})  d{\bs x} \right]_{i,j=1}^{n_l},
\]
and the right-hand side ${\bs I}_{\text{app},l}(t)\in\mathbb{R}^{n_l}$ is
\[
  {\bs I}_{\text{app},l}(t)\isdef \left[ \int_D I_{\text{app},l}({\bs x} , t) \phi_{l,i}({\bs x}) \right]_{i=1}^{n_l}.
\]

We next apply the second order Crank--Nicolson method 
for the time discretization of \eqref{semi-disc} and obtain 
for $k=1,\ldots,m_l-1$ the system of equations
\begin{equation}\label{st-system}
\begin{aligned}
 \left({\bs M}_l + \frac{\Delta t_l}{2}{\bs K}_l \right) {\bs u}_{l,k+1}\
 + \left( - {\bs M}_l + \frac{\Delta t_l}{2}{\bs K}_l \right){\bs u}_{l,k} = {\bs I}_{\text{app},l,k} \quad \text{and} \quad {\bs u}_{l,k}\isdef {\bs u}_l(t_{l, k}),\\
  \text{with} \quad  {\bs I}_{\text{app},l,k}\isdef \frac{\Delta t_l}{2} \big(  {\bs I}_{\text{app},l}(t_{l, k+1}) + {\bs I}_{\text{app},l}(t_{l, k})  \big).
\end{aligned}
\end{equation}
 If we define ${\bs A}_{l}\isdef{\bs M}_l + \frac{\Delta t_l}{2}{\bs K}_l $
 and ${\bs B}_{l}\isdef - {\bs M}_l + \frac{\Delta t_l}{2}{\bs K}_l$,
 the system of equations \eqref{st-system} can be summarized in 
 compact form according to
\begin{equation}
\begin{tikzpicture}[baseline=(current bounding box.center)]
\matrix (m) [matrix of math nodes,nodes in empty cells,right delimiter={]},left delimiter={[} ]{
{\bs A}_l  &  &   &    \\
{\bs B}_l  & {\bs A}_l & &  \\
 & & &     \\
  & & {\bs B}_l & {\bs A}_l  \\
} ;
\draw[loosely dotted] (m-2-1)-- (m-4-3);
\draw[loosely dotted] (m-2-2)-- (m-4-4);
\end{tikzpicture}\begin{bmatrix}
    {\bs u}_{l,1}\\{\bs u}_{l,2}\\\vdots\\{\bs u}_{l,m_l}
\end{bmatrix}=
\begin{bmatrix}
    {\bs I}_{\text{app},l,1}\\{\bs I}_{\text{app},l,2}\\\vdots\\{\bs I}_{\text{app},l,m_l}
\end{bmatrix} 
 \quad \Longleftrightarrow  \quad{\bs C}_l {\bs u}_l = {\bs I}_{\text{app},l},
\label{eq:system_C}
\end{equation}
where ${\bs C}_l\in \mathbb{R}^{n_l m_l \times n_l m_l}$ is a 
large space-time system that can be distributed and solved in parallel and 
\[ {\bs u}_l\isdef [{\bs u}_{l,1},{\bs u}_{l,2},\ldots,{\bs u}_{l,m_l}]^\intercal \quad \text{and} \quad {\bs I}_{\text{app},l}\isdef [{\bs I}_{\text{app},l,1},{\bs I}_{\text{app},l,2},\ldots,{\bs I}_{\text{app},l,m_l}]^\intercal.\] 

\subsection{Space--time assembly of the monodomain equation}

The discretization of \eqref{eq:stochastic-monodomain} is an extension of the assembly procedure described in Section~\ref{disc}.  In particular,
the linear system \eqref{eq:system_C} is modified to contain the discretization of the non-linear reaction term $I_{\text{ion}}$ 
\begin{equation}\label{mono_disc}
{\bs C}_l{\bs u}_l + {\bs r}({\bs u}_l) = {\bs I}_{\text{app},l},
\end{equation}
where ${\bs r}({\bs u}_l)\in \mathbb R^{n_l m_l}$ is given by 
$$ 
{\bs r}({\bs u}_l)\isdef(\Delta t_l I_{m_l} \otimes M_l) {\bs I}_{\text{ion}}({\bs u}_l) \quad \text{with } {\bs I}_{\text{ion}}({\bs u}_l)\isdef [I_{\text{ion}}(u_1),\ldots,I_{\text{ion}}(u_{n_l m_l})]^\intercal.$$
Here, $n$ and $m$ are respectively the space and time degrees of freedom. 

The non--linear equation \eqref{mono_disc} is solved by using Newton's method.
The Jacobian ${\bs J}(\mathbf{u}_l)\in\mathbb R^{n_l m_l\times n_l m_l}$ of the 
non-linear operator on the left-hand side of \eqref{mono_disc} is given by
$$  {\bs J}(\mathbf{u}_l)= {\bs C}_l + (\Delta t_l {\bs I}_{m_l} \otimes{\bs M}_l)
\cdot {\bs J}{\bs I}_{\text{ion}}({\bs u}_l)$$ 
with ${\bs J}{\bs I}_{\text{ion}}({\bs u}_l)\in\mathbb R^{n_l m_l\times n_l m_l}$ 
being the block diagonal matrix
$${\bs J}{\bs I}_{\text{ion}}(\mathbf{u}_l)\isdef
\begin{bmatrix}I_{\text{ion}}'(u_1) & & & \\ & I_{\text{ion}}'(u_2) & & \\ & & \ddots & \\ & & & I_{\text{ion}}'(u_{n_l m_l})\end{bmatrix}.$$

\subsection{Solution strategy}
We rely on a solution strategy that retakes the main 
features of the one used in~\cite{benbader2021space}.
Specifically, we approach the parallelization of the solver,
the preconditioning and the Newton initial guess strategy as follows.

\paragraph{Multiple time blocks strategy.} Combination of the Newton's 
method with the space--time all--at--once approach may face a problem 
of convergence when a large time interval is required. We resort to a 
multiple time blocks strategy, see also~\cite{benedusi2020parallel}, 
consisting of decomposing the original time interval into smaller chunks 
(the so--called time blocks). The problem is solved sequentially on each 
of the time blocks, provided that the initial condition of the current time 
block is set to the final state of the previous one. Considering a time 
interval $[0,T]$, uniformly partitioned in $m$ time steps, the latter is 
further divided in $K$ multiple blocks such that there exists $T' \in \RR$ 
and $m' \in \mathbb{N}$ for which we have $T = K T'$ and $m = K m'$.  
We denote the solution on the $k-$th time block with 
\[
    \textbf{u}_{[k]} = [ \textbf{u}_{[k],1} , \ldots, \textbf{u}_{[k],m'}  ]^\intercal\in \mathbb{R}^{n m'} ,
\]
where $n$ is the number of spatial degrees of freedom. The 
initial condition of the current time block to solve translates to,
\[
     \textbf{u}_{[k+1],0} =  \textbf{u}_{[k],m'}\quad \text{for all}\ k = 1, \ldots, K-1,
\]
with respect to that computed from the previous time block. The 
Newton initial guess is further set to be the final time step solution 
of the previous block, generalized to all the time steps, that is 
\[
    \textbf{u}_{[k+1]}^{(0)} = [ \textbf{u}_{[k],m'} , \ldots ,\textbf{u}_{[k],m'} ]^\intercal\in \mathbb{R}^{n m'} .
\]

\paragraph{Preconditioning.} 
Every linear problem in the form of  \eqref{eq:system_C} (arising 
at every Newton iteration) is solved by means of a space-time parallel 
GMRES with block Jacobi preconditioner. The spectral analysis of the 
space-time system in \eqref{eq:system_C} motivates this choice,
see~\cite{Benedusi2018} for all the details.

\paragraph{Newton initial guess.} 
We use a sample based Newton initial guess strategy. For each 
sample, the first Newton iterate is provided by an unperturbed 
reference solution. This is done in two ways, locally, according 
to the multiple time blocks strategy, and globally with respect to 
the original time interval. We compare their performances in 
Figure \ref{fig::Newton-cube} with the direct method and the 
multiple time blocks strategy described above, by solving the 
monodomain equation on six time blocks. The acronyms DM, 
MTB, LNIG and GNIG designate respectively the Direct Method, 
the Multiple Time Block, the Local Newton Initial Guess and the 
Global Newton Initial Guess.

\begin{figure}[H]
\centering
\includegraphics[width=0.49\textwidth]{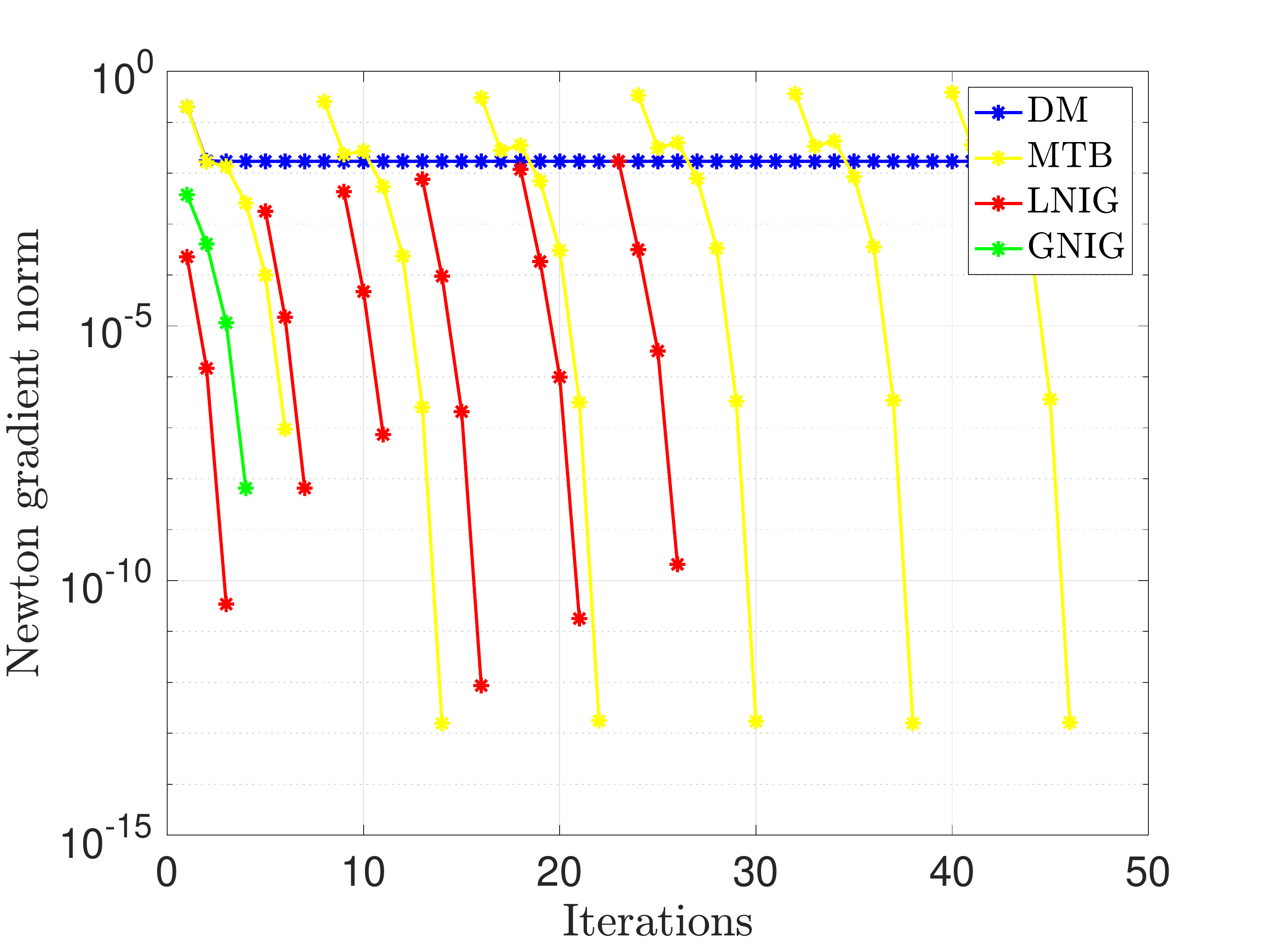}
\captionsetup{justification=centering}
\caption{Comparison of the different Newton initial guess strategies for a sample on cube.}
\label{fig::Newton-cube}
\end{figure} 

\section{Multilevel quadrature methods}  \label{section:methods} 
In order to compute the quantities of interest under consideration, i.e.\
\begin{equation*}
    \operatorname{QoI}[u] = \int_{[-1,1]^M} \mathcal{F}\big(u(\cdot,{\bs\omega})\big) \rho ({\bs\omega}) \dd{\bs\omega},
\end{equation*}
we employ multilevel quadrature methods. To this end, we introduce the
sequence \(\mathcal{F}_l[u]({\bs\omega})\isdef \mathcal{F}\big(u_l(\cdot,{\bs\omega})\big)\) that approximates \(\mathcal{F}[u]({\bs\omega})\isdef \mathcal{F}\big(u(\cdot,{\bs\omega})\big)\)
and, instead of the single level estimator
\begin{equation}\label{eq:SL}
\operatorname{QoI}_L^{\text{SL}}[u]
\isdef \mathcal{Q}_{L} \big(\mathcal{F}_{L}[u](\cdot)\big),
\end{equation}
consider the multilevel estimator
\begin{equation}\label{eq:StandardML}
\operatorname{QoI}_L^{\text{ML}}[u]
\isdef\sum_{l=0}^L
\mathcal{Q}_{L-l}\big(\mathcal{F}_l[u](\cdot)-
\mathcal{F}_{l-1}[u](\cdot)\big),
\end{equation}
where \(\{\mathcal{Q}_{l}\}_{l \geq 0}\) is a sequence of quadrature rules,
and $\mathcal{Q}_{-1} \equiv 0$.
This is the standard and widely used multilevel estimator, 
which has been introduced in \cite{H98,G08,barth2011multi}. 
It consists in defining the multilevel estimator as the sum of
quadratures applied to the difference of finite element solutions. 
The construction of this multilevel estimator has been shown 
to be equivalent to the sparse grid combination technique of 
the finite element space and the stochastic space, compare
\cite{GH12,harbrecht2012multilevel}. In particular, the roles of 
these spaces can be exchanged to present the multilevel estimator 
\eqref{eq:StandardML} in a different way. Namely, it can 
equivalently be written as
\begin{equation}\label{eq:PimpedML}
\operatorname{QoI}_L^{\text{ML}}[u]
\isdef\sum_{l=0}^L
(\mathcal{Q}_{l}-\mathcal{Q}_{l-1})\big(\mathcal{F}_{L-l}[u](\cdot)\big).
\end{equation}
This is especially favourable in case of non-nested meshes. 
In addition, the computational complexity is reduced when 
nested quadrature points are applied, see \cite{GHM20} for 
the details.

For the approximation error of the multilevel quadrature, there holds a
sparse tensor product-like error estimate. If \(\varepsilon_l\to 0\)
is a monotonically decreasing sequence with \(\varepsilon_l\cdot\varepsilon_{L-l}
=\varepsilon_L\) for every \(L\in\mathbb{N}\) and
\[
\|\mathcal{Q}_{L-l}\mathcal{F}[u]-\operatorname{QoI}[u]\|\leq c_1\varepsilon_{L-l}\quad\text{and}\quad
\|\mathcal{F}[u]-\mathcal{F}_l[u]\|\leq c_2\varepsilon_l
\]
for some suitable norms and constants \(c_1,c_2>0\), then
\[
\big\|\operatorname{QoI}_L^{\text{ML}}[u]-\operatorname{QoI}[u]\big\|\leq C L\varepsilon_L
\]
for a constant \(C>0\), provided that $u$ and $u_l$ are sufficiently regular.
We refer to \cite{HPS12} for details on the multilevel quadrature. 
Moreover, we remark that the presented
error estimate is based on error equilibration. It is, however, 
also possible to equilibrate the computational work or the degrees of
freedom, see \cite{gilesmultilevel,griebel2013note}.

An important component for the multilevel quadrature is the
intergrid transfer of the data. The transfer of data from 
the coarse level to the fine level is not needed if
only real-valued quantities of interest, based on point evaluations,
are considered. Otherwise, the estimator \eqref{eq:PimpedML} can
efficiently be computed
by transferring the data only after accumulating it for each level.
The transfer of data from 
the fine level to the coarse level however is mandatory for each sample,
as the Karhunen-Lo\`eve expansion has to be computed on 
the fine grid. 

We remark that has been shown in \cite{siebenmorgen2015quadrature}
that the computation 
of the stiffness matrix with respect to the random diffusion field is
consistent with a piecewise linear finite element discretization
if the midpoint rule with respect to the current grid is applied. 
Therefore, for the transfer of the random fields from the finest 
level to the coarser levels, we perform an element--wise transfer 
based on the midpoint rule; that is, for every element of the coarse 
level, we assign an constant diffusion value corresponding to the 
fine element containing its center. Therefore, the assembly of the
stiffness matrix on the coarser levels can be performed with linear 
cost relative to the particular level of discretization.


\section{Numerical experiments}\label{sct:numerix}
\subsection{Setup}
The numerical experiments have been conducted for three
test-case geometries: a cube, an idealized ventricle and
a heart geometry (atria excluded) that was acquired through real
patient CCT data. We will refer to the latter geometry with
the term ``realistic heart''. The simulations have been realized
using \texttt{SLOTH}, see \cite{quaglino2017sloth}, a UQ 
Python library developed at the Institute of Computational 
Science (ICS) in Lugano. For this work, we extended it to 
the monodomain equation (and in general to all types of 
$3+1$ dimensional PDEs) by employing \texttt{Utopia}, 
see \cite{Zulian2017d}, for the finite element formulation. 

\paragraph{Parameters for the monodomain equation.} 
Regarding the models \eqref{monodomain} and \eqref{ionic_term}, 
we will always rely on the following parameters:    
\begin{itemize}
    \item The values for the ionic channel model $I_\text{ion}(u)$ in \eqref{ionic_term}
    are set as $\alpha = 1.4\cdot10^{-3}$ mV$^{-2}$ms$^{-1}$, 
    $u_\text{rest}=0$ mV, $u_\text{th}=28$ mV, and $u_\text{peak} = 115$ mV. 
    \item We choose
    \[
    I_{\text{app}}({\bs x},t) = \left( u_\text{rest} + u_\text{peak} \exp{ \left( - \dfrac{ ( {\bs x} - {\bs x}_0 )^2 }{ \sigma^2 } \right)} \right) \chi_{[0,t_1)}(t),
    \]
    where $t_1 = \Delta t = 0.005$ ms is the function we rely on for the 
    applied stimulus. Parameters $\sigma$ and ${\bs x}_0$ represent 
    respectively the power and the location of the stimulus. They are 
    geometry dependent.
\end{itemize}

\paragraph{Parameters for Karhunen-Lo\`eve expansion.} 
For the numerical experiments,
we introduce a scaling factor $\theta\in\RR$ into the covariance kernel
$\Cov[{\bs V}]$ to be able to easily scale the applied perturbation size.
The stochastic dimension $M$ in the computed, parametric Karhunen-Lo\`eve expansion
\eqref{eq:KLlowrank}
arises from prescribing the truncation error $\epsilon = 10^{-2}$ in the
pivoted Cholesky decomposition. The other parameters used in 
the Karhunen-Lo\`eve expansion differ from one experiment to 
another and are listed below:
\begin{itemize}
\item \textbf{Cube.} 
We consider isotropic diffusion and a scaling factor $\theta=0.3$ on 
\(D=[-0.5,0.5]^3\). 
The covariance matrix is induced by the scalar covariance kernel  
$\Cov[{\bs V}]({\bs x},{\bs x}') = \theta^2 e^{\frac{-\|{\bs x}-{\bs x}'\|_2^2}{\sigma_{\text{KL}}}}$
with $\sigma_{\text{KL}} =0.25$. The low-rank 
Cholesky approximation of the covariance matrix yielded the stochastic 
dimension $M = 66$. The mean diffusion is set to
$\Ebb[{\bs V}]({\bs x}) = 3.325\cdot 10^{-3}$mm$^2$ ms$^{-1}$ for all ${\bs x} \in D$. 
\item \textbf{Idealized ventricle.}  We consider isotropic diffusion 
and the scaling factor $\theta=0.3$. The covariance kernel is given 
by $ \Cov[{\bs V}]({\bs x},{\bs x}') = \theta^2 e^{\frac{-\|{\bs x}-{\bs x}'\|_2^2}
{\sigma_{\text{KL}}}}$ with $\sigma_{\text{KL}} =0.5$. The stochastic 
dimension is given by $M = 87$. The mean diffusion is set to 
$\Ebb[{\bs V}]({\bs x}) = 3.325\cdot 10^{-3}$mm$^2$ ms$^{-1}$
for all ${\bs x} \in D$. The domain's bounding box is given by
\([0.80,2.8]\times[1.01,3.5]\times[0.60,2.6]\).
\item \textbf{Heart geometry.} We consider anisotropic diffusion
with a block-diagonal covariance matrix given by 
$\Cov_{i,j}[{\bs V}]({\bs x},{\bs x}') = \delta_{i,j} \theta^2 e^{\frac{-\|{\bs x}-{\bs x}'\|_2^2} {\sigma_{\text{KL}}}}$
for $1 \leq i,j \leq 3$, where $\delta_{i,j}$ is the 
Kronecker delta. We set $\sigma_{\text{KL}} =0.16$. The scaling 
factor is set to $\theta=0.3$. The stochastic dimension is $M=135$. 
Furthermore, the perpendicular diffusion in \eqref{eq:AdmV} is 
chosen as $g=1.625\cdot 10^{-3}$mm$^2$ ms$^{-1}$
while $\Ebb[{\bs V}]({\bs x})$ is specified later.
The domain's bounding box is given by
\([-0.41,0.66]\times[0.21,1.03]\times[0.56,1.61]\).
\end{itemize}

\paragraph{Reference solution and error metrics.} 
In all the convergence and work comparison graphs that 
follow, the referenced root mean square error
in the $H^q$ norm ($L^2=H^0$ and $H^1$ for 
respectively $q=0$ and $q=1$) for
$\operatorname{QoI}_l = \operatorname{QoI}_l^{\text{SL}}$
or $\operatorname{QoI}_l = \operatorname{QoI}_l^{\text{ML}}$
is given by 
\begin{equation}\label{eq:space-time-error}
e_l = \Big(\mathbb{E}\Big[\big\|\operatorname{QoI}_l[u] - \operatorname{QoI}_{\text{ref}}[u]\big\|_{L^2((0,T);H^q(D))}^2\Big]\Big)^{1/2}
\end{equation}
in case of a space--time quantity of interest,
such as the transmembrane potential,
\begin{equation}\label{eq:point-error}
e_l = \Big(\mathbb{E}\Big[\big\|\operatorname{QoI}_l[u] - \operatorname{QoI}_{\text{ref}}[u]\big\|_{H^q(0,T)}^2\Big]\Big)^{1/2}
\end{equation}
in case of a time quantity of interest,
such as the action potential at a given location ${\bs x}_0$ in space,
and
\begin{equation}\label{eq:activation-time-error}
e_l = \Big(\mathbb{E}\Big[\big|\operatorname{QoI}_l[u] - \operatorname{QoI}_{\text{ref}}[u]\big|^2\Big]\Big)^{1/2}
\end{equation}
in the case of a scalar quantity of interest,
such as the activation time at a given point ${\bs x}_0$ in space.
Note that the mean in the above expressions is taken over the realisations
of the possibly non-deterministic quadrature formulas.
Specifically, for the Monte Carlo quadrature, the expectation for 
both, the single-level and multilevel runs, are approximated 
by averaging over 10 simulations at each level of precision 
for the nested case study, and 5 simulations for the 
non--nested example.
The reference quantity of interest $\operatorname{QoI}_{\text{ref}}[u]$
is computed by using $N=\text{10'000}$ samples drawn from the Halton sequence. 

The intergrid transfer of a space--time quantity of interest
from a given coarse level $l$ to the fine level $L$, required 
to evaluate the error \eqref{eq:space-time-error}, is 
performed by means of the tensor product of the space and 
time interpolation matrices.
Obviously, the intergrid transfer of a time quantity of interest
from a given coarse level $l$ to the fine level $L$, required 
to evaluate the error \eqref{eq:point-error}, is 
performed analogously by means of the time interpolation matrices.

\paragraph{Quadrature methods.} 
In our experiments, we will consider the Monte Carlo
(MC) and quasi-Monte Carlo (QMC) quadrature method and
their multilevel pendants MLMC and MLQMC. Let us recall that 
the error is of order $2^{-2l}$ in the $L^2$ norm ($2^{-l}$ in the $H^1$ norm) when using 
linear finite elements of mesh size $h_l = 2^{-l}$. Therefore, in view
of the convergence rates for MC and QMC, the number of samples 
to be executed by these methods on a level $l$ to get the
same order of error $2^{-2l}$ is respectively given by 
\begin{equation}\label{eq:SampleNumbers}
N_{\text{MC},l} = 2^{4l}\quad\text{and}
\quad N_{\text{QMC},l} = 2^{2l}.
\end{equation}
Regarding the $H^1$--error, the number of samples 
to be executed on a level $l$ to get the
same order of error $2^{-l}$ is respectively given by
\begin{equation}\label{eq:SampleNumbersH1}
N_{\text{MC},l} = 2^{2l}\quad\text{and}
\quad N_{\text{QMC},l} = 2^{l}.
\end{equation}

\subsection{Scalar random diffusion for simple geometries and nested meshes}
In these first experiments, we consider a scalar, thus isotropic, 
random diffusion for the sake of simplicity. The experiments 
are conducted on the cube and the idealized ventricle geometry.
Note that the intergrid mesh transfer in case of the space--time 
dependent quantity of interest is straightforward as the meshes 
are nested.

\subsubsection{Cube geometry} 
We use a hierarchy of $L=6$ nested mesh 
levels. Starting from the finest level $l=L-1$, the coarser levels 
$l=0,1,\dots,L-2$ are successively obtained from the prior finer 
levels $l+1=1,2,\dots,L-1$ by uniformly coarsening in space and time.
The number of space-time degrees of freedom (DOF), the space 
and time discretization steps of all the different levels are 
reported in Table \ref{table:cube-levels}.

\begin{table}[h!]
\centering
\begin{tabular}{|c||l|l|l|l|l|l|} 
\hline $l$ & 0 &  1 & 2 & 3 & 4 & 5  \\
\hline
\hline DOF  & 16 & 256 & 4'096 & 65'536 & 1'048'576 & 16'777'216   \\
\hline
\hline $ h $  & 0.5 & 0.25 & 0.125 & 0.0625 & 0.03125 & 0.015625   \\
\hline $ \Delta t $  & 0.16 & 0.08 & 0.04 & 0.02 & 0.01 & 0.005   \\
\hline 
\end{tabular}
\caption{Details about the considered mesh hierarchy for the cube geometry.}
\label{table:cube-levels}
\end{table}

\paragraph{Controlled convergence of the over-all error.}
We intend to estimate and verify the convergence rate for 
the quadrature methods under consideration. The number of 
samples on each level is determined by the sampling strategy 
for controlling the error, cf.\ Section \ref{section:methods},
by using the sample numbers \eqref{eq:SampleNumbers} 
and \eqref{eq:SampleNumbersH1}. We  report in Figure 
\ref{fig::controlled-convergence-cube} the convergence of 
the error in $L^2$ and $H^1$ norms. 

The plots show that we recover the expected convergence 
rates of the general error for all quadrature methods tested. Note that this
does not imply that these quadrature methods are all equally efficient,
but rather that they yield the same precision with a vastly different balancing 
of samples on every level. This is demonstrated quite clearly in the corresponding
work comparison plot found in Figure \ref{fig::asymptotical-work-cube}.

\begin{figure}[H] 
\centering
\includegraphics[width=0.45\textwidth]{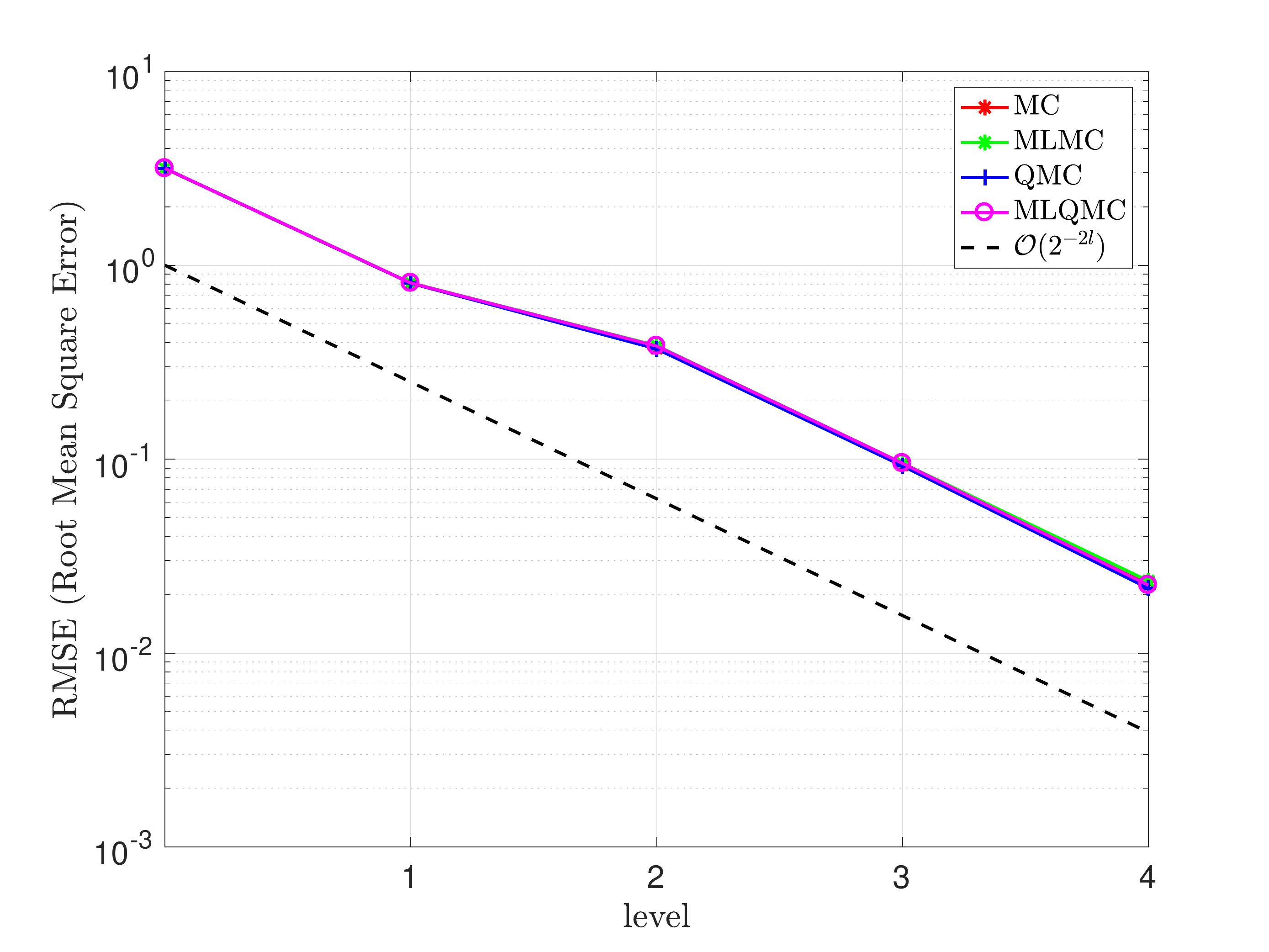}
\includegraphics[width=0.45\textwidth]{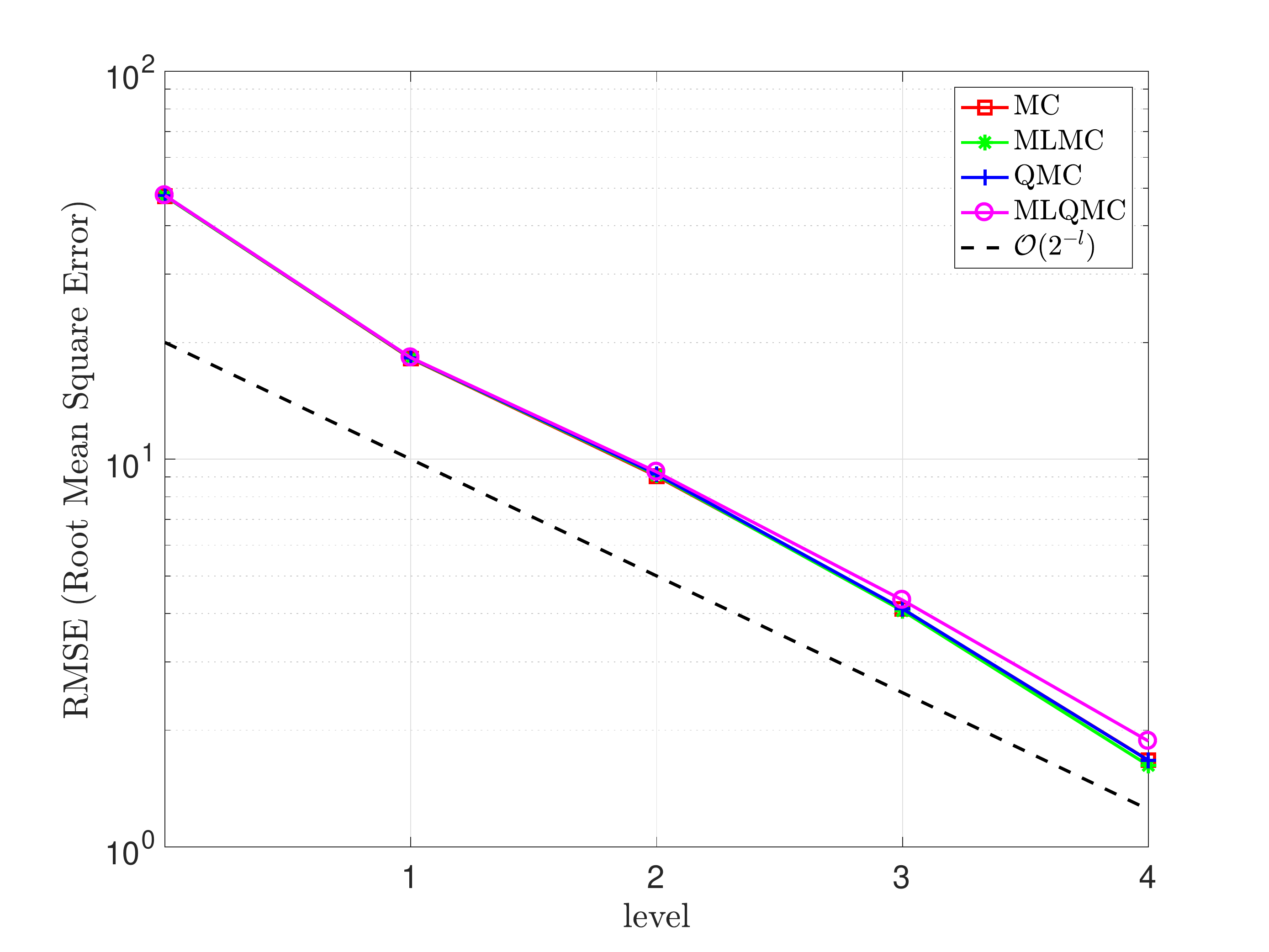}
\caption{Convergence rate for the cube in $L^2$ (left) and $H^1$ (right) norms.}
\label{fig::controlled-convergence-cube}
\end{figure}

\paragraph{Asymptotical work behaviour.}
Equivalently to the controlled convergence concept with the previously 
introduced sampling strategy, we would like now to study the work in 
the context of the controlled error. To this end, we assume that the
cost per solve on level $l$ is given by $C_{\text{FE},l} = 2^{\gamma d l}$, 
where $\gamma$ is the complexity of the finite element solver used 
and $d$ is the dimension of the physical problem considered (here $d=4$). 
In view of \eqref{eq:SampleNumbers}, we can hence recover the 
total amount of work required by MC and QMC given a 
discretization level $L$ with \eqref{eq:SL} by
\[
W_{\text{MC},L} = C_{\text{FE},L} N_{\text{MC},L}
= 2^{\gamma d L} 2^{4L} = 2^{ (\gamma d + 4) L }
\]
and 
\[
W_{\text{QMC},L} = C_{\text{FE},L} N_{\text{QMC},L}
= 2^{\gamma d L} 2^{2L} = 2^{ (\gamma d + 2) L } . 
\]

The total work for the multilevel pendants with 
$L$ discretization levels can also be deduced from 
\eqref{eq:SampleNumbers} with \eqref{eq:StandardML}, and we can write 
\[
    W_{\text{MLMC},L} = \sum_{l=1}^{L} C_{\text{FE},L} N_{\text{MC},L-l} 
    = \sum_{l=1}^{L} 2^{\gamma d l} 2^{4(L-l)} =  2^{4L} \sum_{l=1}^{L} 2^{ (\gamma d - 4) l}
\]
and 
\[
    W_{\text{MLQMC},L} = \sum_{l=1}^{L} C_{\text{FE},L} N_{\text{QMC},L-l} 
    = \sum_{l=1}^{L} 2^{\gamma d l} 2^{2(L-l)} =  2^{2L} \sum_{l=1}^{L} 2^{ (\gamma d - 2) l}.
\]
These can further be reformulated as
\[
     W_{\text{MLMC},L} = 
    \begin{cases}
      L 2^{4L}  \text{  if  }  \gamma d = 4 , \\
      \dfrac{ 2^{\gamma d L} - 2^{4L}  }{2^{\gamma d - 4} - 1} \text{if  } \gamma d \neq  4 ,
    \end{cases}
\]
and
\[
     W_{\text{MLQMC},L} = 
    \begin{cases}
      L 2^{2L}  \text{  if  }  \gamma d = 2 , \\
      \dfrac{ 2^{\gamma d L} - 2^{2L}  }{2^{\gamma d - 2} - 1} \text{if  } \gamma d \neq  2 .
    \end{cases}
\]
Therefore, the asymptotical work behaviour for MLMC is bounded by 
$\mathcal{O}(2^{4L})$ when $\gamma d < 4$ (with an additional log-factor 
if $\gamma d = 4$), and by $ \mathcal{O}(2^{\gamma d L }) $ if $ \gamma d > 4$. 
Likewise, for MLQMC, the asymptotical work is bounded by $\mathcal{O}(2^{2L})$ 
when $\gamma d < 2$ (with an additional log-factor if $\gamma d = 2$), and 
by $ \mathcal{O}(2^{\gamma d L }) $ if $ \gamma d > 2$.

The complexity parameter $\gamma$ is therefore of major importance 
in the asymptotical work behaviour of the considered quadrature methods. 
In our case, given the Newton initial guess strategy, the solver preconditioning 
and the difference in parallel resources used from one level to another do not 
allow to give this parameter a concise value over all levels (compare Section 
\ref{sec:FEM}). We suggest however to evaluate the work in terms of total 
execution time, in which the cost $C_{\text{FE},l}$ for solving a sample at level 
$l$ is given by the time to solution (averaged over 100 samples). The 
resulting plot of work comparison between the different methods is 
reported in Figure \ref{fig::asymptotical-work-cube}.
On the on hand it is clearly visible that both MLMC and MLQMC show a significantly
improved asymptotic efficiency compared to their single level counterpart.
On the other hand comparing MLMC with MLQMC shows that MLQMC seems to be able to
use the higher convergence order of QMC versus MC to achieve
an improved asymptotic efficiency over MLMC.

\begin{figure}[H] 
\centering
\includegraphics[width=0.6\textwidth]{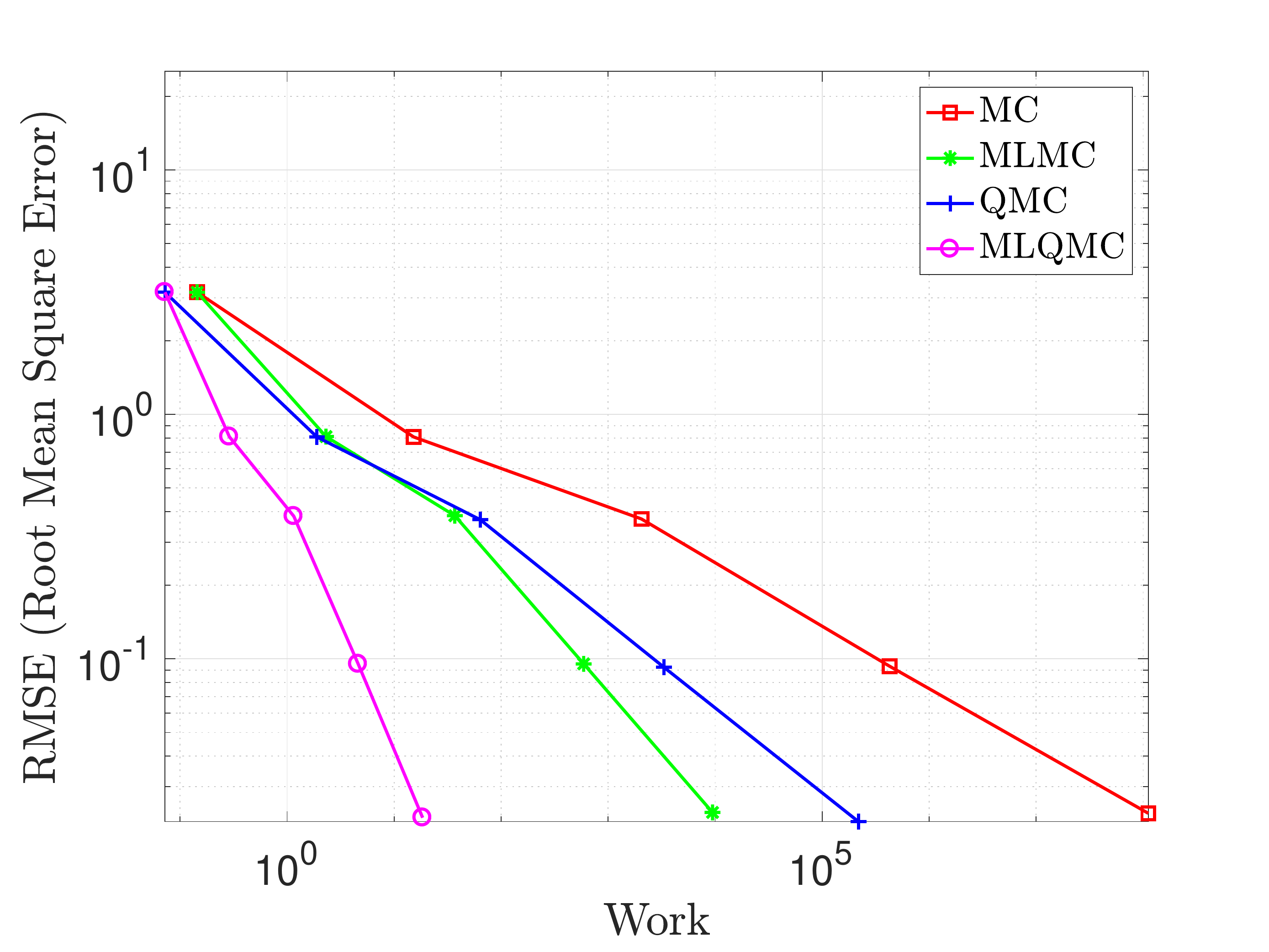}
\captionsetup{justification=centering}
\caption{Work comparison for MC/MLMC/QMC/MLQMC. 
Work is computed on the basis of execution time on a single thread for all levels.}
\label{fig::asymptotical-work-cube}
\end{figure} 

\subsubsection{Idealized ventricle}
In the second test case, we rely on a mesh hierarchy of $L=3$ 
levels. The main reason for the limitation of levels number for this 
geometry is essentially due to the nestedness condition. Indeed, 
for this geometry as opposed to the simple cube one, we proceed 
in an inverted way, i.e.\ refining a given initial mesh. This procedure 
becomes demanding at the memory level very quickly since the 
refinement step increases the degrees of freedom by the factor 
$2^4 = 16$ due to the space--time discretization. In general, 
this limitation can very often be encountered when dealing with 
nested meshes for realistic geometries. This is the main motivation 
for relying on non-nested meshes for the last test case,
see Subsection \ref{subsection:anisotropic}.

The number of space-time degrees of freedom (DOF), the space 
and time discretization steps of the three different levels are 
reported in Table \ref{table:ventricle-levels}. The meshes 
are visualized in Figure \ref{fig::mesh-hierarchy-ventricle}.
The rates of convergence and the work of the different quadrature
methods are found in Figure \ref{fig::controlled-convergence-ventricle}.
As we have only three levels, the meaningfulness of the results is 
limited as it is impossible to conclude the asymptotic behaviour. 
Nevertheless, it is clearly seen that MLQMC is superior over the 
other methods.

\begin{table}[h!]
\centering
\begin{tabular}{|c||l|l|l|} 
\hline $l$ & 0 &  1 & 2   \\
\hline
\hline DOF  & 154'546 & 2'120'420 & 31'184'747   \\
\hline 
\hline 
$h$  & 0.1 & 0.05 & 0.025   \\
\hline
\hline
$\Delta t$  & 0.02 & 0.01 & 0.005   \\
\hline
\end{tabular}
\caption{Details about the mesh hierarchy for the idealized ventricle geometry.}
\label{table:ventricle-levels}
\end{table}

\begin{figure}[H] 
\centering
\includegraphics[width=0.3\textwidth]{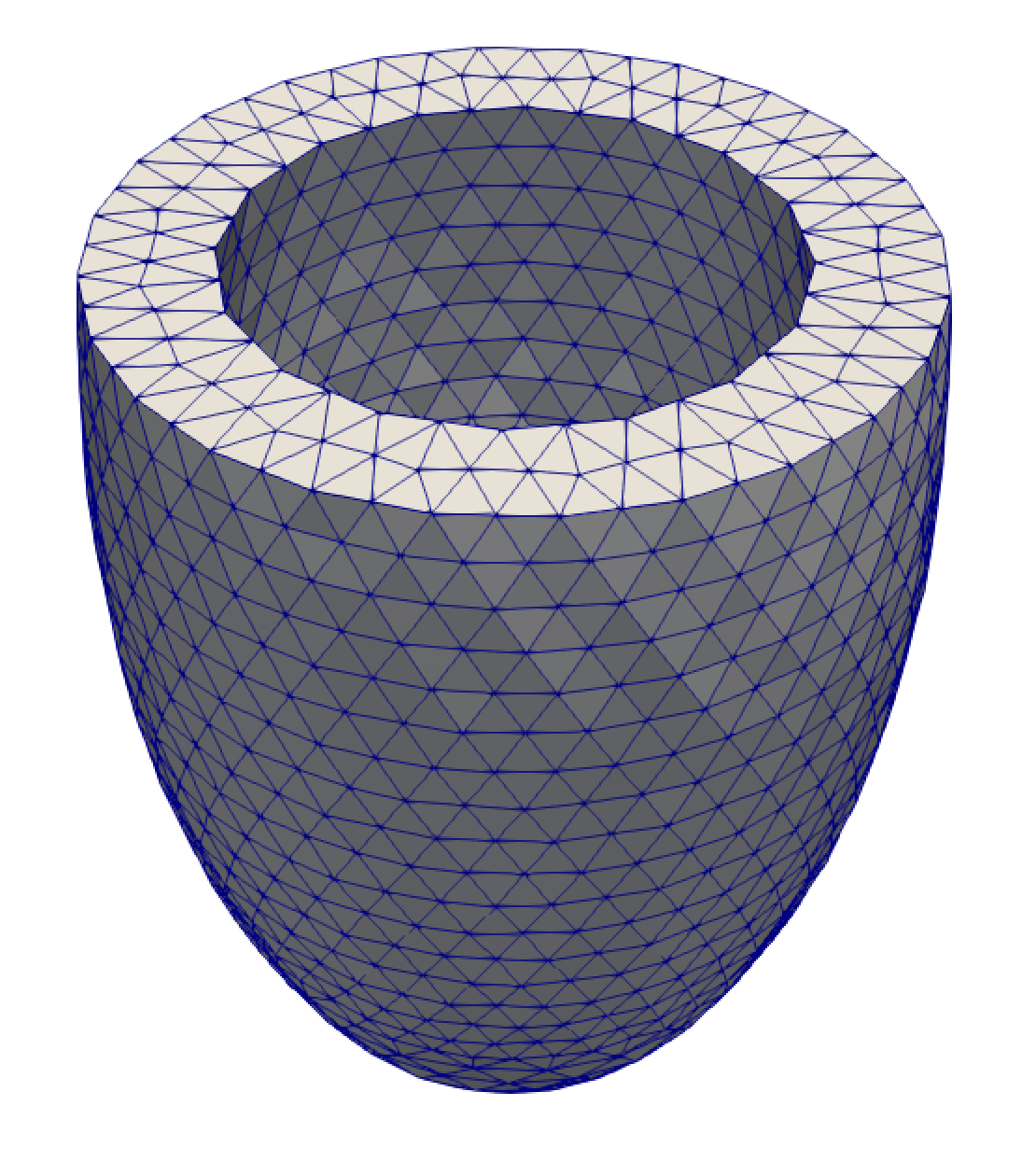}
\includegraphics[width=0.3\textwidth]{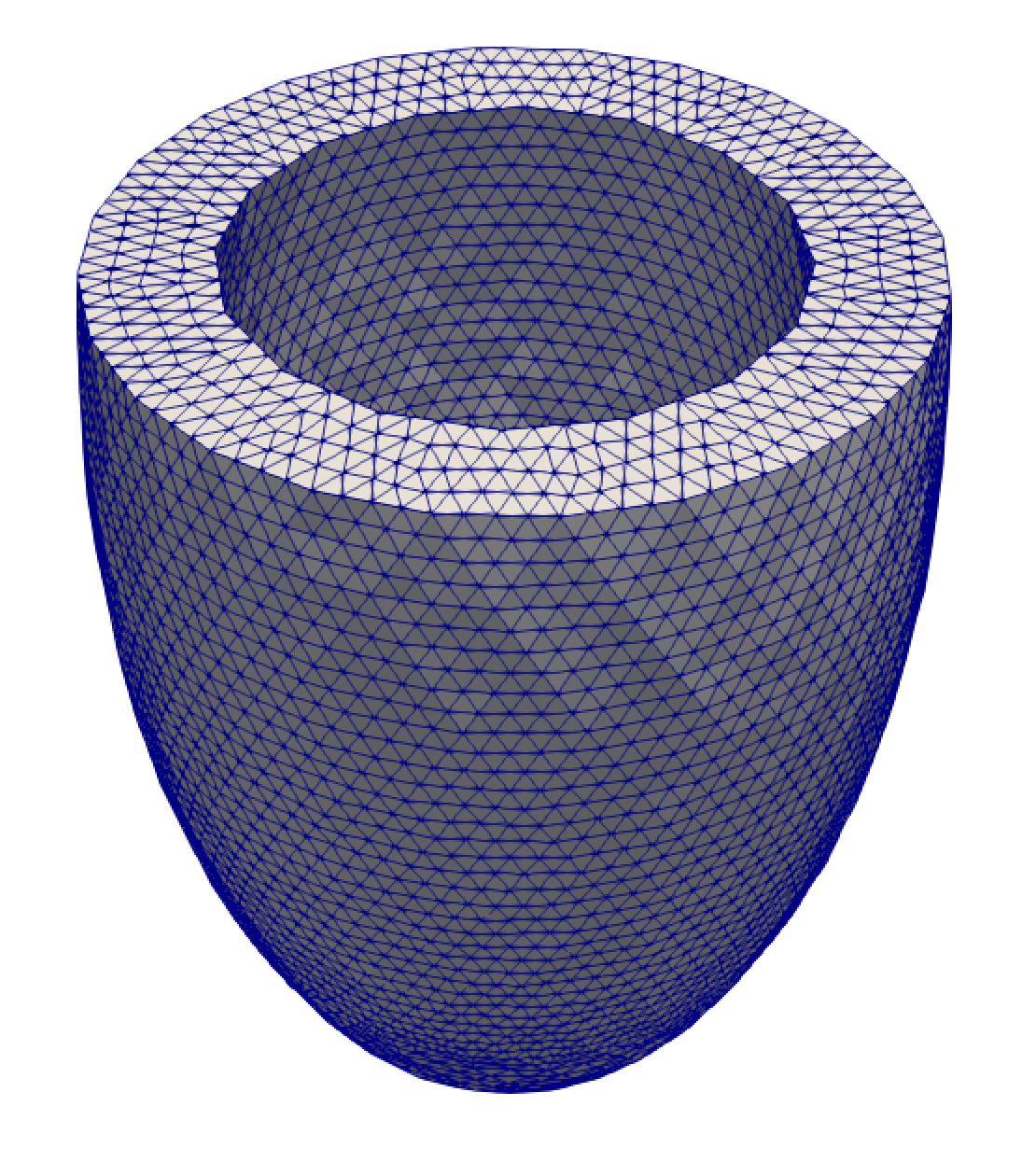}
\includegraphics[width=0.3\textwidth]{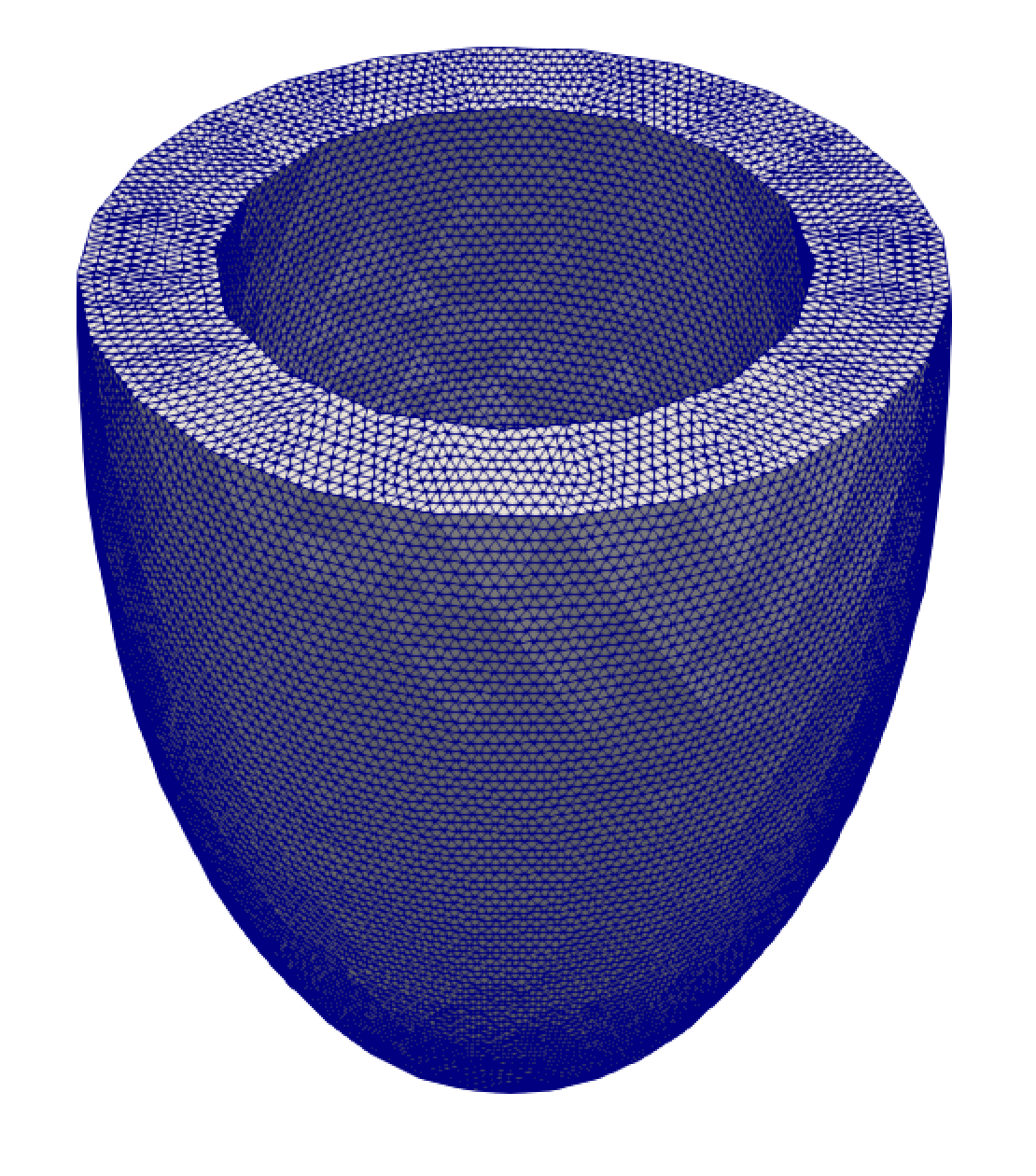}
\caption{Nested mesh hierarchy for the idealized ventricle.}
\label{fig::mesh-hierarchy-ventricle}
\end{figure} 

\begin{figure}[H] 
\centering
\includegraphics[width=0.45\textwidth]{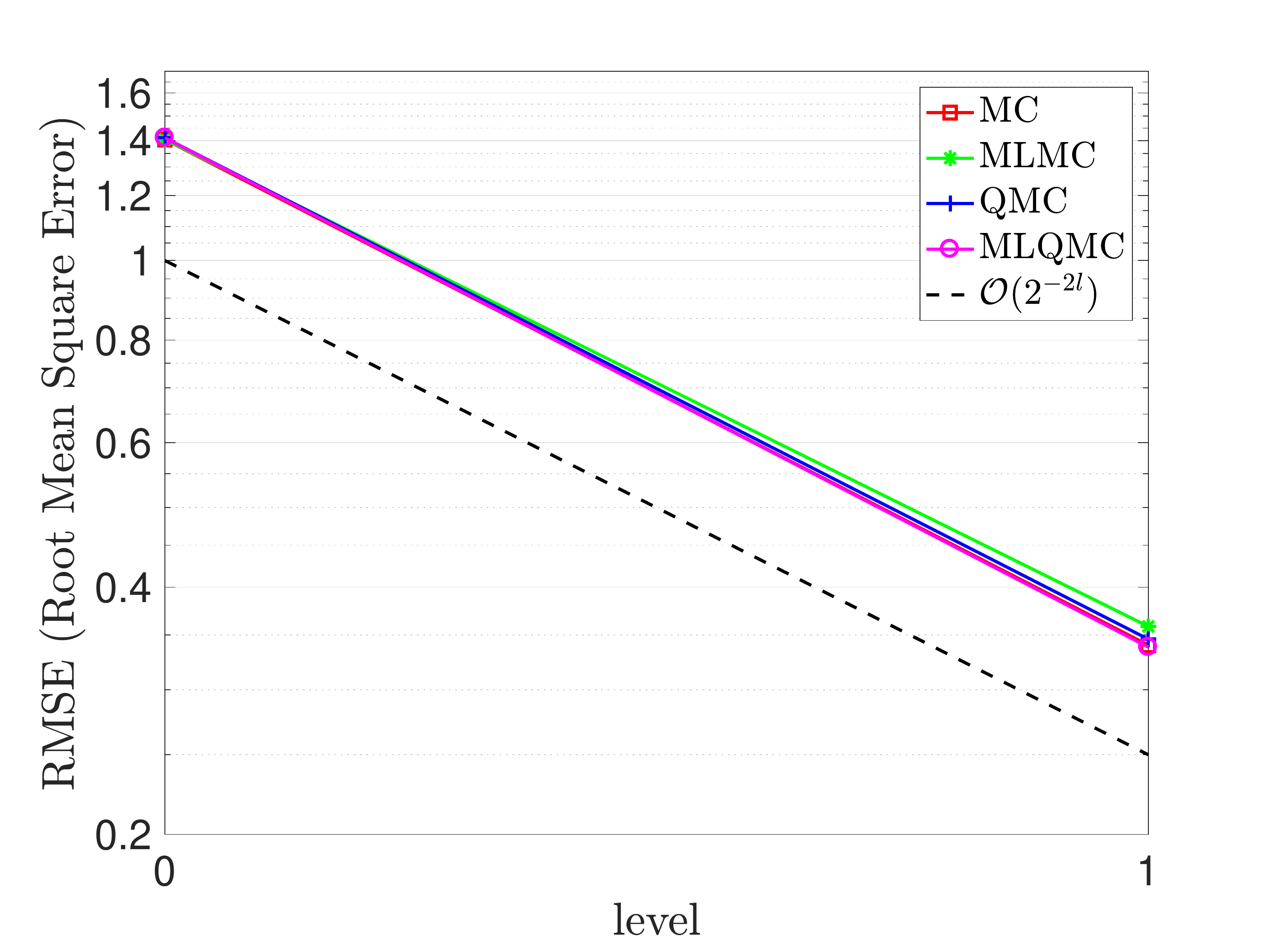}
\includegraphics[width=0.45\textwidth]{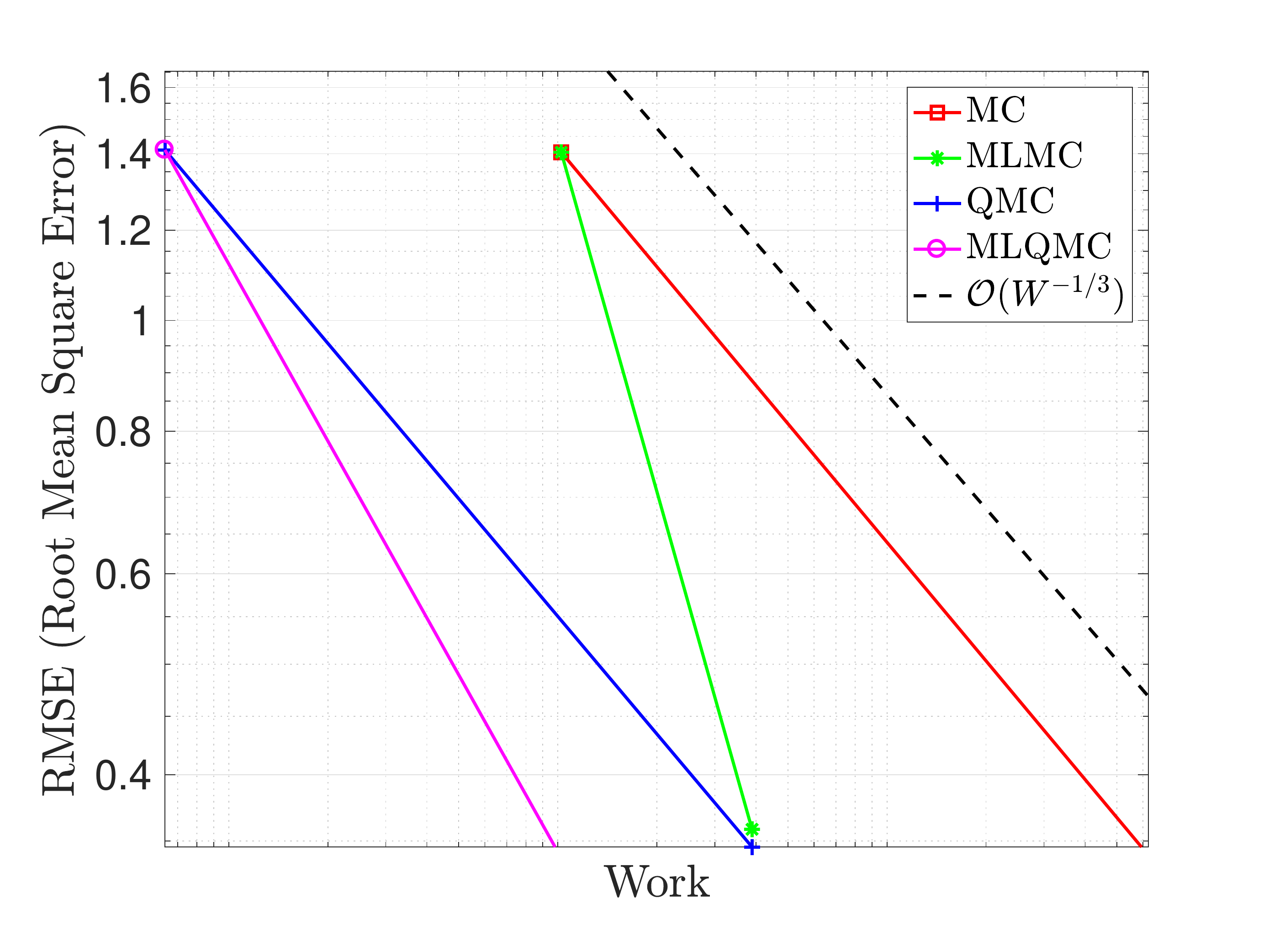}
\captionsetup{justification=centering}
\caption{Convergence rate for the idealized ventricle in the $L^2$ norm (left)
and work comparison (right).}
\label{fig::controlled-convergence-ventricle}
\end{figure} 

\subsection{Random fibers in a complex geometry with non-nested meshes} 
\label{subsection:anisotropic}
The last test case concerns a realistic heart geometry with 
data acquired from clinical measurements. As this is meant to be 
the synthesis of this work, we also account for anisotropic 
diffusion defined in \eqref{eq:AdmV}. The associated expected 
fiber field $\mathbb{E}[{\bs V}] ({\bs x})$ is shown in Figure 
\ref{fig::initial-state-fibers}. It is obtained from a mathematical 
reconstruction using transmural coordinates~\cite{potse2006comparison}. 
The transmural coordinates are derived by initially solving a diffusion 
problem with adapted boundary conditions at the contour of the left 
and right ventricles~\cite{bayer2012novel}.

\begin{figure}[htb]
\centering
\includegraphics[width=0.45\textwidth]{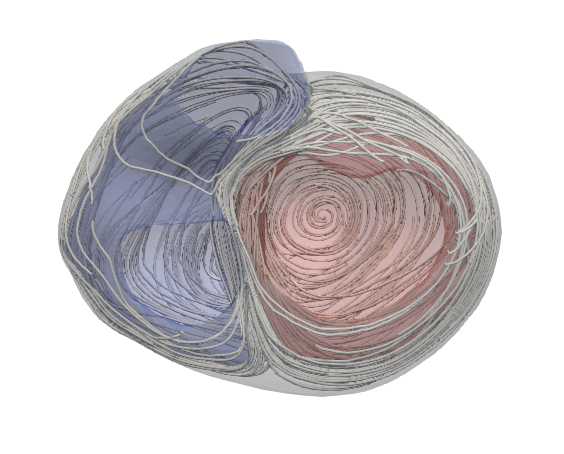}
\includegraphics[width=0.45\textwidth]{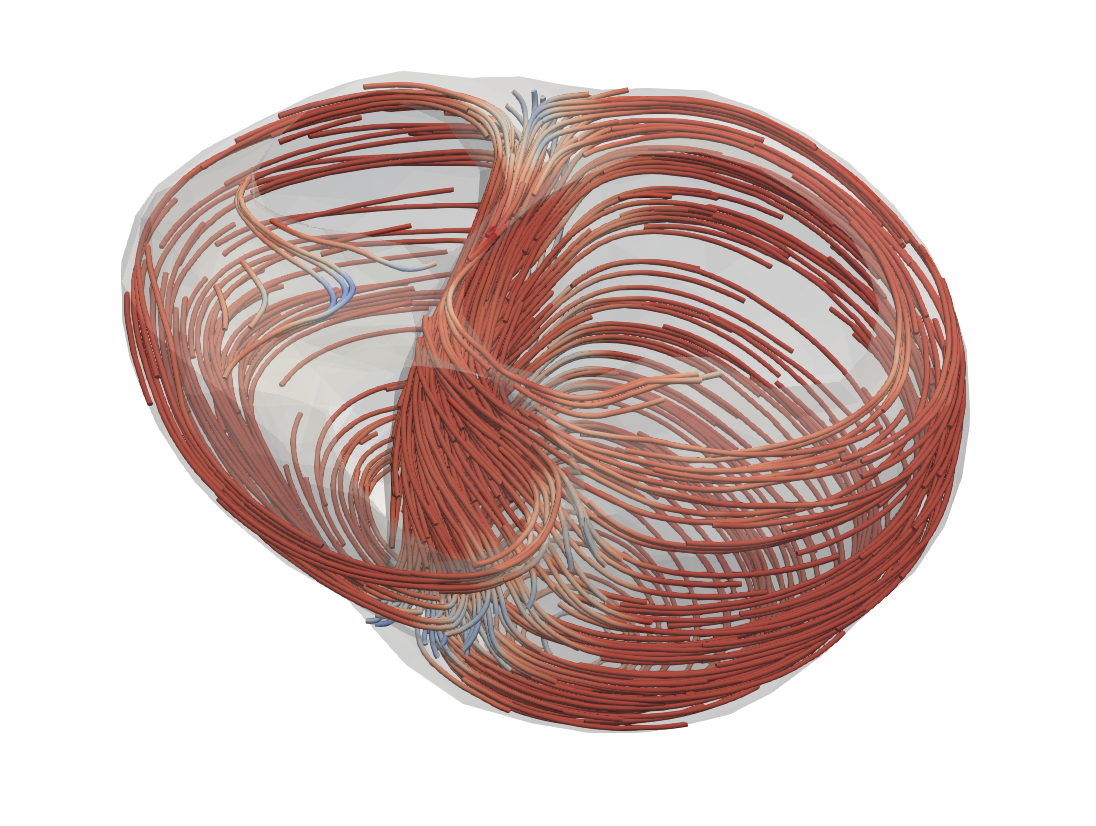}
\caption{Initial state for fibers $\mathbb{E}[{\bs V}] ({\bs x})$.}
\label{fig::initial-state-fibers}
\end{figure} 
We also relax the nestedness condition by considering a hierarchy of non-nested 
meshes. As we have previously argued, the nestedness condition 
very quickly becomes a burden in considering a large number of levels. 
We rely on a mesh hierarchy with 6 levels in this example. They are shown in Figure 
\ref{fig::mesh-hierarchy-heart}. The details on the space-time DOF 
and discretization steps are reported in Table \ref{table:heart-levels}.

\begin{figure}[H] 
\centering
\includegraphics[width=0.3\textwidth]{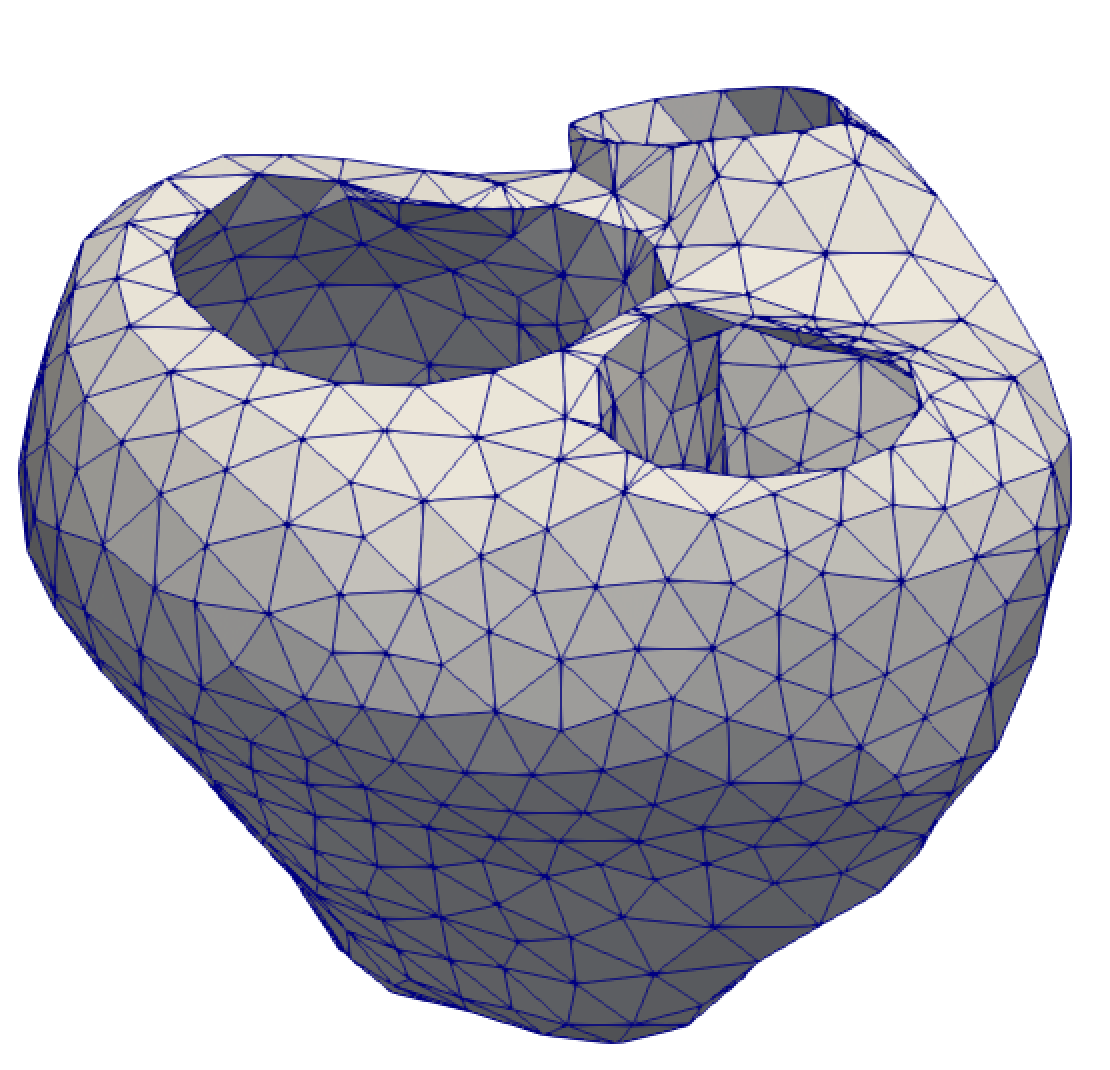}
\includegraphics[width=0.3\textwidth]{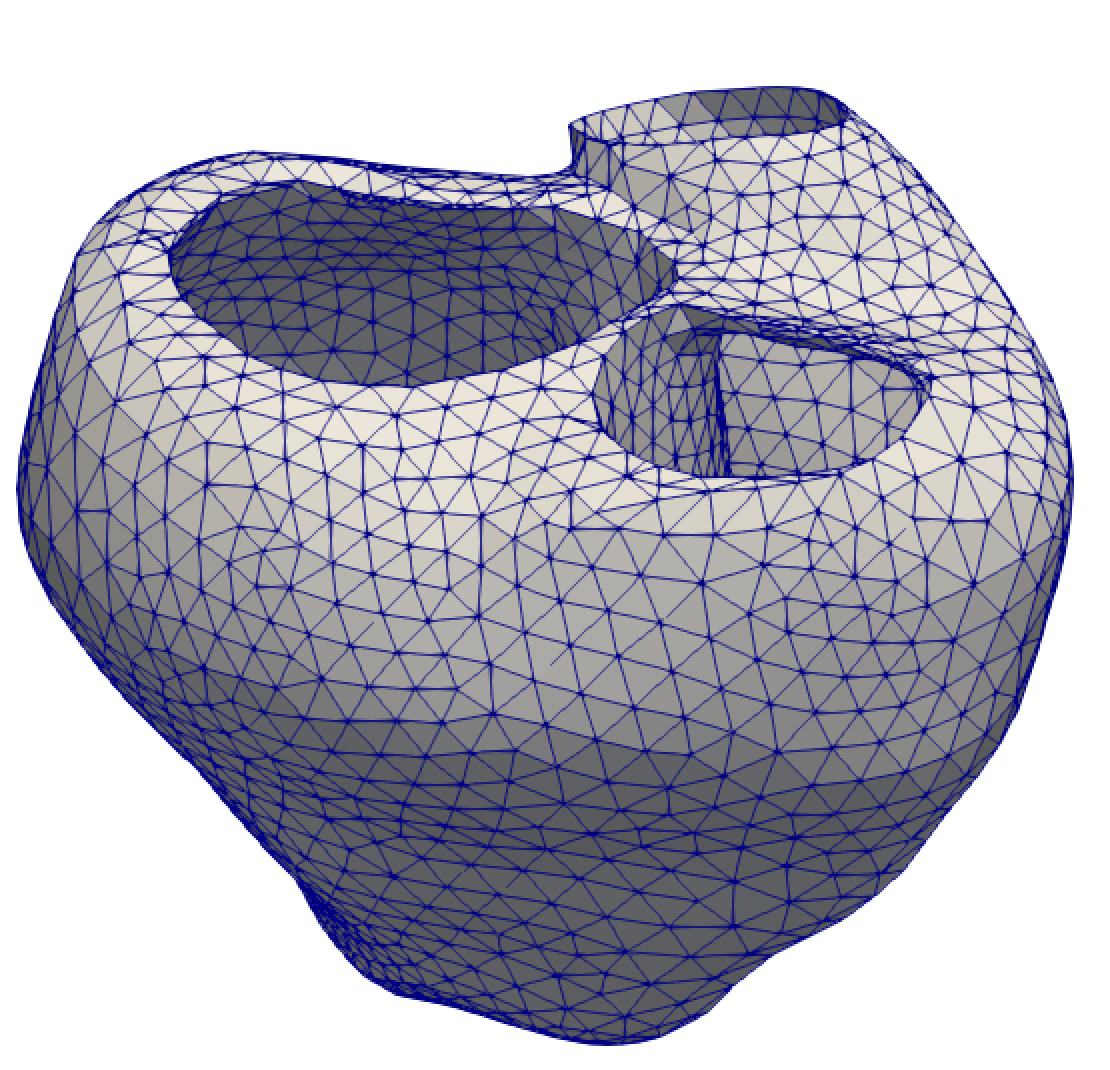}
\includegraphics[width=0.3\textwidth]{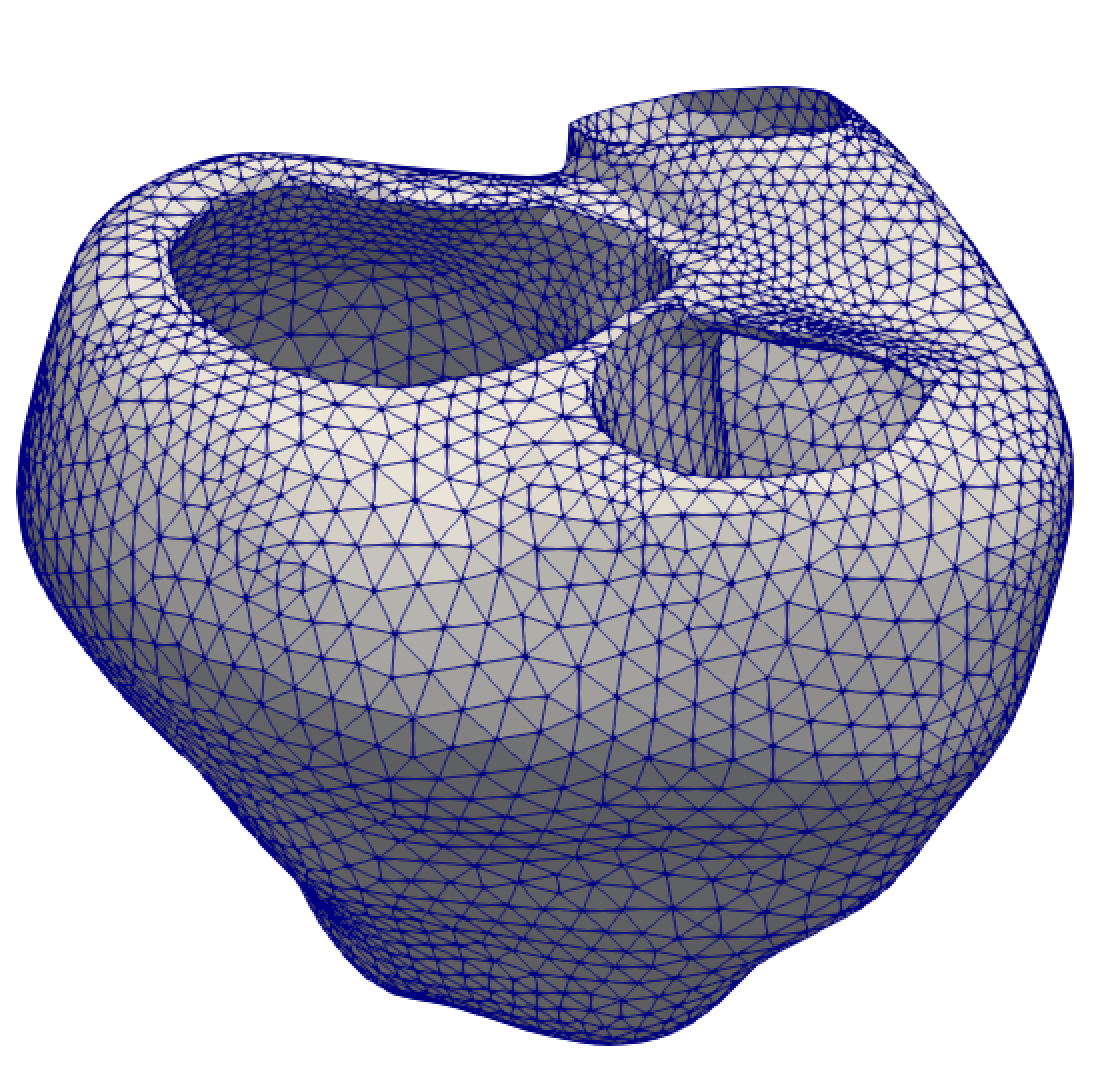}
\includegraphics[width=0.3\textwidth]{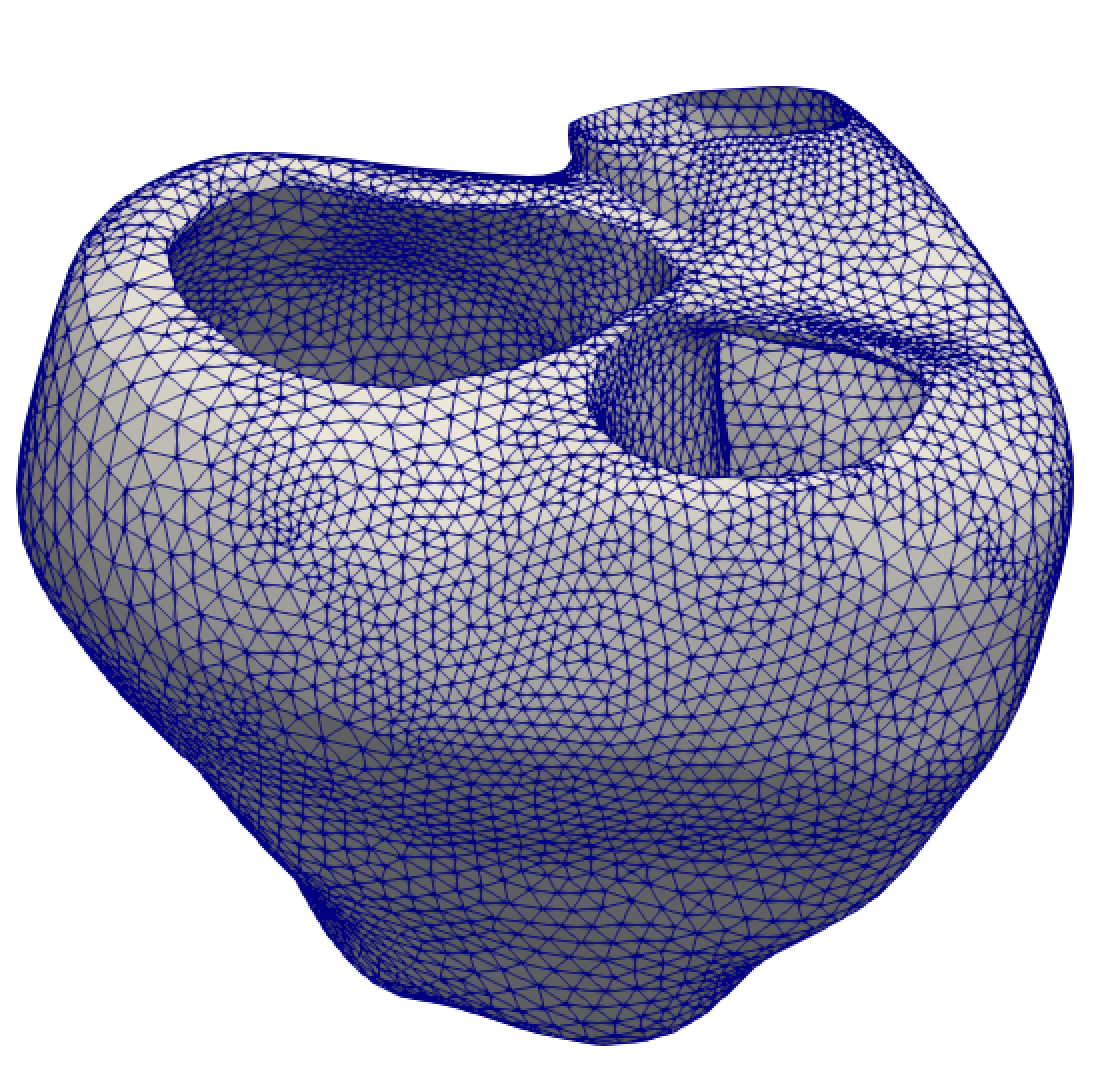}
\includegraphics[width=0.3\textwidth]{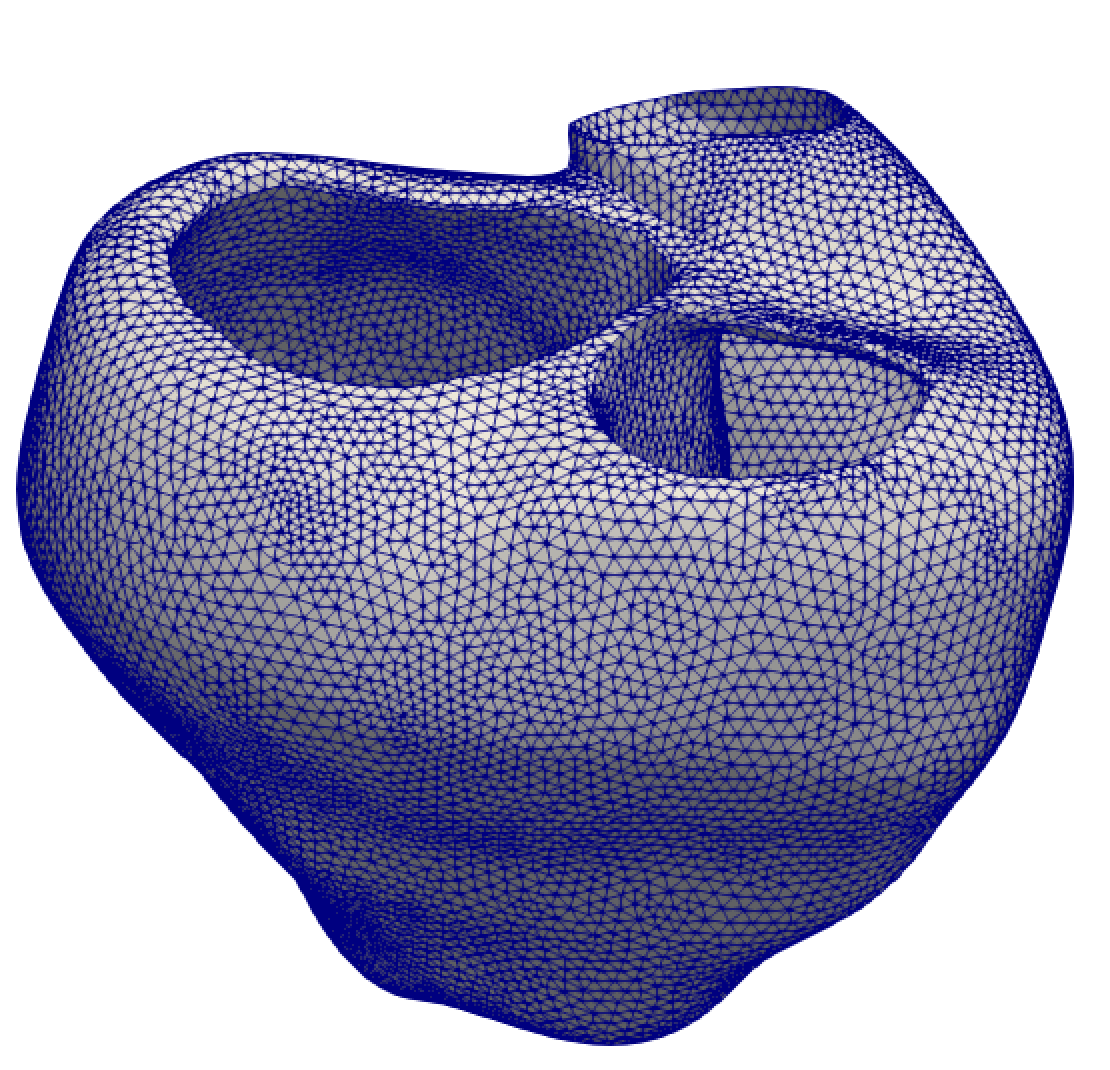}
\includegraphics[width=0.3\textwidth]{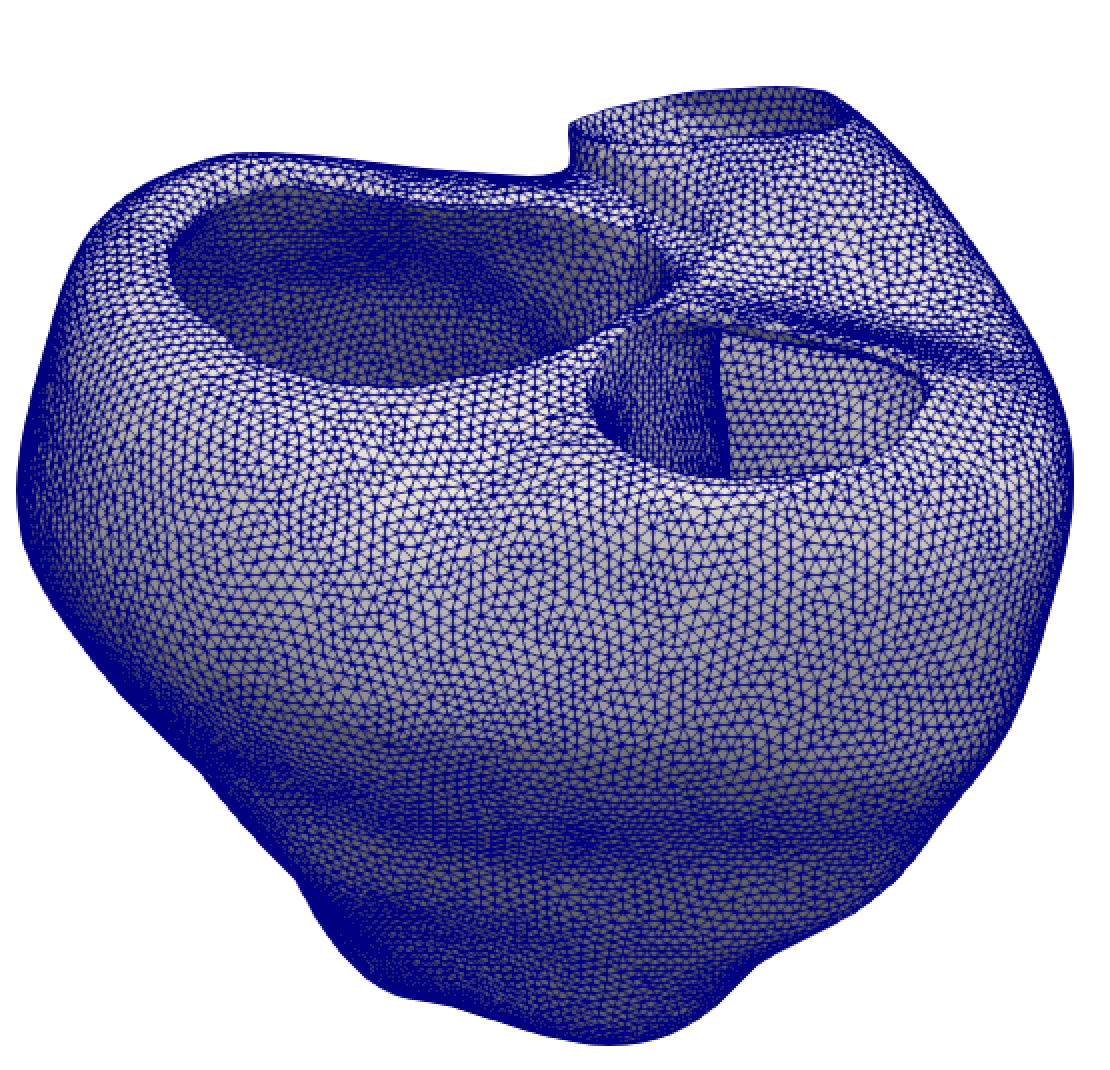}
\caption{Non-nested mesh hierarchy for the realistic heart geometry.}
\label{fig::mesh-hierarchy-heart}
\end{figure} 

\begin{table}[h!]
\centering
\begin{tabular}{|c||l|l|l|l|l|l|} 
\hline $l$ & 0 &  1 & 2 & 3 & 4 & 5  \\
\hline
\hline DOF  & 18'480 & 113'312 & 583'104 & 1'740'800 & 8'777'728 & 34'894'848   \\
\hline
\hline $ h $  & 0.16 & 0.08 & 0.04 & 0.03 & 0.02 & 0.01   \\
\hline $ \Delta t $  & 0.16 & 0.08 & 0.04 & 0.02 & 0.01 & 0.005   \\
\hline 
\end{tabular}
\caption{Details about the mesh hierarchy for the realistic heart geometry.}
\label{table:heart-levels}
\end{table}

As we do not have nested finite element spaces, we rely here
on the multilevel estimator \eqref{eq:PimpedML}. Moreover, 
we evaluate the convergence for the two quantities of interest, 
namely the action potential and the activation times for given 
locations in the domain.


\subsubsection{Action potential}
We evaluate the evolution of the action potential in several 
locations of the heart domain. The first example considers a 
set of points that are placed along the wall separating the left 
and right ventricles. These points are shown in Figure \ref{fig::wall-points}. 

\begin{figure}[h!]
\centering
\begin{tikzpicture}
\draw(0,0)node{\includegraphics[width=0.49\textwidth]{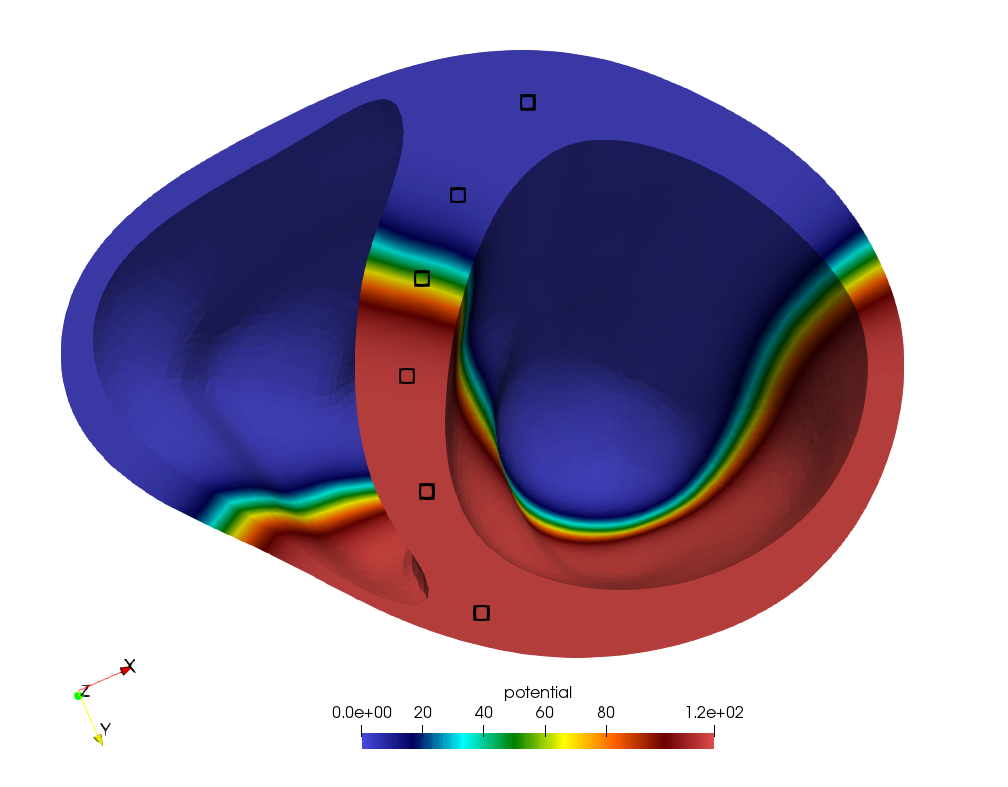}};
\draw[line width=0.8pt] (-0.12,-1.53)--(3.6,-1.53);
\draw[line width=0.8pt] (-0.52,-0.66)--(3.6,-0.66);
\draw[line width=0.8pt] (-0.64,0.16)--(3.6,0.16);
\draw[line width=0.8pt] (-0.55,0.86)--(3.6,0.86);
\draw[line width=0.8pt] (-0.27,1.46)--(3.6,1.46);
\draw[line width=0.8pt] (0.14,2.12)--(3.6,2.12);
\draw (3.8,-1.53)node{P1};
\draw (3.8,-0.66)node{P2};
\draw (3.8,0.16)node{P3};
\draw (3.8,0.86)node{P4};
\draw (3.8,1.46)node{P5};
\draw (3.8,2.12)node{P6};
\end{tikzpicture}

\caption{Locations selected along the wall separating the left and right ventricles.}
\label{fig::wall-points}
\end{figure}
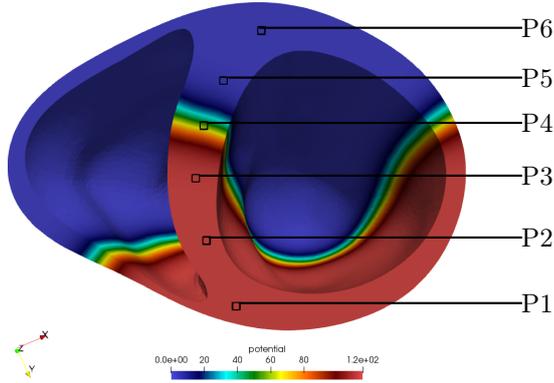 

As one can see, these points have been selected such that they 
trace the behaviour of locations at different distance from the stimulus 
center, starting from very close (the very below point) to relatively 
far (the very top point). We report in Figure~\ref{fig::AP-wall} the 
action potential obtained by the MC quadrature method for different 
discretization levels.

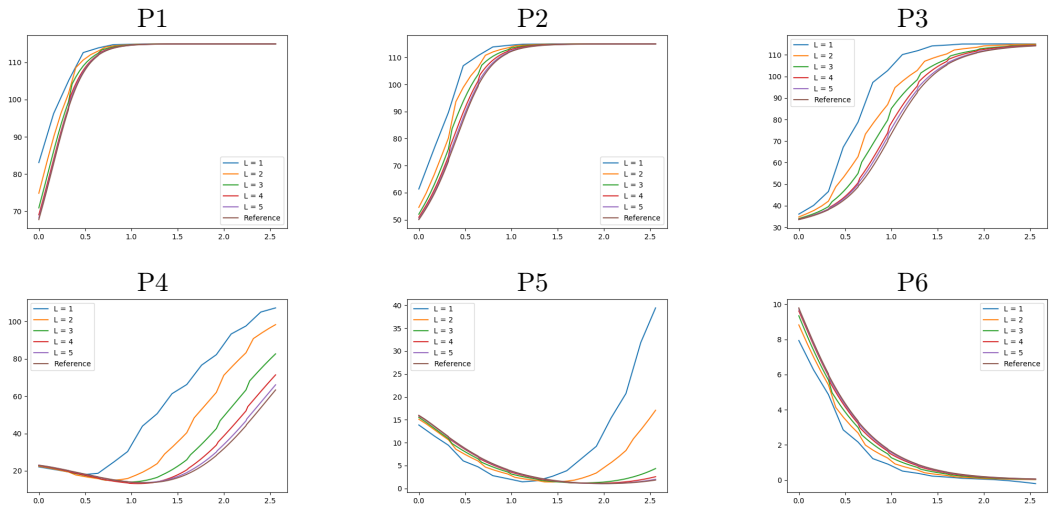
\begin{figure}[H] 
\begin{tikzpicture}
\draw(0,0)node{
\includegraphics[width=0.3\textwidth]{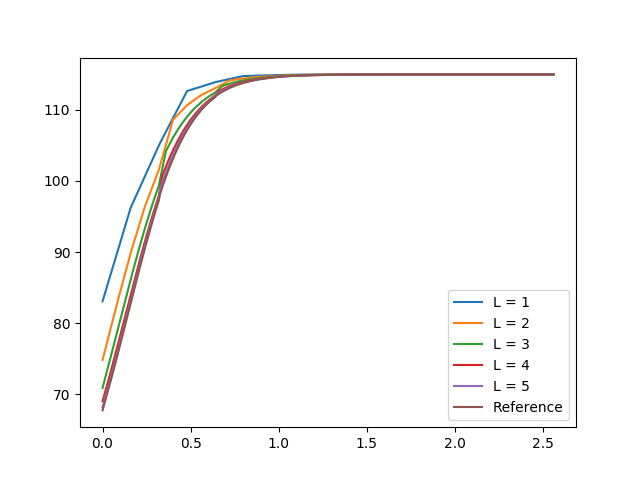}};
\draw(5,0)node{
\includegraphics[width=0.3\textwidth]{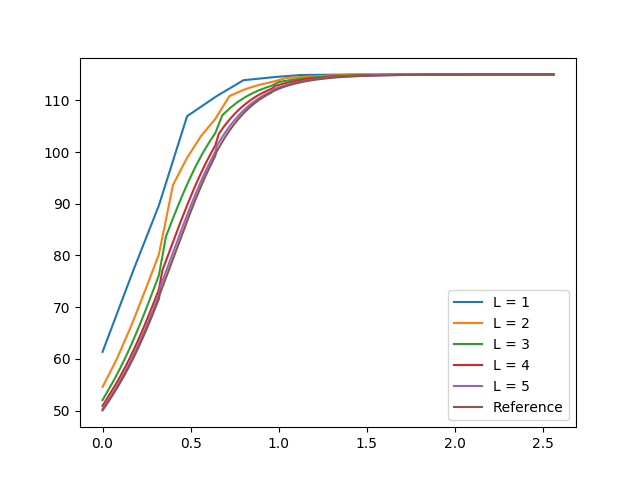}};
\draw(10,0)node{
\includegraphics[width=0.3\textwidth]{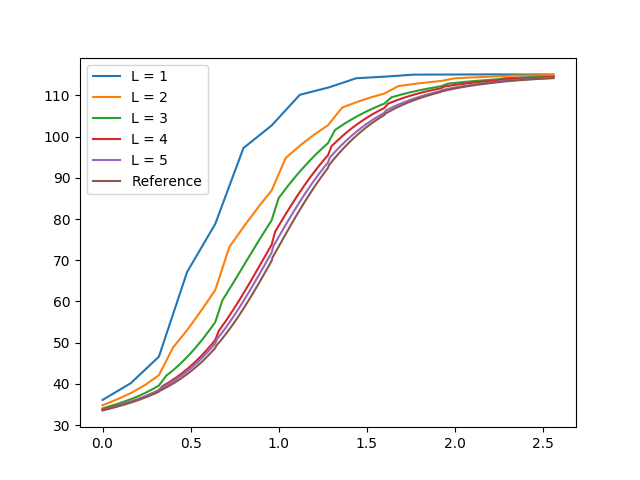}};
\draw(0,-3.5)node{
\includegraphics[width=0.3\textwidth]{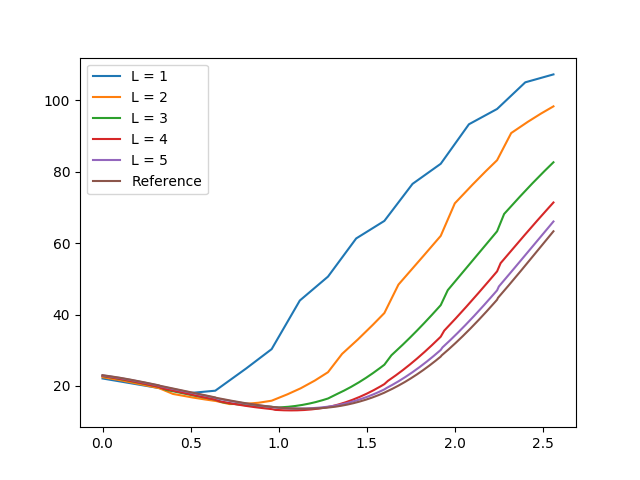}};
\draw(5,-3.5)node{
\includegraphics[width=0.3\textwidth]{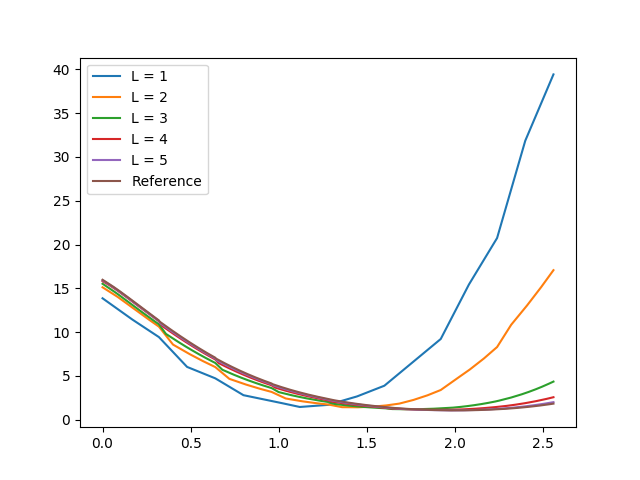}};
\draw(10,-3.5)node{
\includegraphics[width=0.3\textwidth]{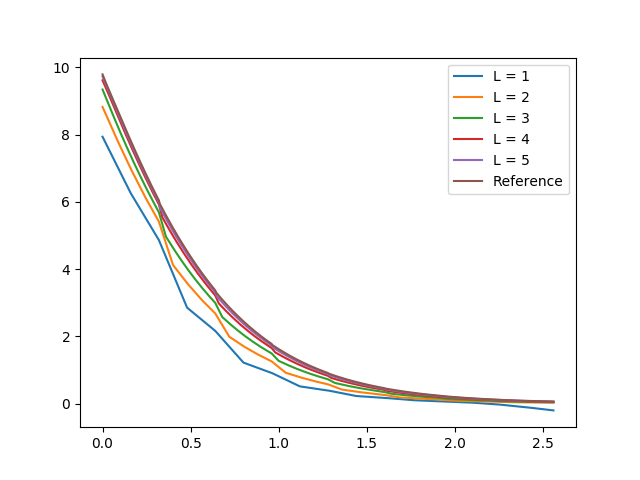}};
\draw(0,1.5)node{P1};
\draw(5,1.5)node{P2};
\draw(10,1.5)node{P3};
\draw(0,-2)node{P4};
\draw(5,-2)node{P5};
\draw(10,-2)node{P6};
\end{tikzpicture}
\captionsetup{justification=centering}
\caption{Action potential behaviour given different mesh level discretizations 
for the points specified in Figure \ref{fig::wall-points} following the order going from the bottom to the top.}
\label{fig::AP-wall}
\end{figure} 

In Figure \ref{fig::AP-convergence-1}, we report the convergence graphs 
of the (pointwise) error \eqref{eq:point-error} for the action potential 
at the locations introduced in Figure \ref{fig::wall-points} and $q=1$.
Notice that the graphs report the root mean square errors.
The expected convergence rate is achieved for all quadrature methods tested.

\begin{figure}[H] 
\centering
\includegraphics[width=0.3\textwidth]{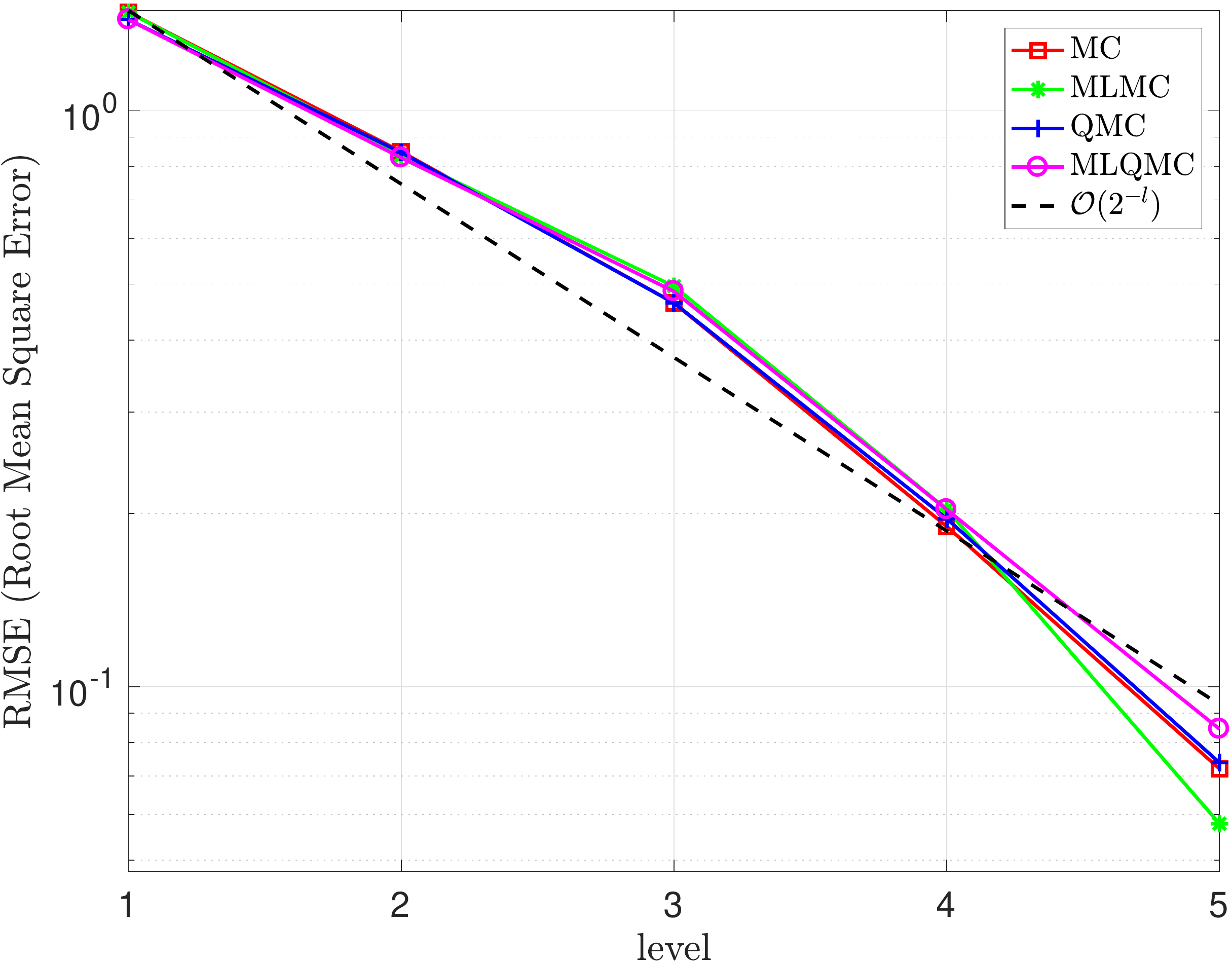}
\includegraphics[width=0.3\textwidth]{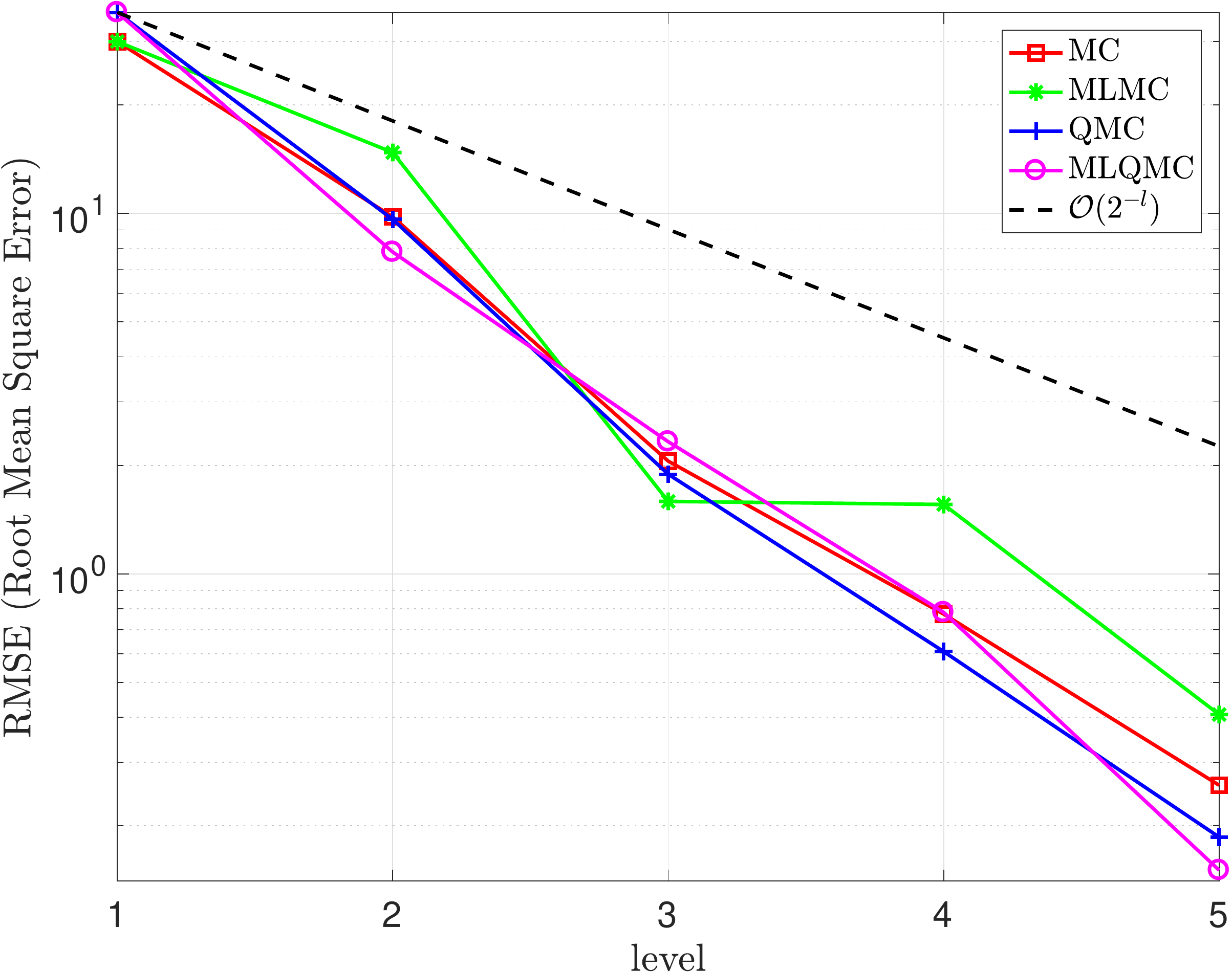}
\includegraphics[width=0.3\textwidth]{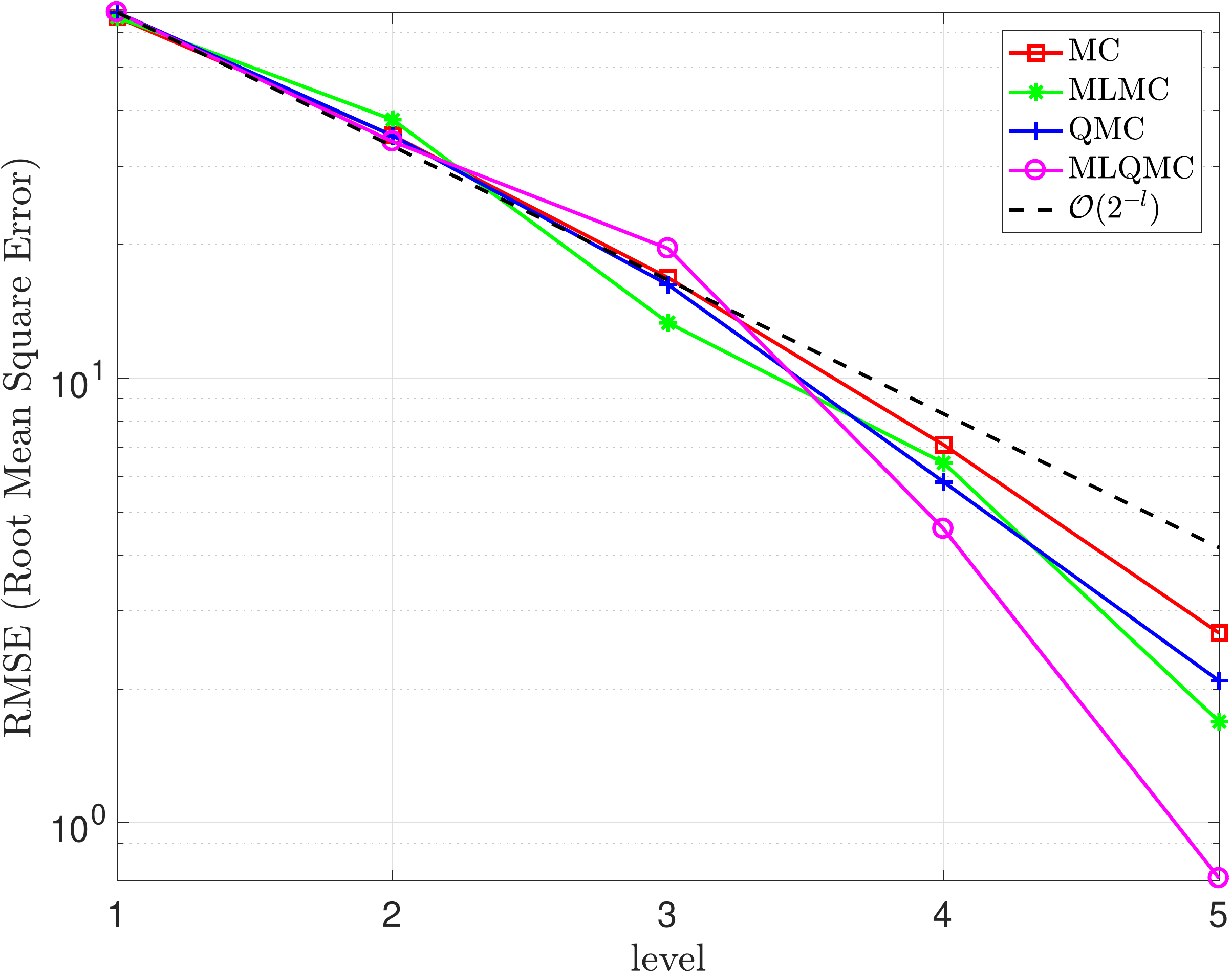}
\includegraphics[width=0.3\textwidth]{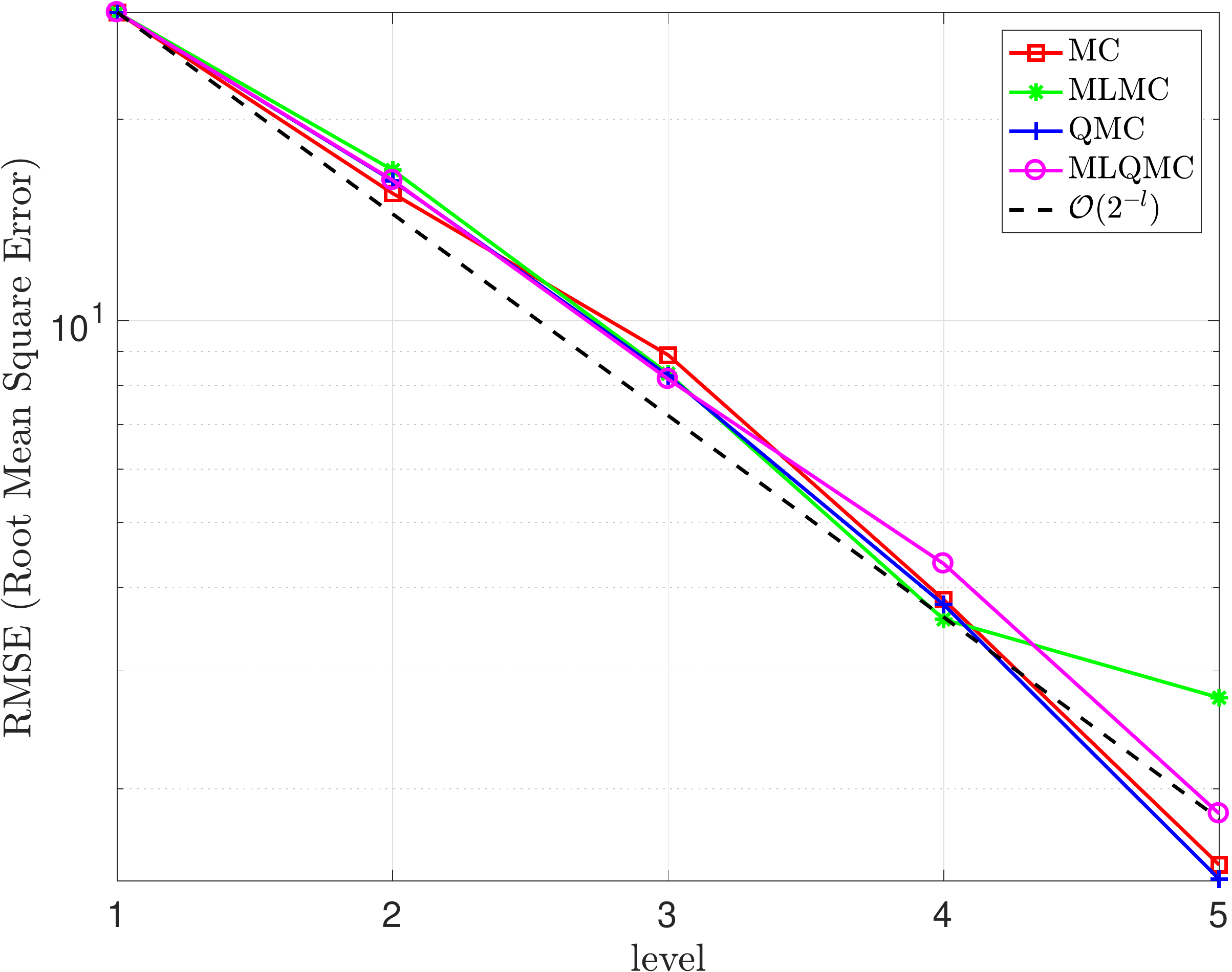}
\includegraphics[width=0.3\textwidth]{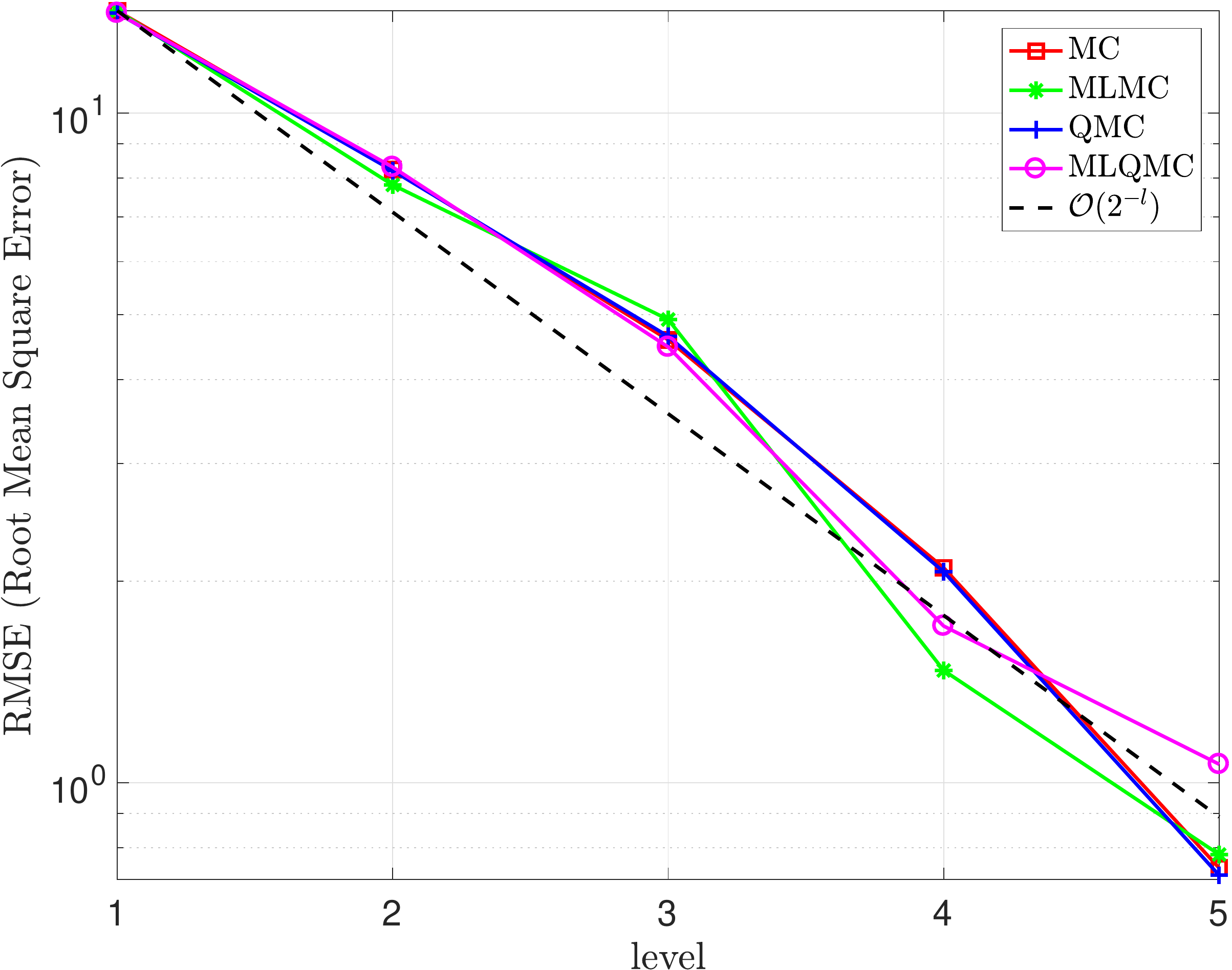}
\includegraphics[width=0.3\textwidth]{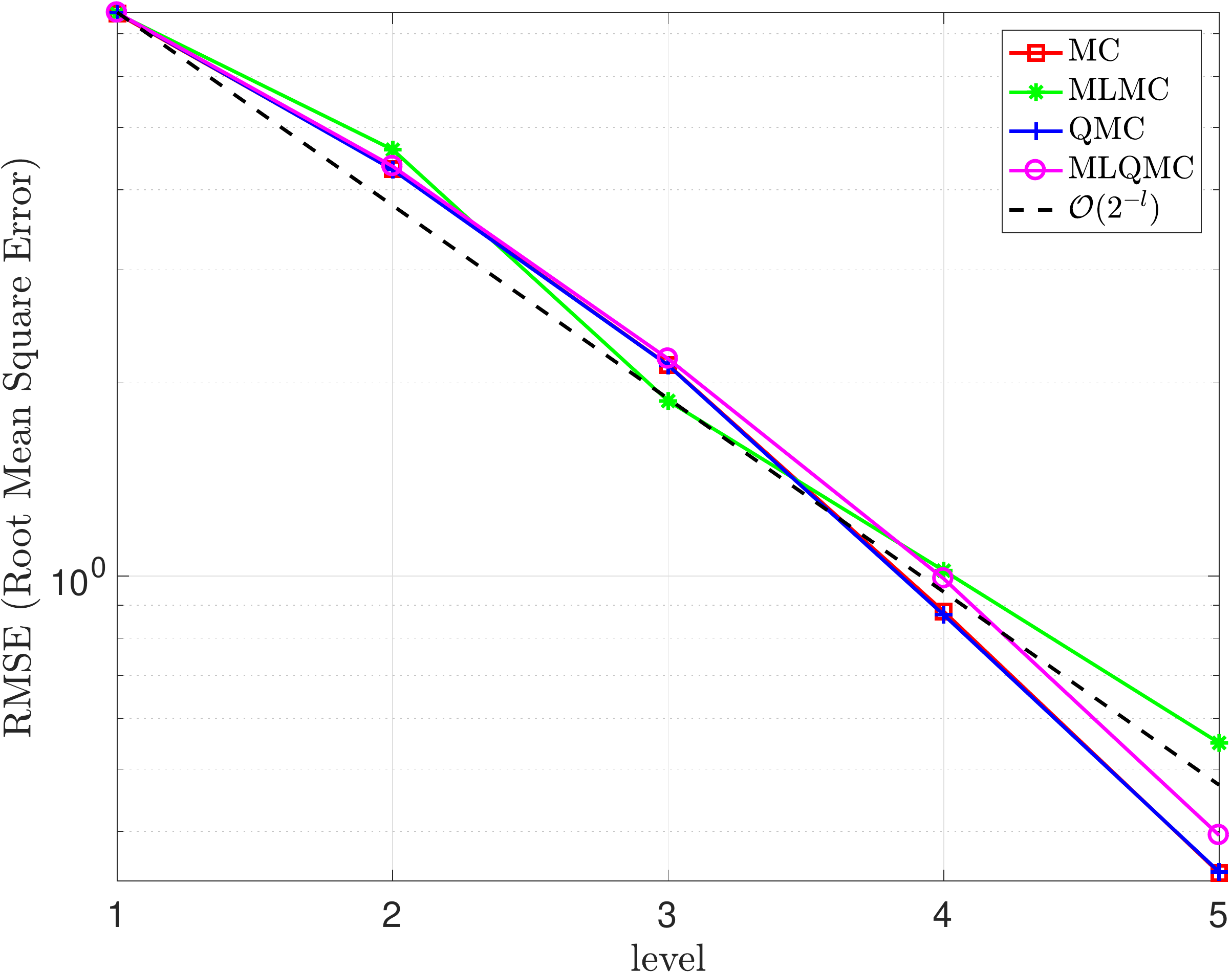}
\captionsetup{justification=centering}
\caption{Convergence in the $H^1$ norm of the action potential at the 
locations specified in Figure \ref{fig::wall-points} following the order going from the bottom to the top.}
\label{fig::AP-convergence-1}
\end{figure} 

The second test is concerned with points located at the circumference 
of a horizontal cut of the heart surface. These are shown in Figure 
\ref{fig::circonference-points}.
\begin{figure}[h!]
\centering
\includegraphics[width=0.49\textwidth]{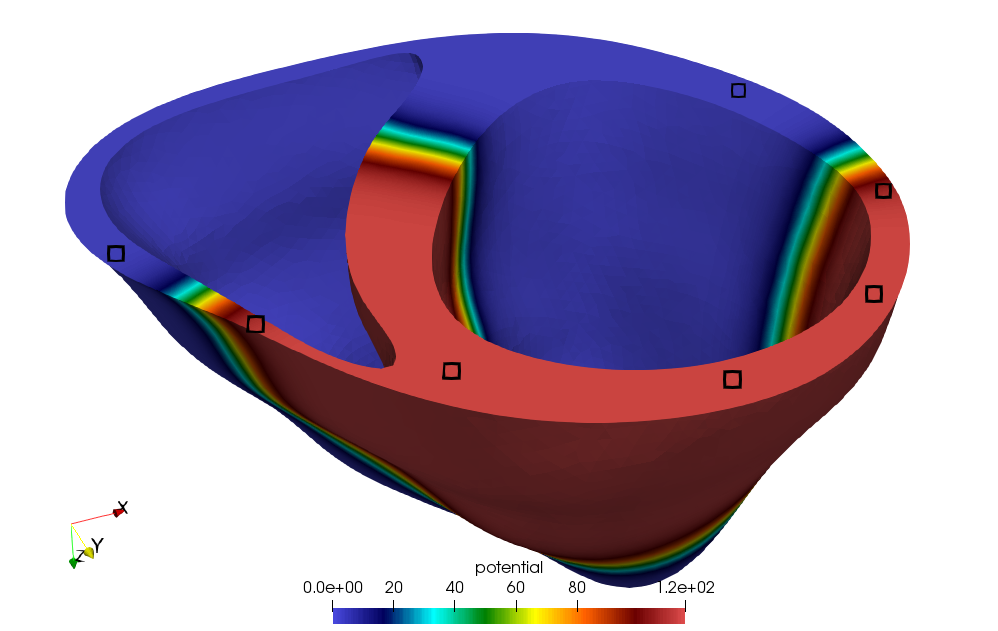}
\captionsetup{justification=centering}
\caption{Locations selected at the circumference of a horizontal cut of the heart surface.}
\label{fig::circonference-points}
\end{figure} 

Since the behaviour of the action potential at these points follows
a similar pattern to that of the previously shown ones, cf.\ Figure 
\ref{fig::AP-wall}, we directly show the graph regarding the convergence of 
the action potential at these points in Figure \ref{fig::AP-convergence-2}.
Again, we see the expected convergence rate for all quadrature methods tested.

\begin{figure}[H] 
\centering
\includegraphics[width=0.3\textwidth]{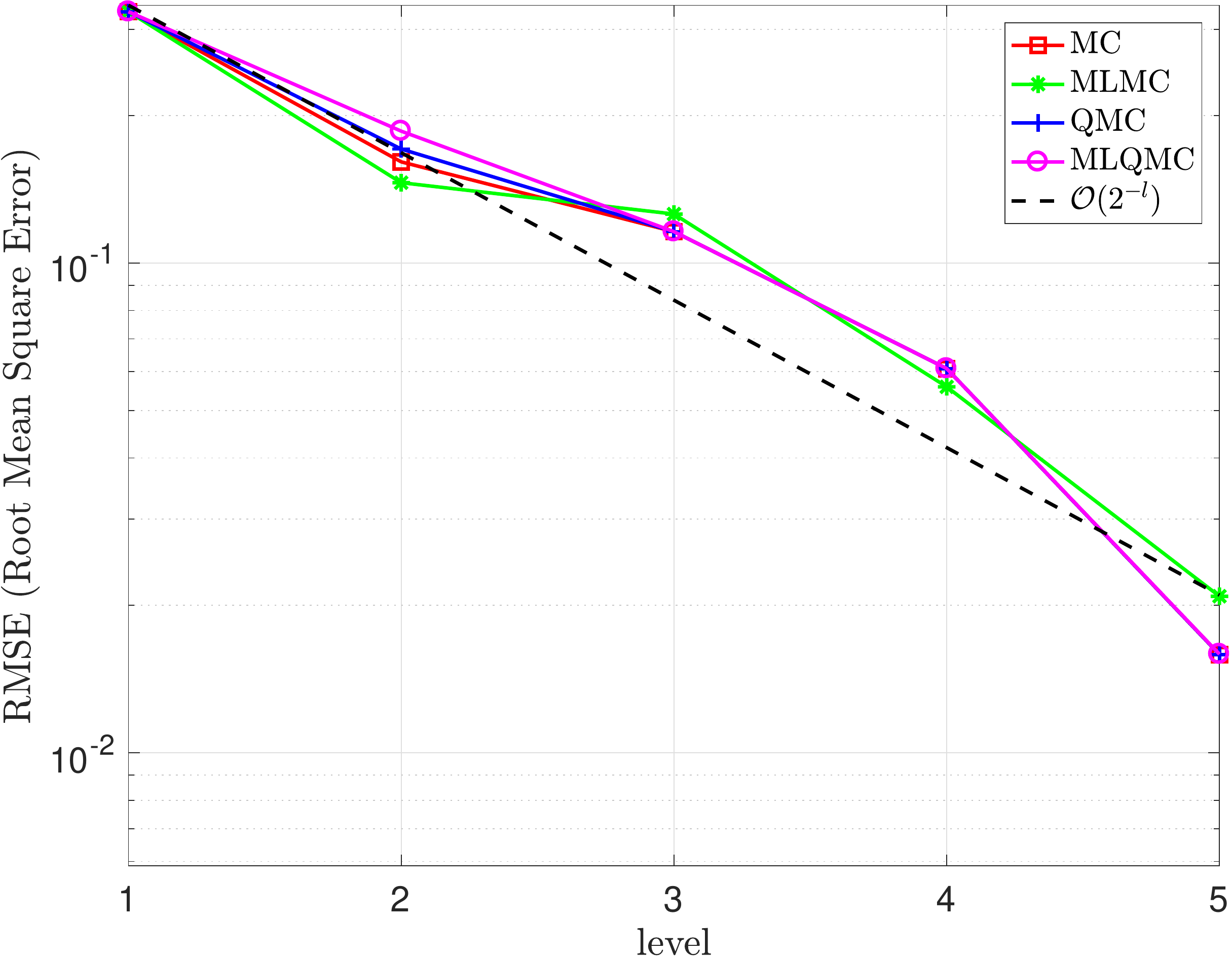}
\includegraphics[width=0.3\textwidth]{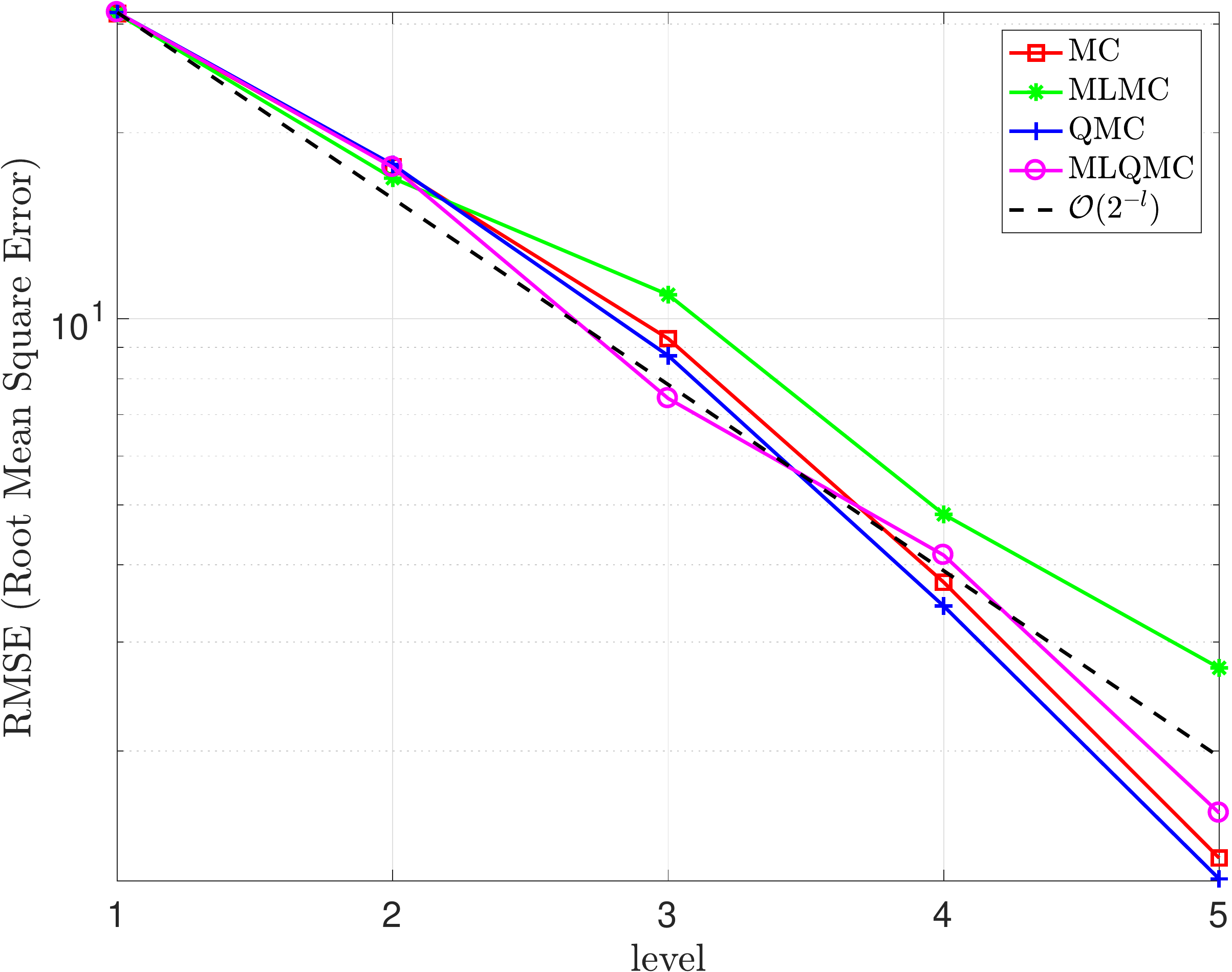}
\includegraphics[width=0.3\textwidth]{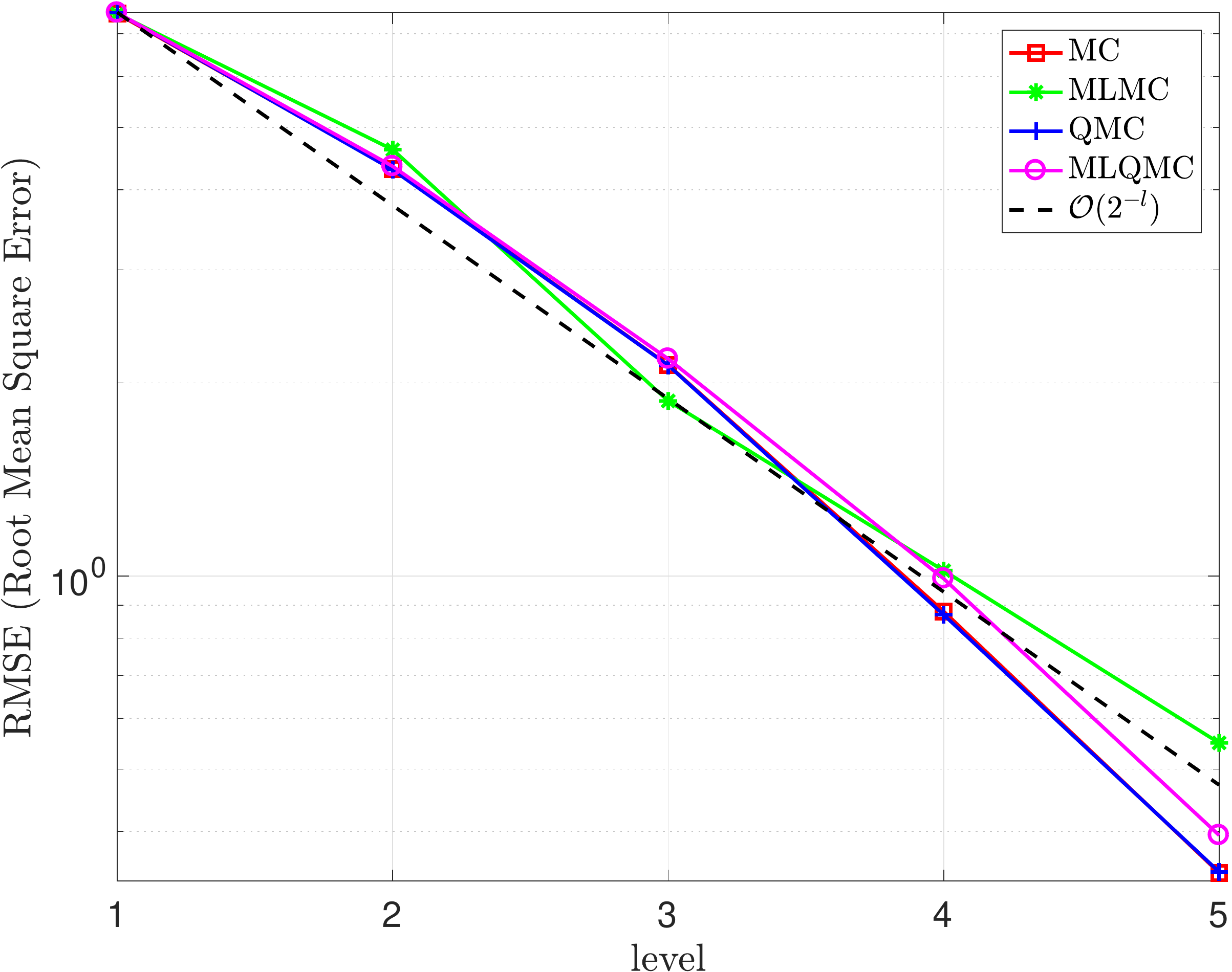}
\includegraphics[width=0.3\textwidth]{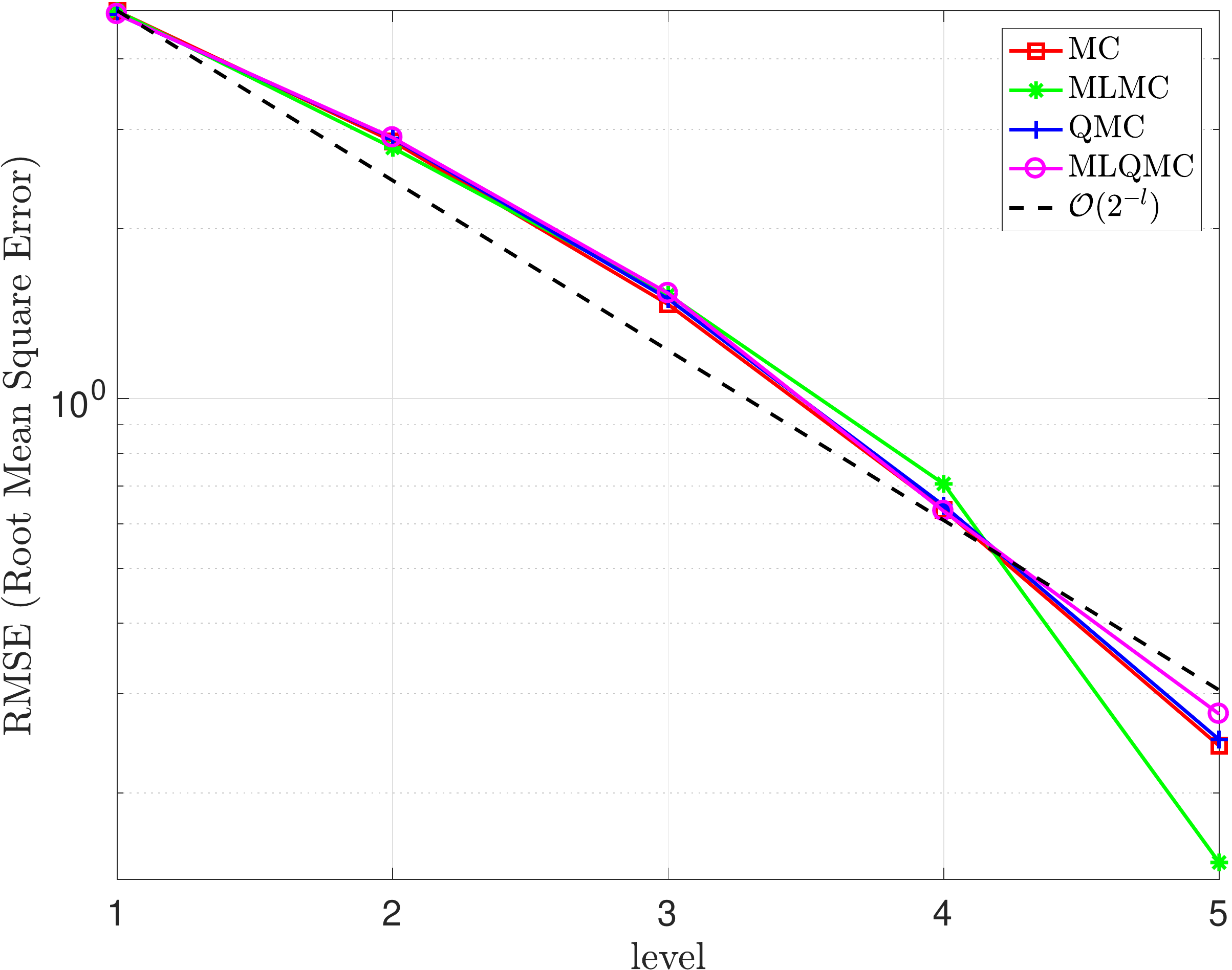}
\includegraphics[width=0.3\textwidth]{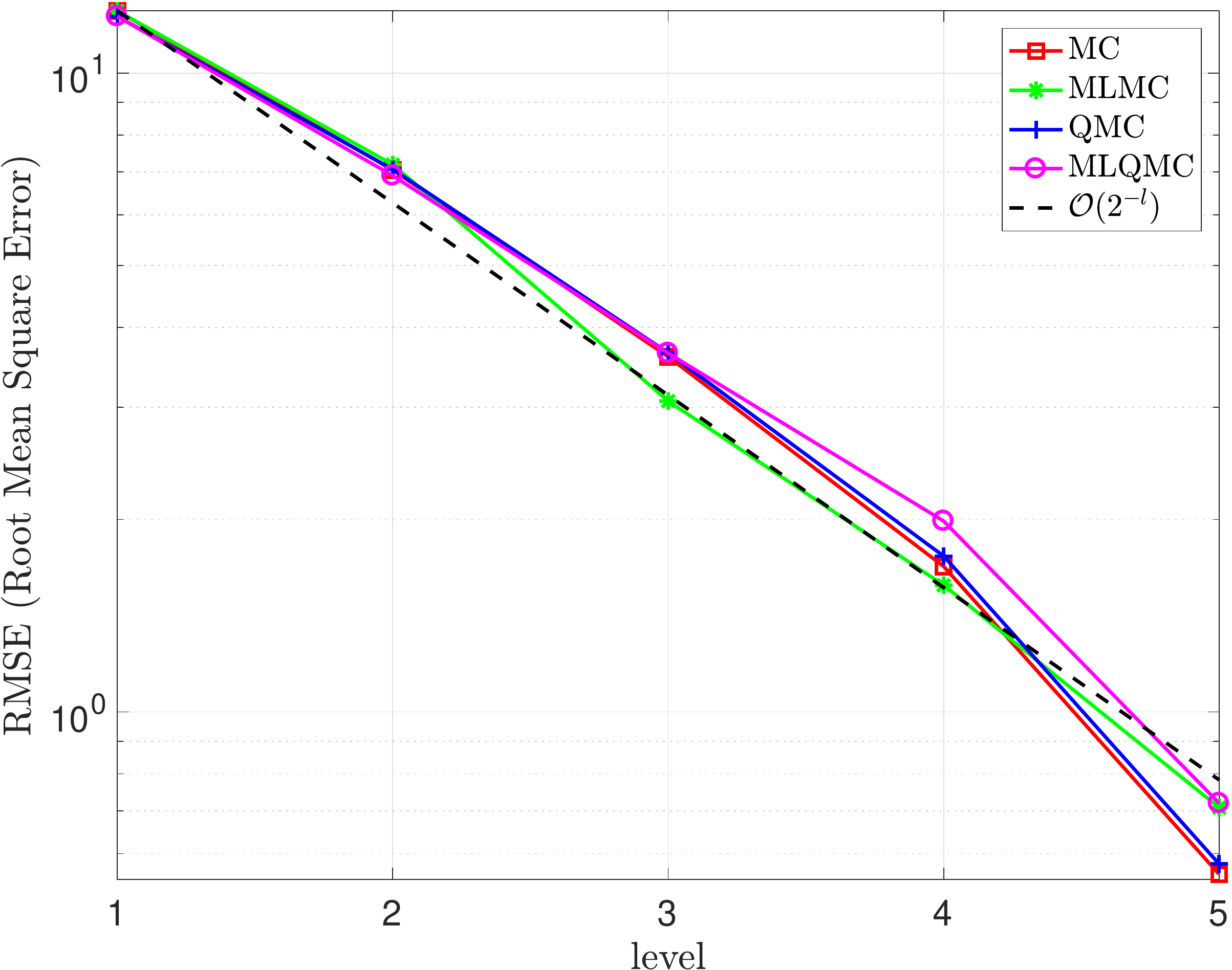}
\includegraphics[width=0.3\textwidth]{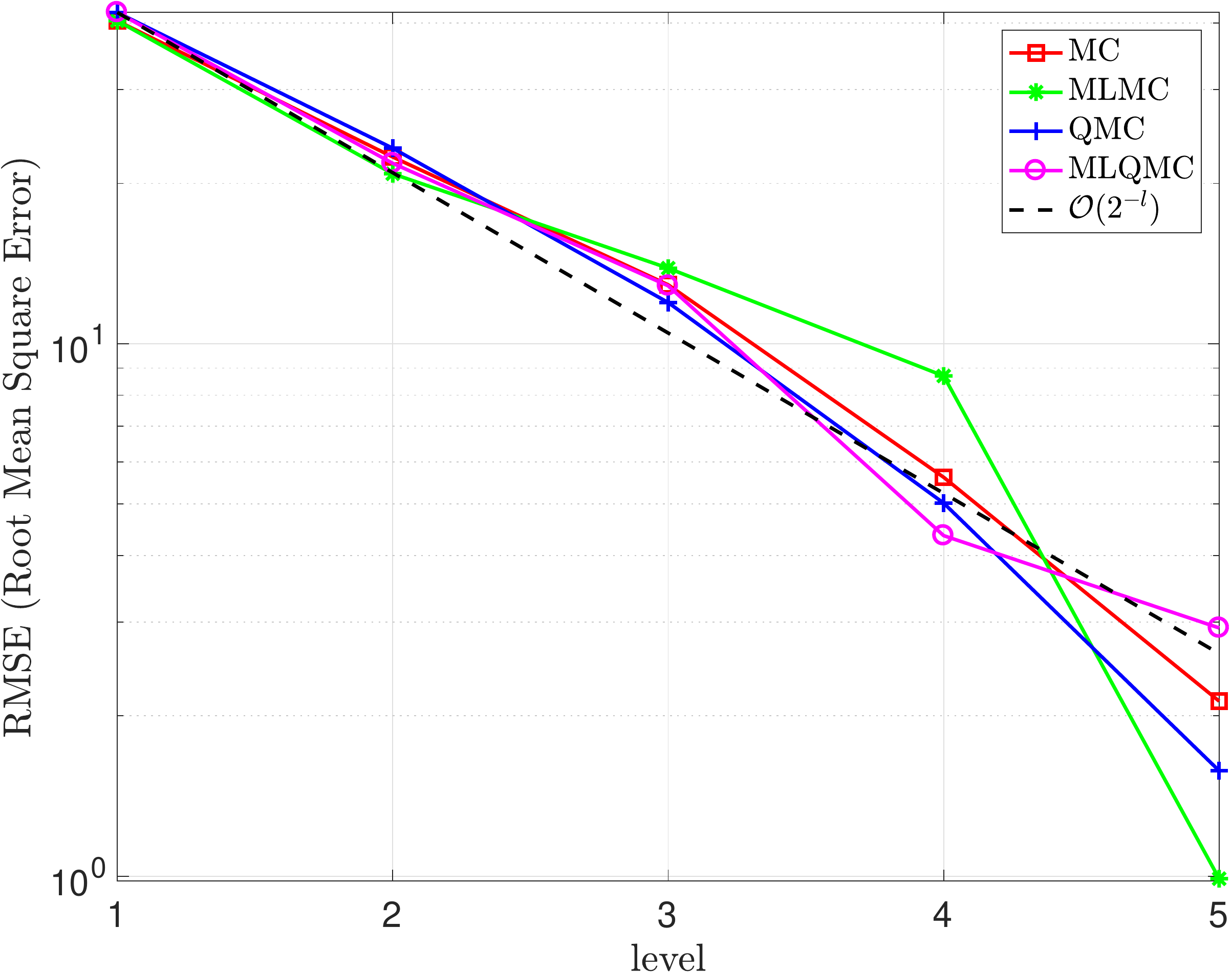}
\includegraphics[width=0.3\textwidth]{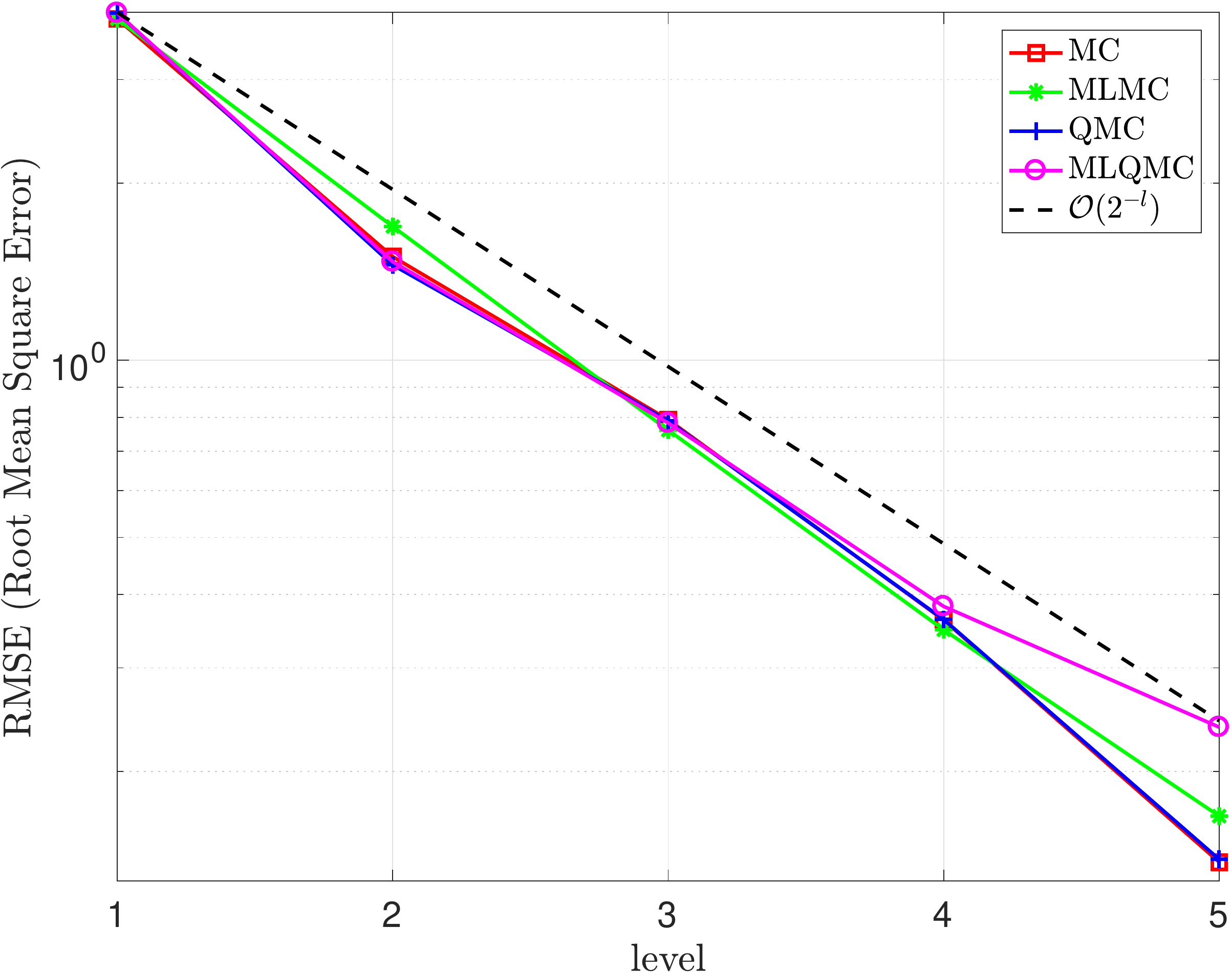}
\captionsetup{justification=centering}
\caption{Convergence of the action potential at the locations 
specified in Figure \ref{fig::circonference-points}
following the order going from left to right.}
\label{fig::AP-convergence-2}
\end{figure} 


\subsubsection{Activation time}
We start by selecting points at equivalent geodesic 
distance from the stimulus location. These points are 
shown in Figure \ref{fig::activation-contour}. The geodesic 
distance is calculated by solving an eikonal problem with a 
zero initial condition on the originating point~\cite{pezzuto2019sampling}, 
i.e.\ the stimulus in our case. The graphs showing the 
convergence of the activation times for these locations
are reported in Figure \ref{fig::activation-graph-1},
with all of them showing the expected rate of convergence for all quadrature methods tested.

\begin{figure}[h!] 
\centering
\includegraphics[width=0.45\textwidth]{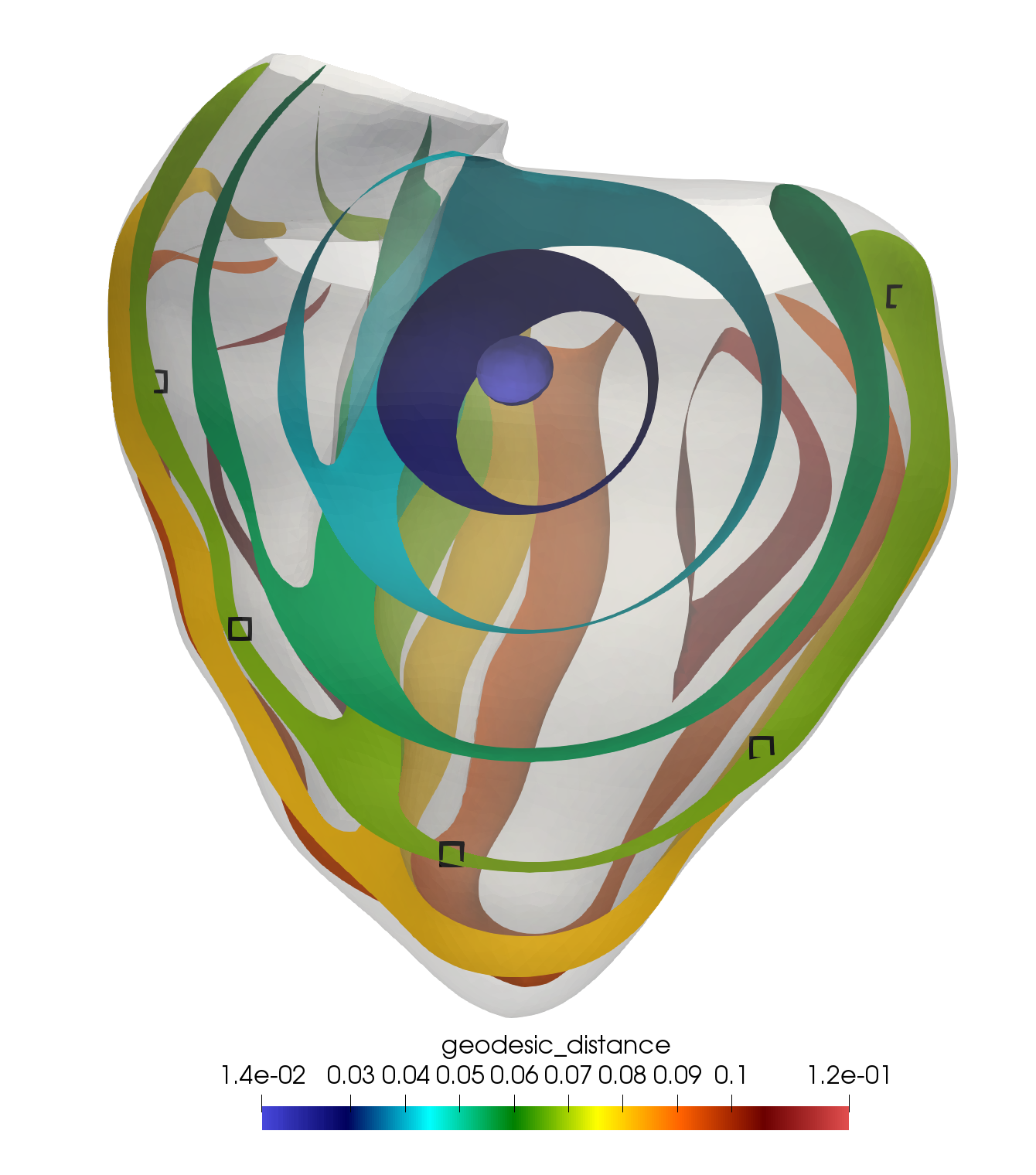}
\includegraphics[width=0.45\textwidth]{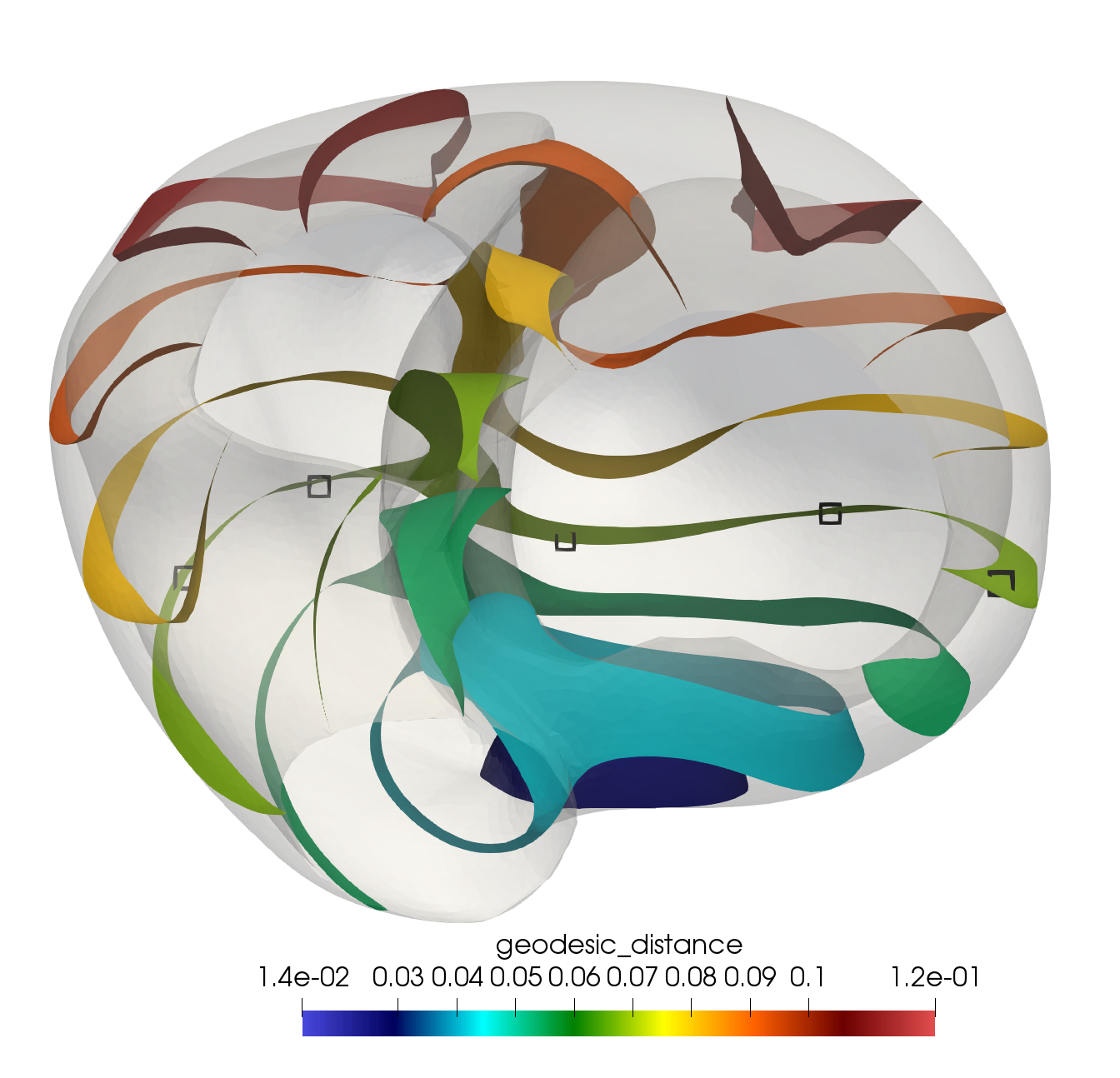}
\caption{Locations selected at equivalent geodesic distance from the stimulus.}
\label{fig::activation-contour}
\end{figure} 

\begin{figure}[H] 
\centering
\includegraphics[width=0.3\textwidth]{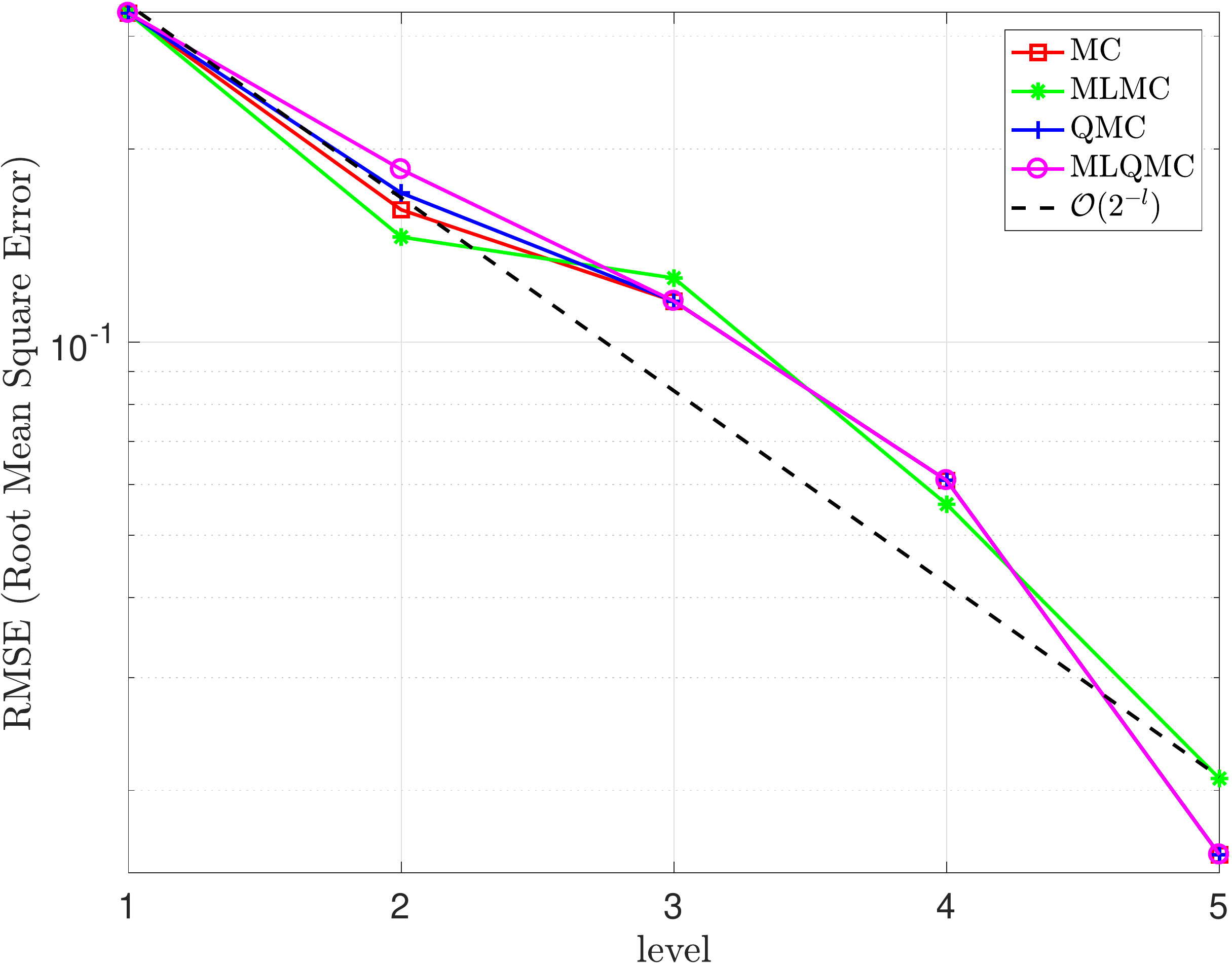}
\includegraphics[width=0.3\textwidth]{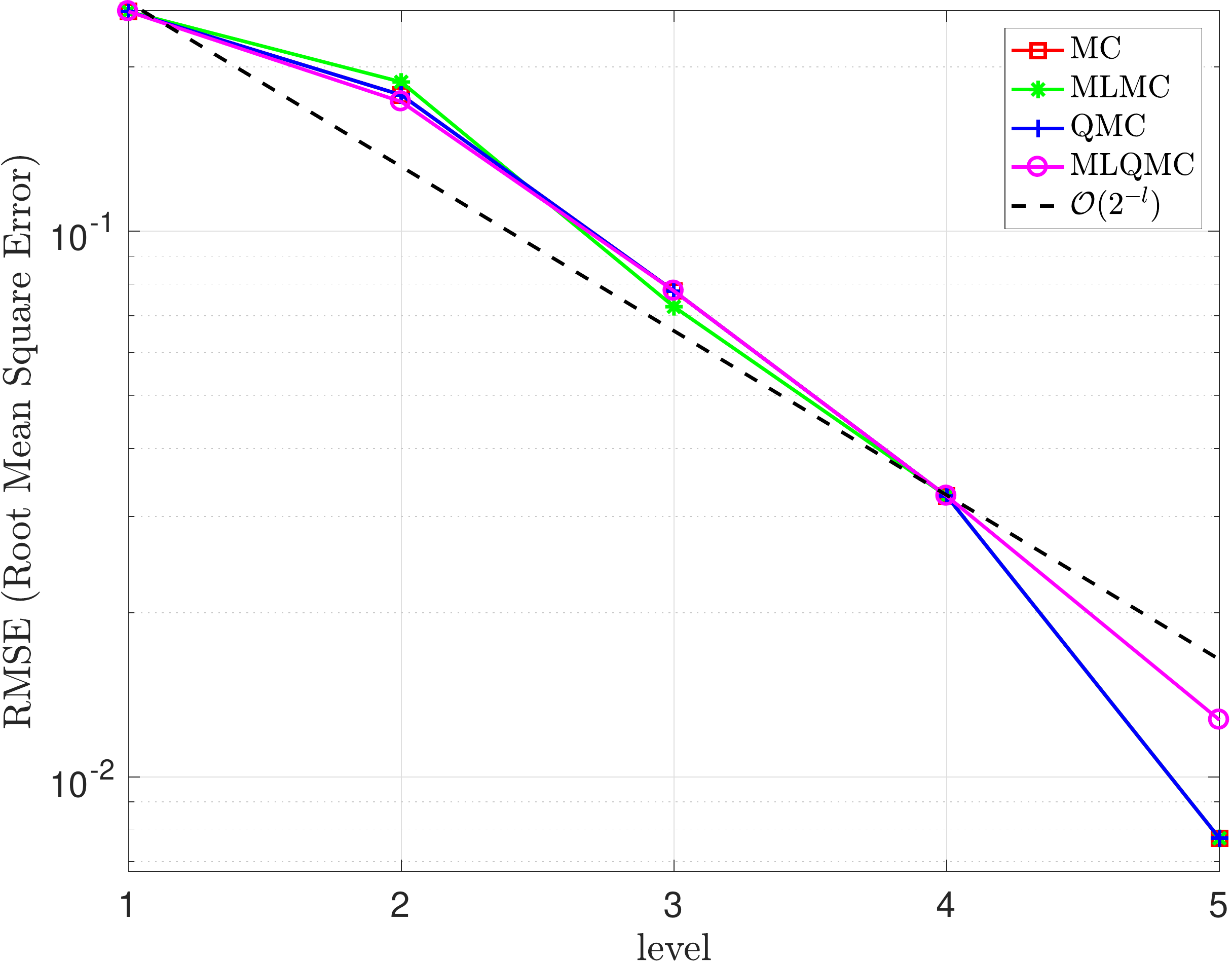}
\includegraphics[width=0.3\textwidth]{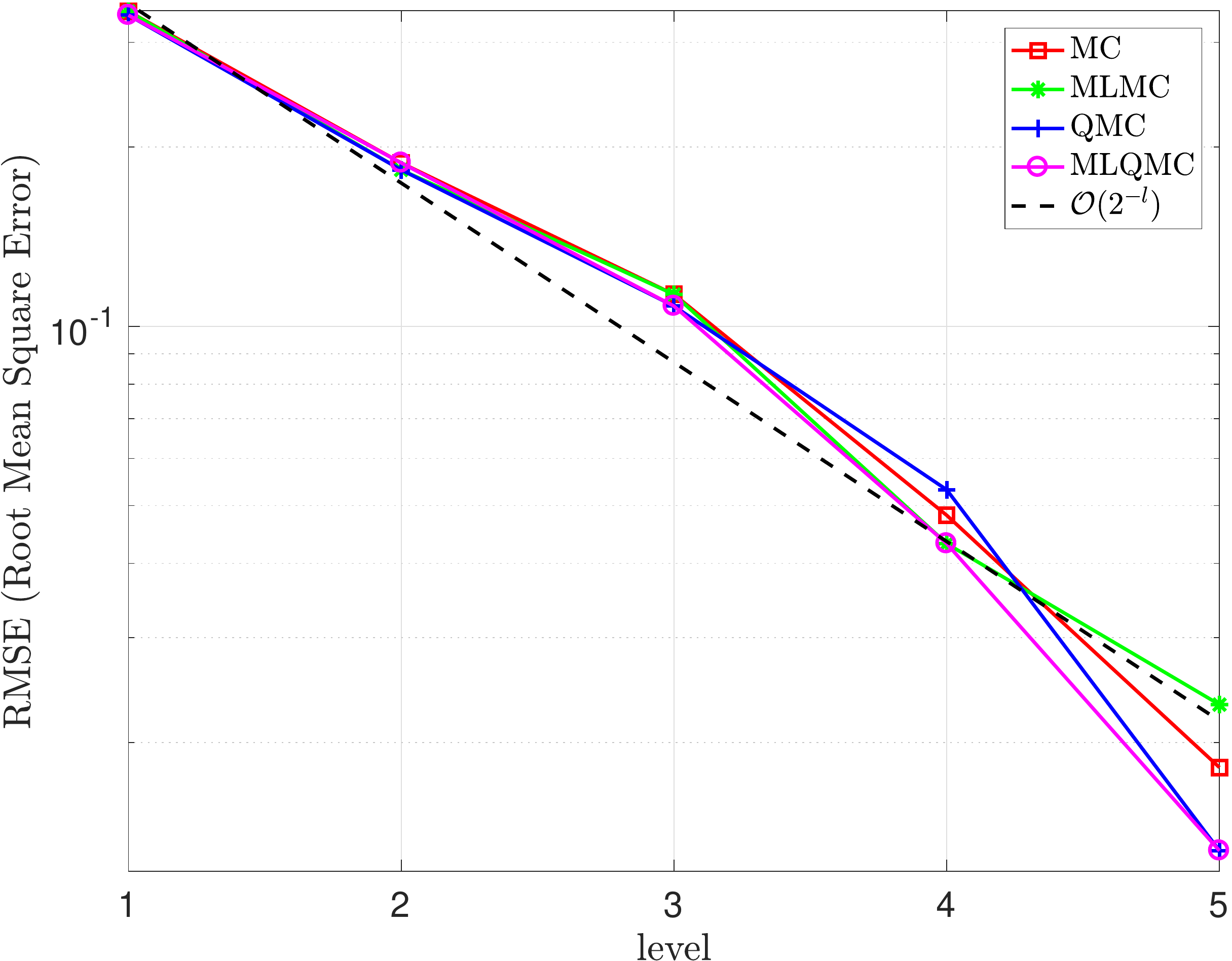}
\includegraphics[width=0.3\textwidth]{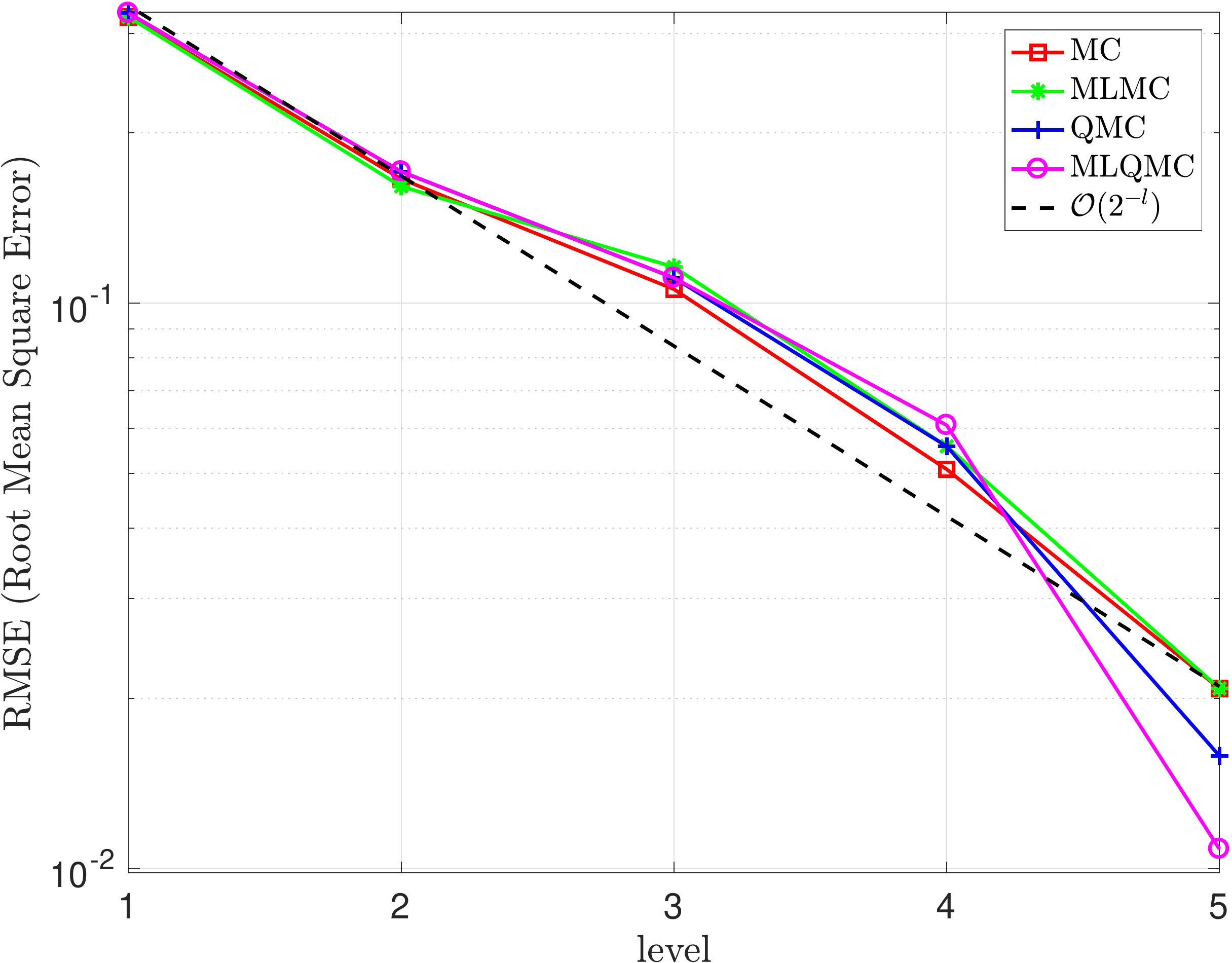}
\includegraphics[width=0.3\textwidth]{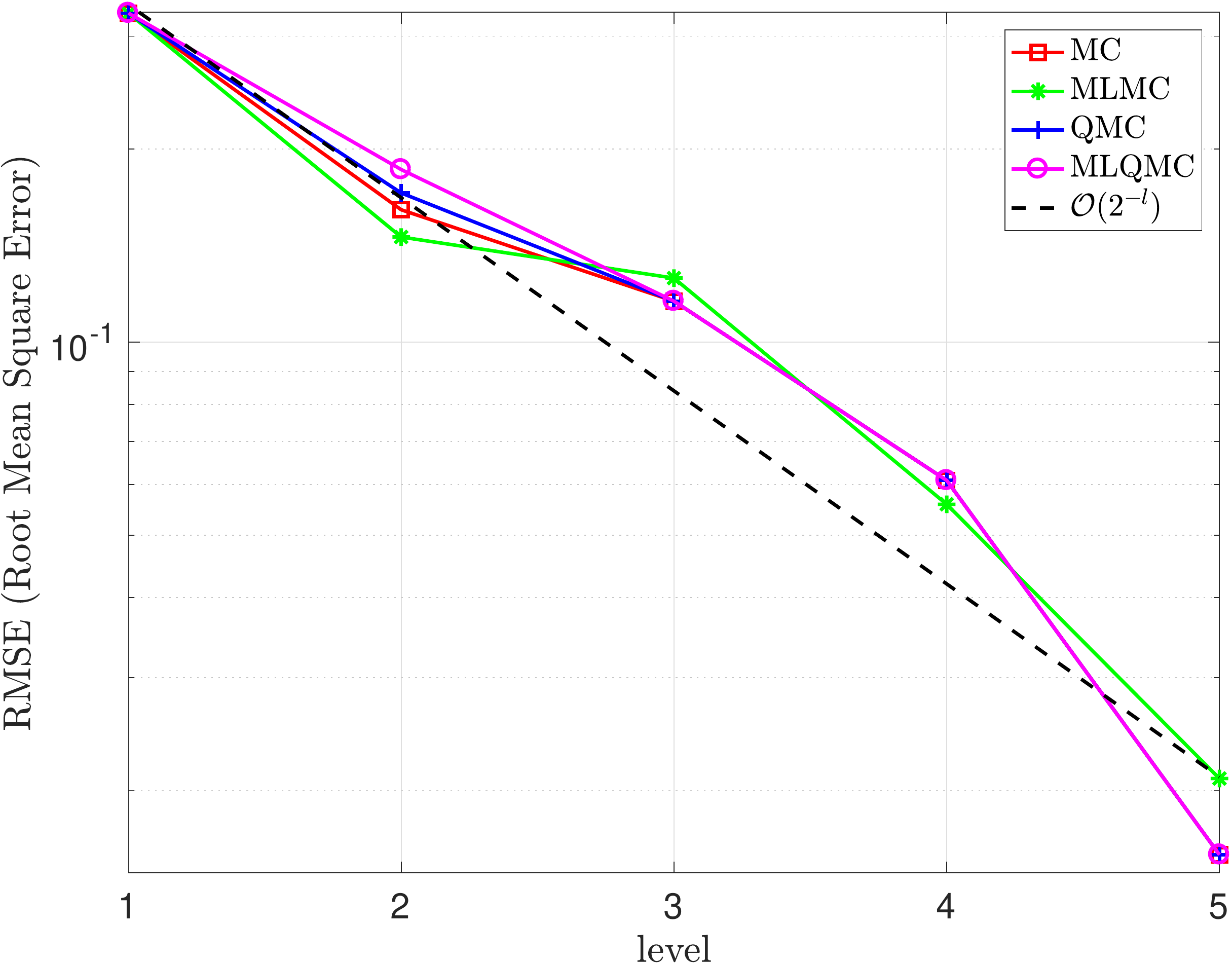}
\captionsetup{justification=centering}
\caption{Convergence of the activation time at the 
locations specified in Figure \ref{fig::activation-contour} following the order going from left to right.}
\label{fig::activation-graph-1}
\end{figure} 

We next select locations at the circumference of the left ventricle. 
These are shown in Figure \ref{fig::activation-transversal}. 
This electrical signal, when propagated to the chest,
is exactly what is perceived clinically (on a 
electrocardiogram monitor).
Mathematically, it is possible to map the surface
potential to the chest by solving an additional
diffusion problem, see \cite{fernandez2010decoupled,
boulakia2010mathematical}.
As we have several 
discretization levels, we need to ensure that these points are 
well-defined on each one of them. The convergence graphs for the activation
times at these locations,
which validate the the expected rate of convergence
for all quadrature methods tested,
are reported in Figure \ref{fig::activation-graph-2}.

\begin{figure}[htb] 
\centering
\includegraphics[width=0.45\textwidth]{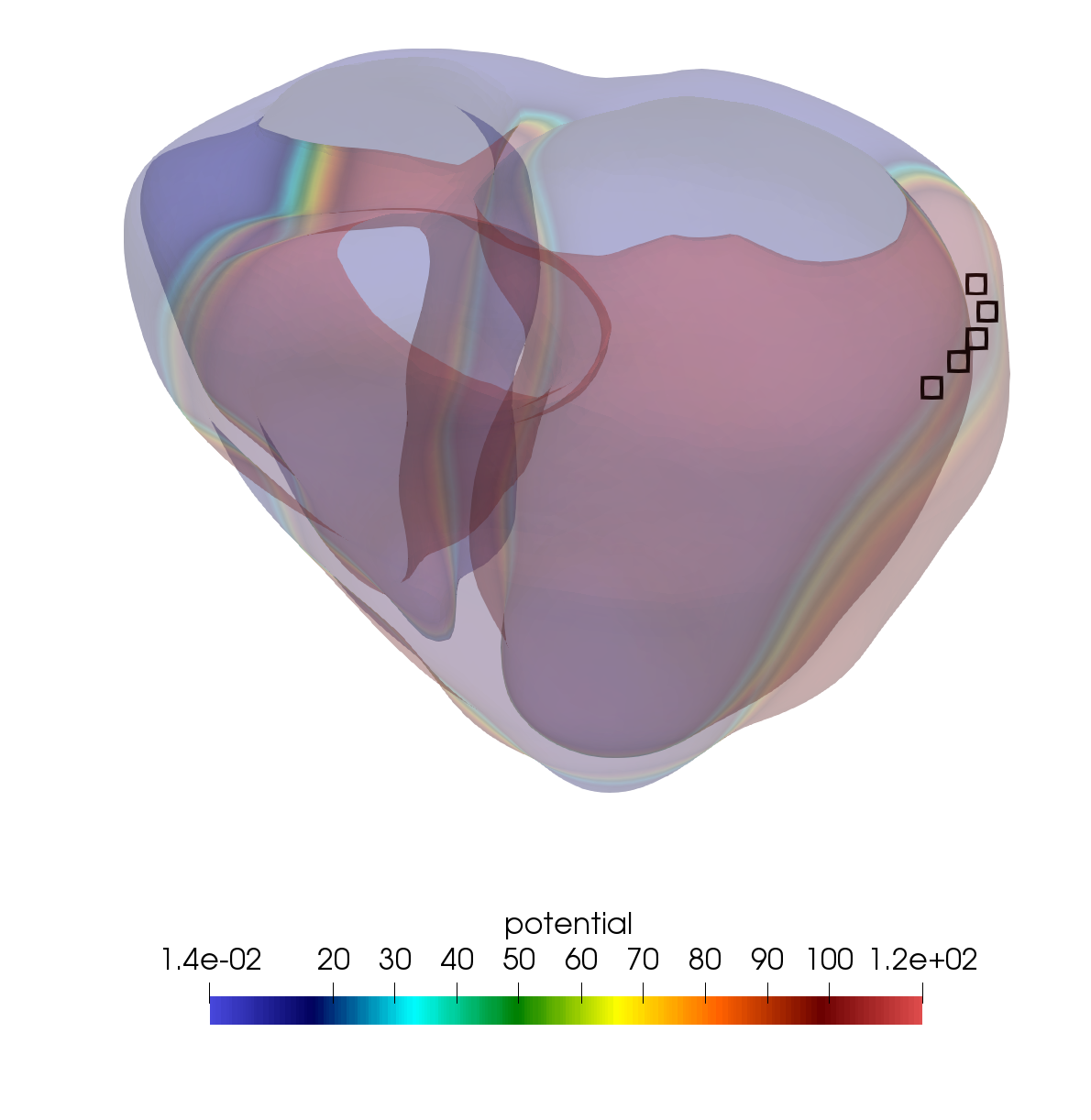}
\includegraphics[width=0.45\textwidth]{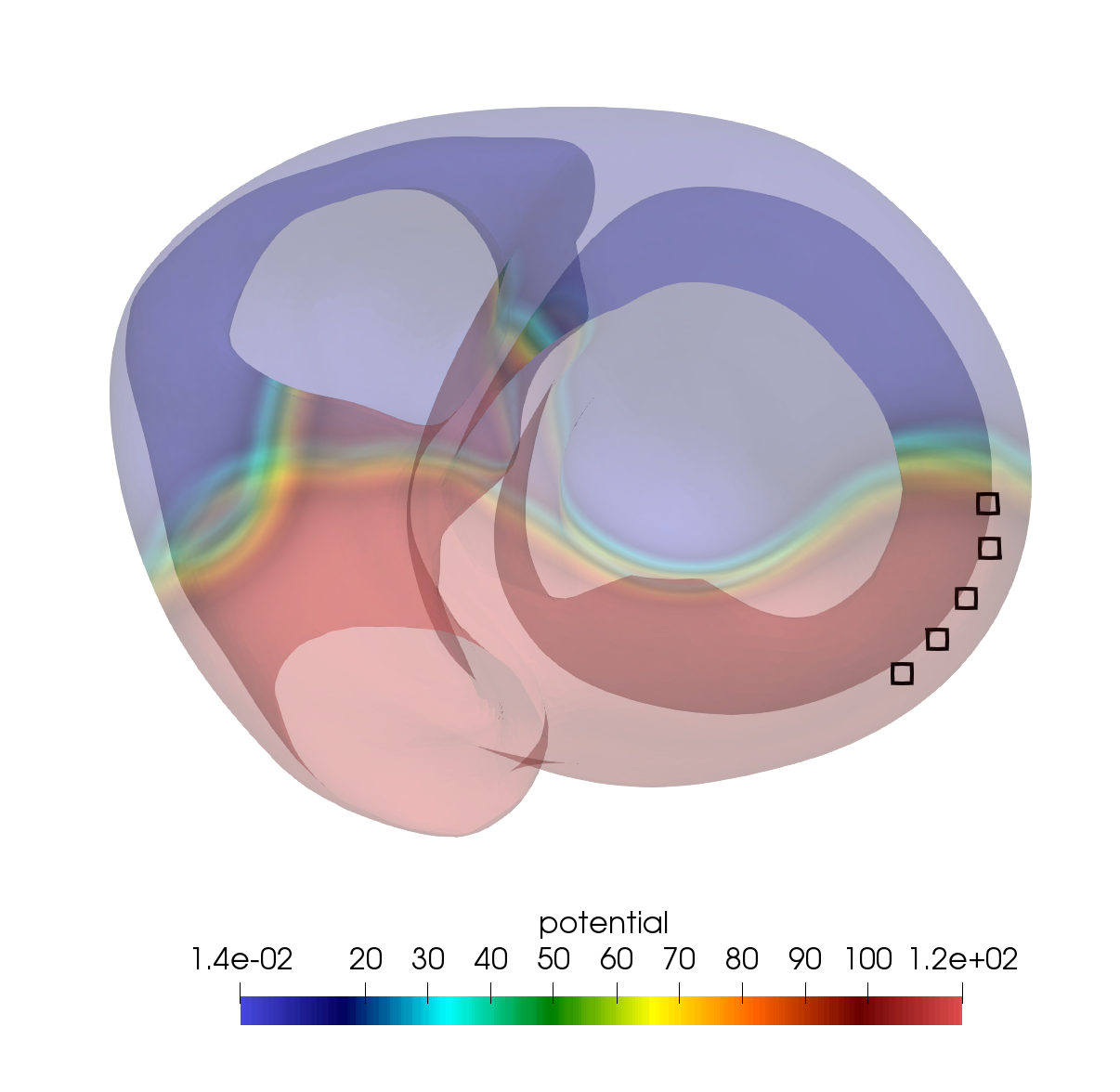}
\captionsetup{justification=centering}
\caption{Locations selected at the periphery of the heart surface.}
\label{fig::activation-transversal} 
\end{figure} 

\begin{figure}[H] 
\centering
\includegraphics[width=0.3\textwidth]{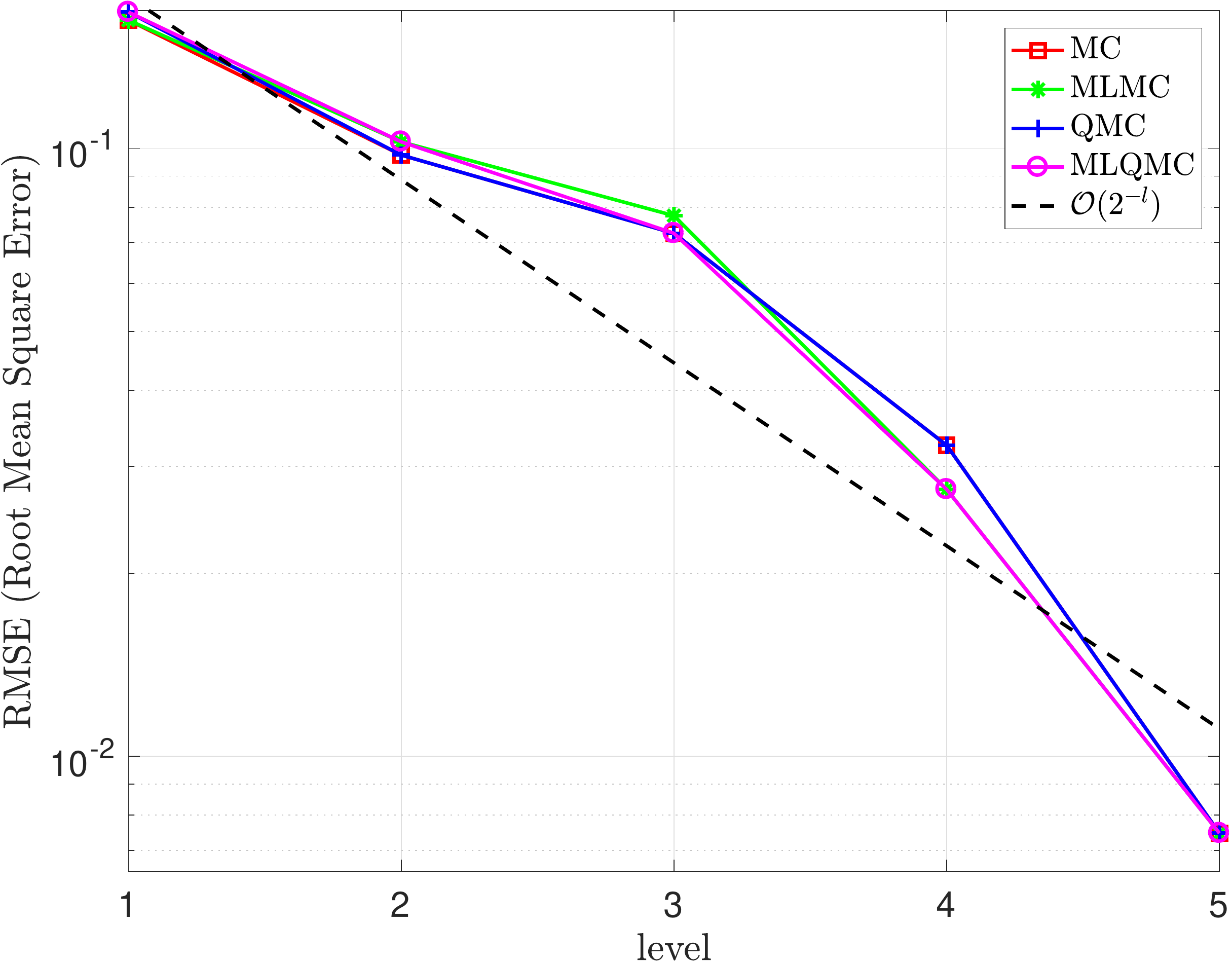}
\includegraphics[width=0.3\textwidth]{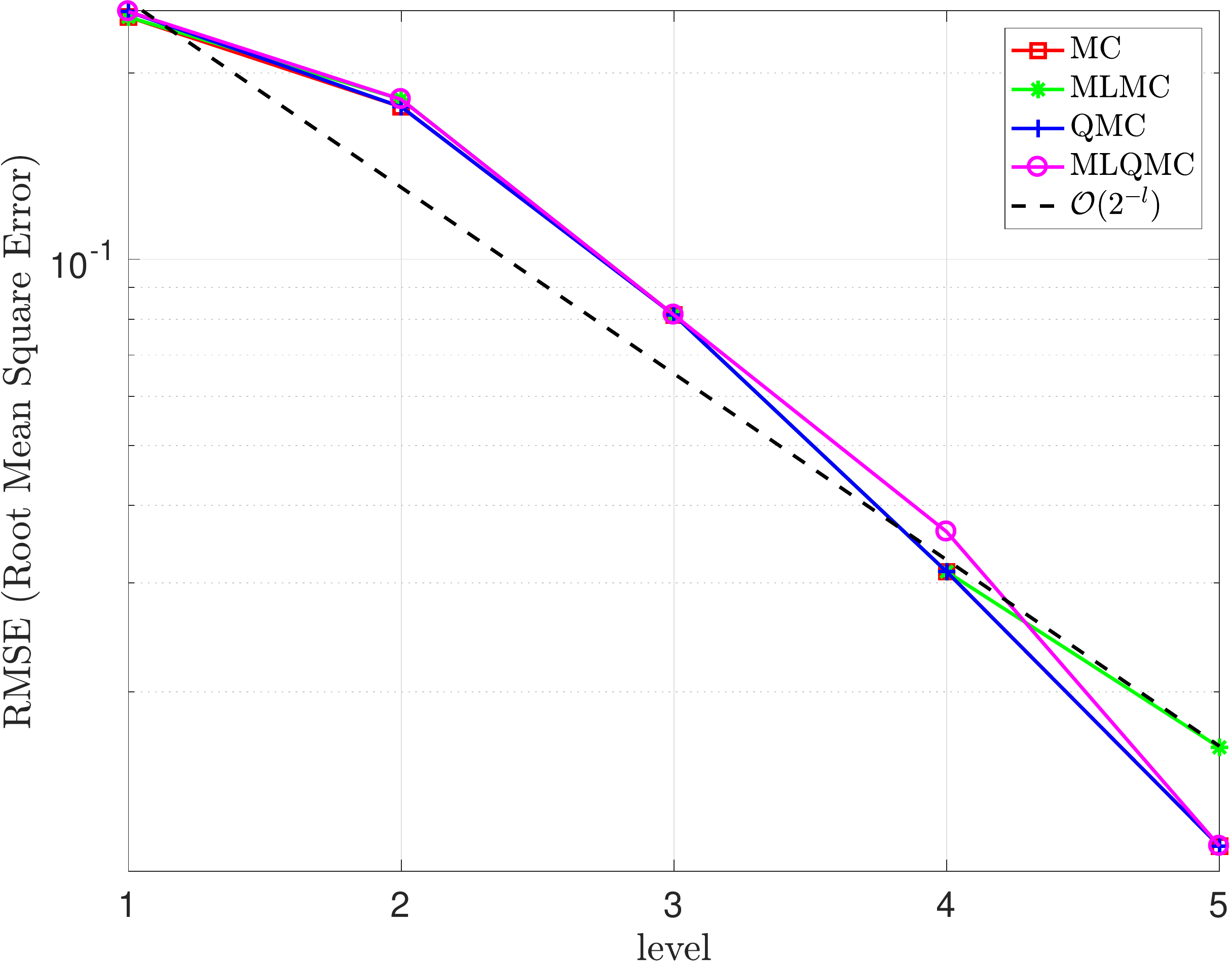}
\includegraphics[width=0.3\textwidth]{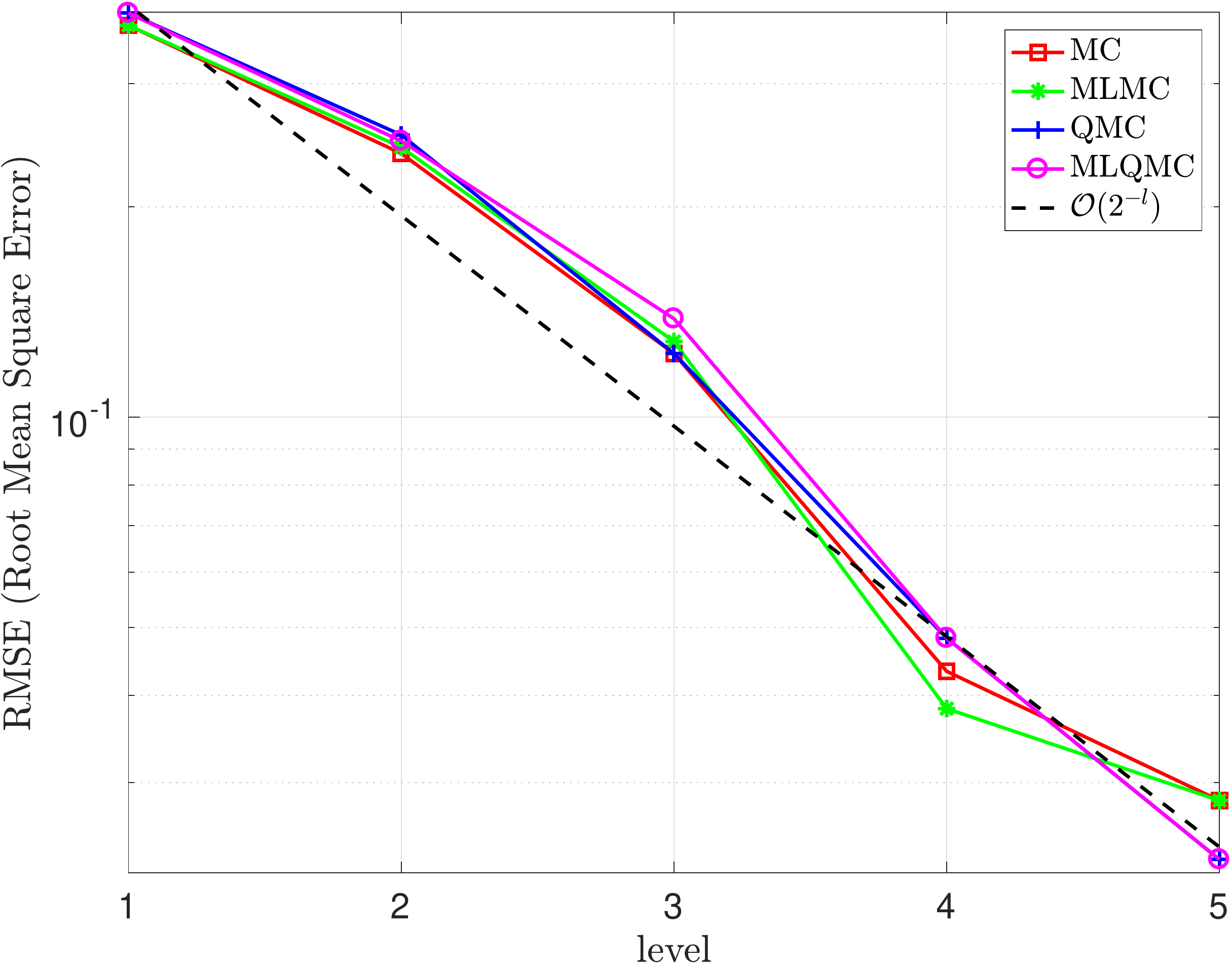}
\includegraphics[width=0.3\textwidth]{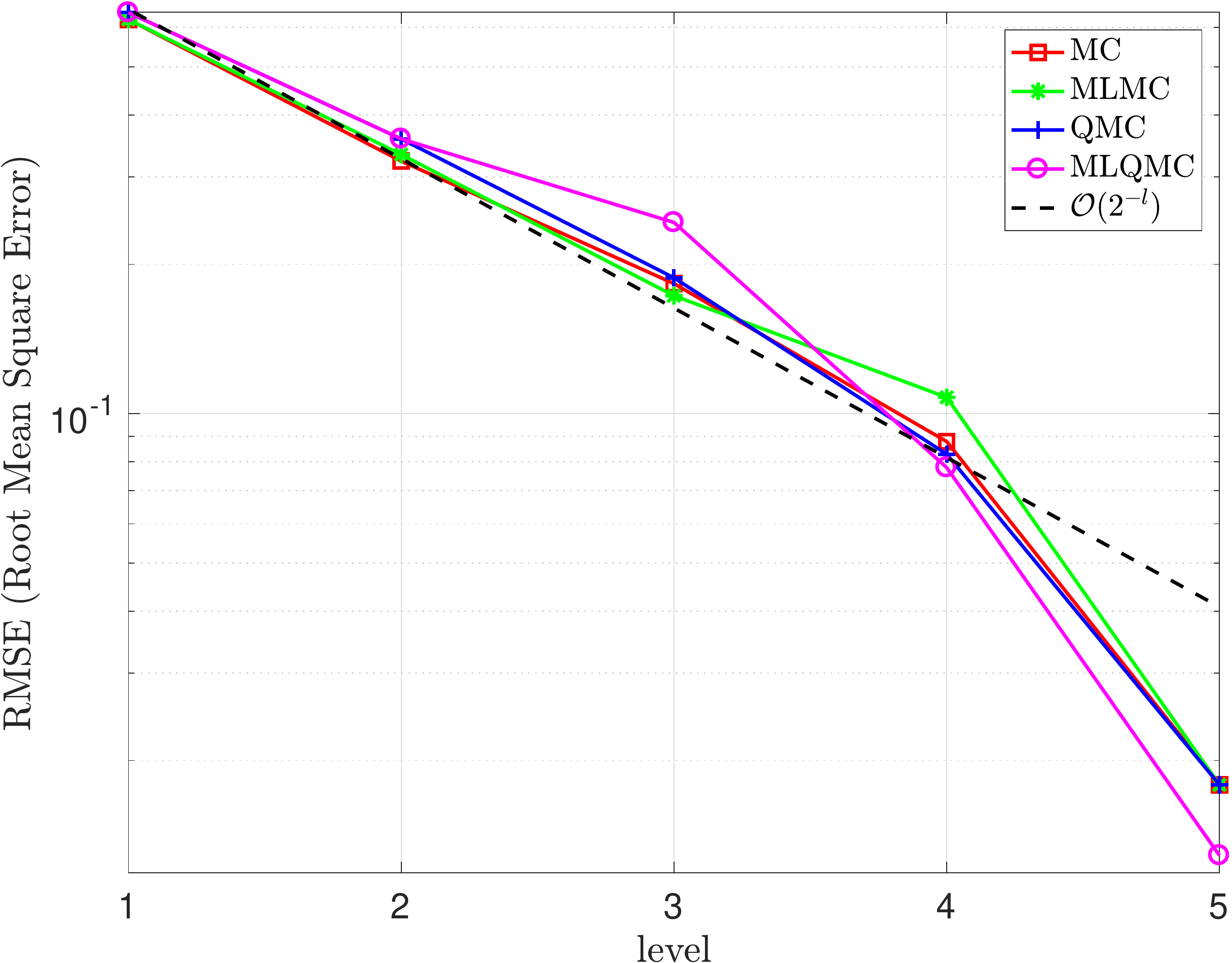}
\includegraphics[width=0.3\textwidth]{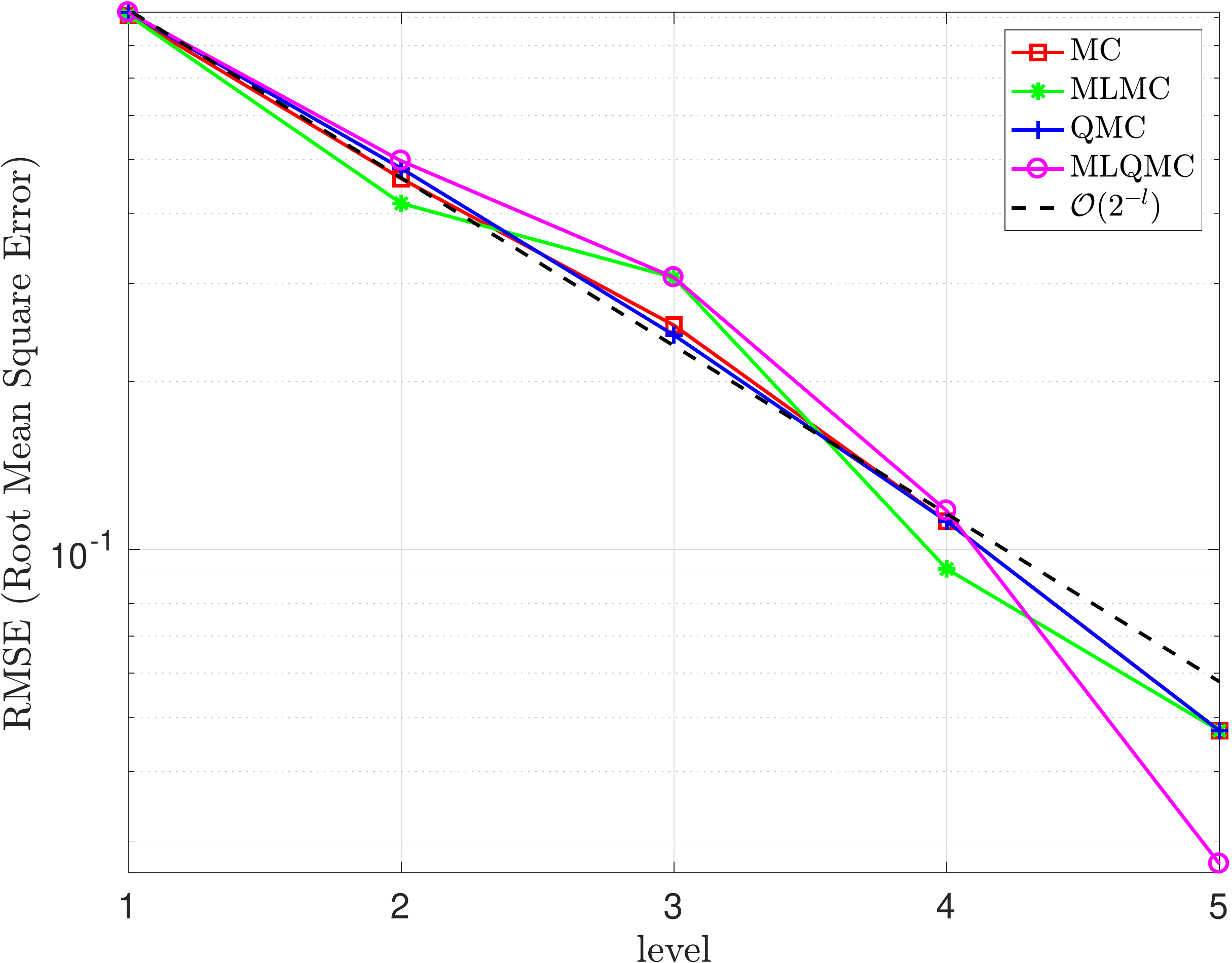}
\captionsetup{justification=centering}
\caption{Convergence of the activation time at the 
locations specified in Figure \ref{fig::activation-transversal} following an order going from the furthest of the (final state) travelling wave to the nearest.}
\label{fig::activation-graph-2}
\end{figure} 


\section{Conclusion}  \label{section:conclusion} 
In this article,
we have considered the monodomain equation from cardiac electrophysiology
with the Fitz-Hugh Nagumo model and an anisotropic conductivity tensor
that can account for increased diffusion along the direction of the heart fibers.
Modelling the heart fibers as a random vector field,
by means of the Karhunen-Lo\'eve expansion
from given expectation and covariance vector fields of the heart fibers,
we arrive at a parametric monodomain equation.
Thus,
common quantities of interest such as the action potential and the activation time
then also are subject to this uncertainty
and we therefore aim at computing their statistics,
which amounts to the evaluation of a high-dimensional integral.

To enable the approximate computation of the high-dimensional integral,
we propose to combine a space-time discretisation of the monodomain equation
using finite elements in space and the Crank--Nicolson method in time,
which yields a method with good parallel scalability,
in a multilevel manner with dimension robust quadrature methods.
The resulting scheme is fully parallelized in space, time and stochastics.

Our numerical experiments show that the approach is feasible and that
the considered quadrature methods consistently satisfy their theoretical
convergence rates.
This indicates that in the settings of our numerical experiments
the more restrictive regularity requirements for the QMC quadrature are fulfilled,
and that the mixed regularity requirement for the multilevel quadrature methods
MLMC and MLQMC are fulfilled, as well.
The results show that we can significantly improve the amount of work
required for a certain error by using the QMC and MLQMC methods
instead of the MC and MLMC methods.
It is important to note,
that as the QMC and MLQMC quadrature methods using Halton points
are essentially a MC or MLMC method where the random sequence of sample points
is replaced with fewer Halton points,
the increased performance does not require any additional non-trivial
implementation modifications to be made when changing from the MC and MLMC methods
to the QMC and MLQMC methods.

Lastly,
the numerical experiments on the realistic heart geometry additionally show
the utility of the multilevel estimator using the quadrature differences
instead of solution differences,
\begin{equation*}
  \operatorname{QoI}_L^{\text{ML}}[u] \isdef\sum_{l=0}^L
  (\mathcal{Q}_{l}-\mathcal{Q}_{l-1})\big(\mathcal{F}_{L-l}[u](\cdot)\big) ,
\end{equation*}
combined with the usage of non-nested meshes when considering
involved space geometries.

\section{Acknowledgements} 

The authors would like to thank the Swiss National Science Foundation (SNSF) for their support through the project “Multilevel Methods and Uncertainty Quantification in Cardiac Electrophysiology” in collaboration with the University of Basel (grant agreement SNSF-$205321\_169599$). The authors also gratefully acknowledge the support of the Center of Computational Medicine in Cardiology and in particular Dr. Simone Pezzuto for providing the heart geometry and many fruitful discussions.
\bibliography{biblio.bbl}
\end{document}